\documentclass[a4paper,10pt]{amsart}
\usepackage[utf8]{inputenc}
\usepackage{amssymb,amsmath,bbm}
\usepackage[all]{xy}

\newtheorem{theorem}{Theorem}[section]
\newtheorem{definition}[theorem]{Definition}
\newtheorem{example}[theorem]{Example}
\newtheorem{proposition}[theorem]{Proposition}
\newtheorem{lemma}[theorem]{Lemma}
\newtheorem{remark}[theorem]{Remark}

\newtheorem{corollary}[theorem]{Corollary}
\newtheorem{exercise}[theorem]{Exercise}

\newcommand{\altA}{\mathcal{A}}
\renewcommand{\AA}{\mathbb{A}}

\newcommand{\CC}{\mathbb{C}}
\newcommand{\altC}{\mathcal{C}}
\newcommand{\EE}{\mathcal{E}}

\newcommand{\GG}{\mathbb{G}}
\newcommand{\II}{\mathcal{I}}
\newcommand{\kk}{\CC}

\newcommand{\LL}{\mathbb{L}}
\newcommand{\altL}{\mathcal{L}}
\newcommand{\NN}{\mathbb{N}}
\newcommand{\PP}{\mathbb{P}}
\newcommand{\PPP}{\mathfrak{P}}
\newcommand{\altP}{\mathcal{P}}
\newcommand{\QQ}{\mathbb{Q}}

\newcommand{\altV}{V}
\newcommand{\WW}{\mathcal{T}r(W)}
\newcommand{\WWW}{\mathfrak{Tr}(W)}
\newcommand{\XXX}{\mathfrak{X}}
\newcommand{\YYY}{\mathfrak{Y}}
\newcommand{\ZZZ}{\mathfrak{Z}}
\newcommand{\ZZ}{\mathbb{Z}}
\newcommand{\Mst}{\mathfrak{M}}
\newcommand{\Msp}{\mathcal{M}}

\newcommand{\unit}{\mathbbm{1}}
\newcommand{\ICS}{\mathcal{IC}}
\newcommand{\DTS}{\mathcal{DT}}

\newcommand{\AAA}{{\mathbb{A}^1}}
\newcommand{\muu}{{\hat{\mu}}}
\newcommand{\fst}{\mathfrak{f}}
\newcommand{\gst}{\mathfrak{g}}

\DeclareMathOperator{\Hom}{Hom}
\DeclareMathOperator{\End}{End}
\DeclareMathOperator{\Ext}{Ext}

\DeclareMathOperator{\Vect}{Vect}
\DeclareMathOperator{\rep}{-Rep}
\DeclareMathOperator{\Ka}{K}
\DeclareMathOperator{\Aut}{Aut}

\DeclareMathOperator{\Perv}{Perv}
\DeclareMathOperator{\MHM}{MHM}

\DeclareMathOperator{\DT}{DT}
\DeclareMathOperator{\Sym}{Sym}

\DeclareMathOperator{\Sch}{Sch}
\DeclareMathOperator{\Spec}{Spec}
\DeclareMathOperator{\Gl}{GL}
\DeclareMathOperator{\id}{id}
\DeclareMathOperator{\Jac}{Jac}

\DeclareMathOperator{\Tr}{Tr}
\DeclareMathOperator{\Crit}{Crit}

\DeclareMathOperator{\Con}{Con}

\DeclareMathOperator{\rk}{rk}

\DeclareMathOperator{\pr}{pr}

\DeclareMathOperator{\PGl}{PGL}
\DeclareMathOperator{\Mat}{Mat}
\DeclareMathOperator{\Mor}{Mor}
\DeclareMathOperator{\MOR}{\underline{Mor}}
\DeclareMathOperator{\Obj}{Obj}
\DeclareMathOperator{\Coh}{Coh}
\DeclareMathOperator{\Gr}{Gr}
\DeclareMathOperator{\MF}{MF}
\DeclareMathOperator{\Bl}{Bl}
\DeclareMathOperator{\Cone}{Cone}
\DeclareMathOperator{\Stab}{Stab}
\DeclareMathOperator{\nilp}{nilp}
\DeclareMathOperator{\coker}{coker}
\DeclareMathOperator{\im}{im}
\DeclareMathOperator{\Fr}{Fr}
\DeclareMathOperator{\Supp}{Supp}

\title{An introduction into (motivic) Donaldson--Thomas theory}
\author{Sven Meinhardt}

\begin{document}

\begin{abstract}
The aim of the paper is to provide a rather gentle introduction into Donaldson--Thomas theory using quivers with potential. The reader should be familiar with some basic knowledge in algebraic or complex geometry. The text contains many examples and exercises to support the process of understanding the main concepts and ideas.     	
\end{abstract}

\maketitle

\tableofcontents

\section{Introduction}

The theory of Donaldson--Thomas invariants  started around 2000 with the seminal work of R.\ Thomas \cite{Thomas1}. He associated  integers to those moduli spaces of sheaves on a compact Calabi--Yau 3-fold which only contain stable sheaves. After some years, K.\ Behrend realized in \cite{Behrend} that these numbers, originally written as ``integrals'' over algebraic cycles or characteristic classes, can also be obtained by an integral over a constructible function, the so-called Behrend function, with respect to the measure given by the Euler characteristic. This new point of view did not only extend the theory to non-compact moduli spaces but revealed  also the ``motivic nature'' of this new invariant. It has also been realized that quivers with potential  provide another class of examples to which Donaldson--Thomas theory applies.  Starting around 2006, D. Joyce \cite{JoyceI},\cite{JoyceCF},\cite{JoyceII},\cite{JoyceIII},\cite{JoyceMF},\cite{JoyceIV} and Y.\ Song \cite{JoyceDT} extended the theory using all kinds of ``motivic functions'' to produce (possibly rational) numbers even in the presence of semistable objects which is the generic situation when classifying objects in abelian categories. Around the same time, M.\ Kontsevich and Y.\ Soibelman \cite{KS1},\cite{KS2},\cite{KS3} independently proposed a theory producing even motives, some sort of refined ``numbers'', instead of simple numbers, also in the presence of semistable objects. The technical difficulties occurring in their approach disappear in the special situation of representations of quivers with potential. The case of zero potential has been  intensively studied by M.\ Reineke in a series of papers \cite{Reineke2},\cite{Reineke3},\cite{Reineke4}.
Despite some computations of motivic or even numerical Donaldson--Thomas invariants for quivers with  potential (see \cite{BBS},\cite{DaMe1},\cite{DaMe2},\cite{MMNS}), the true nature of Donaldson--Thomas invariants for quiver with potential remained mysterious for quite some time. A full understanding has been obtained recently and is the content of a series of papers  \cite{DavisonMeinhardt4},\cite{DavisonMeinhardt3},\cite{MeinhardtReineke},\cite{Meinhardt4}. \\
The present text aims at giving a gentle introduction to Donaldson--Thomas theory in the case of quiver with potential. We have two reasons for our restriction to quivers. Firstly, so-called orientation data will not play any role, and secondly, we do not need to touch derived algebraic geometry. Apart from this, many important ideas and concepts are already visible in the case of quiver representations, and since the theory is fully understood, we belief that this is a good starting point for your journey towards an understanding of Donaldson--Thomas theory. There are more survey articles available focusing on different aspects of the theory. (see  \cite{JoyceDT},\cite{KS3},\cite{Szendroi}) \\
Let us give a short outline of the paper. The next section starts very elementary by discussing the problem of classifying objects. The objects which are of interest to us form an abelian category although many ideas of section 2  also apply to ``non-linear'' moduli problems. We study in detail the difficulties arising from the construction of moduli spaces and develop slowly the concept of a (moduli) stack. Although the theory of stacks is rather rich and complicated, we can restrict ourselves to quotient stacks throughout this paper. Hence, a good understanding of  a quotient stacks is inevitable. We try to illustrate this concept by giving important examples. We should mention that only very little knowledge of algebraic or complex geometry is needed. In many cases, you can easily replace ``schemes'' with ``varieties'' or ``complex manifolds''.\\
Section 3 provides the background on quivers and their representations. The point of view taken here is that quivers are the categorical (noncommutative) analogue of polynomial algebras in ordinary commutative algebra. In other words, they are a useful tool for practical computations when dealing with linear categories, but at the end of the day the result should only depend on the linear category and not its presentation as a quotient of the path category of a quiver by some ideal of relations. The relations important in this paper are given by noncommutative partial derivatives of a so-called potential. \\
The next two sections provide the language and the framework to formulate Donaldson--Thomas theory in section 6. We start in section 4 with the concept of ``motivic theories''. The best example the reader should have in mind are constructible functions. It should be clear that constructible functions can be pulled back and multiplied. Using fiberwise integrals with respect to the Euler characteristic, we can even push forward constructible functions. Moreover, every locally closed subscheme/subvariety/submanifold determines a constructible function, namely its characteristic function. In a nutshell, a motivic theory is just a generalization of this associating to every scheme $X$ an abelian group $R(X)$ of ``functions'' on $X$ which can be pulled back, pushed forward and multiplied. Moreover, to every locally closed subscheme in $X$ there is a ``characteristic function'' in $R(X)$ such that the characteristic function of a disjoint union is the sum of the characteristic function of its summands. Its is this property what makes a theory of generalized functions ``motivic''. As usual in algebraic geometry, the term ``function'' should be used with some care. Every function on say a complex variety $X$ determines a usual function from the set of points in $X$ to the coefficient ring $R(\mbox{point})$ of our theory, but this is not a one-one correspondence. \\
In section 5 we introduce vanishing cycles. We do not assume that the reader is familiar with any theory of vanishing cycles. As in the previous section, a vanishing cycle is just  an additional  structure on motivic theories formalizing the properties of ordinary classical vanishing cycles. The Behrend function mentioned at the beginning of this introduction provides a good example of a vanishing cycle on the theory of constructible functions. In fact, we will construct in a functorial way two vanishing cycles associated to a given motivic theory. The first construction is rather stupid, but the second one essentially covers all  known nontrivial examples. At the end of sections 4 and 5 we extend motivic theories and vanishing cycles to quotient stacks as quotient stacks arise naturally in moduli problems. There is a way to circumvent stacks in Donaldson--Thomas theory by considering framed objects, but we belief that the usual approach of using stacks is more conceptual and should be known by anyone who wants to understand Donaldson--Thomas theory seriously.\\
In the last section 6 we finally introduce Donaldson--Thomas functions and invariants. After stating the main results, we consider many examples to illustrate the theory. Finally, we develop some tools used in Donaldson--Thomas theory such as Ringel--Hall algebras, an important integration map and the celebrated wall-crossing formula. \\

The reader will realize shortly that the text contains tons over exercises and examples. Most of the exercises are rather elementary and require some elementary computations and standard arguments. Nevertheless, we suggest to the reader to do them carefully in order to get your hands on the subject and to obtain a feeling about the objects involved. There is a lot of material in this text which is not part of the standard graduate courses at universities, and if you are not already familiar with the subject you certainly need some practice as we cannot provide a deep and lengthy discussion of the material presented here.\\

\textbf{Acknowledgments.} The paper is an expanded version of a couple of lectures the author has given in collaboration with Ben Davison at KIAS in February 2015. He is more than grateful to  Michel van Garrel and Bumsig Kim for giving him the opportunity to visit this wonderful place. A lot of work on this paper has also been done at the University of Hong Kong, where the author gave another lecture series on Donaldson--Thomas theory. The author wants to use the opportunity to thank Matt Young for being such a wonderful host. He also wants to thank Jan Manschot for keeping up the pressure to finish this paper and for offering the opportunity to publish the paper. Finally, the author is very grateful to Markus Reineke for giving him as much support as possible.

\section{The problem of constructing a moduli space}

\subsection{Moduli spaces}

Let us start by recalling the general idea of a moduli space. Depending on the situation, mathematicians are trying to classify objects of various types. The general pattern is the following. There is some set (or class)   of objects and isomorphisms between two objects. Such a structure is called a groupoid. A groupoid is a category with every morphism being an isomorphism. If the set of objects has cardinality one, a groupoid is just a group. The other extreme is a groupoid such that every morphism is the identity morphism of some object. Such groupoids are in one-to-one correspondence with ordinary sets. Hence, a groupoid interpolates between sets and groups. There are two main sources of groupoids.
\begin{example}  \rm Let $X$ be a topological space. The fundamental groupoid $\pi_1(X)$ is the groupoid having the points of $X$ as objects, and given two points $x,y\in X$, the set of morphisms from $x$ to $y$ is the set of homotopy classes of paths from $x$ to $y$. Fixing a base point $x\in X$, the usual fundamental group $\pi_1(X,x)$ is just the automorphism group of $x$ considered as an object in the groupoid $\pi_1(X)$. Denote by $\pi_0(X)$ the set of path connected components, i.e.\ the set of objects in $\pi_1(X)$ up to isomorphism. 
\end{example}
\begin{example}  \rm
Given a category $\altC$, one can consider the subcategory $\II so(\altC)$ of all isomorphisms in $\altC$. Thus, $\II so (\altC)$ is a groupoid, and $\altC/_\sim$ denotes the set of objects in $\altC$ up to isomorphism.     
\end{example}
These two examples are related as follows. To every (small) category one can construct a topological space $X_\altC$ - the classifying space of $\altC$ - such that $\pi_1(X_\altC)\cong \II so (\altC)$ and $\pi_0(X_\altC )=\altC/_\sim$. \\
Let us come back to the classification problem, say of objects in $\altC$ up to isomorphism. The problem is to describe the set $\altC/_\sim$. If it is discrete in a reasonable sense, one tries to find a parameterization by less complicated (discrete) objects. This applies for instance to the classification of semisimple algebraic groups or finite dimensional representations of the latter. In many other situations, $\altC/_\sim$ is uncountable, and one wants to put a geometric structure on the set $\altC/_\sim$ to obtain  a ``moduli space''. However, if for instance $\altC/_\sim$ has the cardinality of the field of complex numbers, one can always choose a bijection $\altC/_\sim \cong M$ to the set of points of any complex variety or manifold $M$ of dimension greater than zero. Pulling back the geometric structure of $M$ along this bijection, we can equip $\altC/_\sim$ in many different (non-isomorphic) ways with a structure of a complex manifold. Hence, we should ask:\\

\textbf{Question:} Is there a natural geometric structure on $\altC/_\sim$? What does ``natural'' actually mean?\\

There is a very beautiful idea of what ``natural'' should mean, and which applies to many situations. Assume there is a notion of a family of objects in $\altC$ over some ``base'' scheme/variety/(complex) manifold $S$, i.e.\ some object on $S$ which has ``fibers'' over $s\in S$, and these fibers should be objects in $\altC$. 
\begin{example}  \rm
 Given a $\CC$-algebra $A$, a family of finite dimensional $A$-representations is a (holomorphic) vector bundle $V$ on $S$ and an $\CC$-algebra homomorphisms $A\to \Gamma(S,\mathcal{E}nd(V))$ from $A$ into the algebra of sections of the endomorphism bundle of $V$. 
\end{example}
\begin{example}  \rm
 Given a scheme/variety/manifold $X$ over $\CC$ and some parameter space $S$, a family of coherent sheaves on $X$ parametrized by $S$ is just a coherent sheaf $E$ on $S\times X$ which is flat over $S$. The latter condition ensures that taking fibers and pull-backs of families behaves well. If $E$ is a family of zero dimensional sheaves on $X$, i.e.\ if the projection $p:\Supp(E)\to S$ has zero-dimensional fibers, flatness of $E$ over $S$ is equivalent to the requirement that $E$ is locally free over $S$. Using the coherence of $E$ once more, one can show that $p:\Supp(E)\to S$ is a finite morphism and if $X=\Spec A$ is affine, $E$ is completely determined by the vector bundle $V:=p_\ast E$ on $S$ together with a $\CC$-algebra homomorphism $A\to \Gamma(S,\EE nd(V))$. From that perspective, the previous example can be seen as a non-commutative version, namely families of zero dimensional sheaves on the non-commutative affine scheme $\Spec A$ for $A$ being a $\CC$-algebra. 
\end{example}
\begin{example}  \rm A $G$-homogeneous space with respect to some (algebraic) group $G$ is a scheme $P$ with a right $G$-action such that $P\cong G$ as varieties with right $G$-action, where $G$ acts on $G$ by right multiplication. A (locally trivial) family of $G$-homogeneous spaces over $S$ is defined as a principal $G$-bundle on $S$.  
\end{example}
Once a family is given, by taking the ``fiber'' over $s\in S$ we get an object in $\altC$ and, hence, a point in $\altC/_\sim$. Varying $s\in S$, we end up with a map $u:S\to \altC/_\sim$. Moreover, we see that the pull-back of a family on $S$ along a morphism $f:S'\to S$ induces a morphism $u':S'\to \altC/_\sim$ such that $u'=u\circ f$. Coming back to the question formulated above, we can now be more precise by asking:  \\

\textbf{Question:} Is there a structure of a variety or scheme on $\altC/_\sim$ such that for every family of objects over any $S$, the induced map $S\to \altC/_\sim$ is a morphism of schemes? If so, is there any way to get back the family by knowing only the morphism $S\to \altC/_\sim$?\\

If the first question has a positive answer, we call $\Msp=(\altC/_\sim,\mbox{scheme structure}) $ a coarse moduli space for $\altC$. If the second part of the question is also true, we should be able to (re)construct a ``universal'' family on $\Msp$ by considering the map $\id:\Msp\to \Msp$. Moreover, given a map $u':S'\to \Msp$ such that $u'=u\circ f$ for some map $f:S'\to S$, the family on $S'$ should be the pull-back of the family on $S$ associated to $u$ by uniqueness.   As every morphism $u:S\to \Msp$ has an obvious factorization $S\xrightarrow{u} \Msp\xrightarrow{\id} \Msp$, we finally see that every family on $S$ must be the pull-back of the ``universal'' family on $\Msp$. In such case, we call $\Msp$ a fine moduli space. 

\begin{example}  \rm
Let $\altC=\Vect_\CC$ be the category of finite dimensional $\CC$-vector spaces. A (locally trivial) family of finite dimensional vector spaces is just a vector bundle on some parameter space $S$. As a vector space is classified by its dimension, we can put the simplest scheme structure on $\Vect_\CC/_\sim=\NN$ by thinking of $\NN$ as a disjoint union of countably many copies of $\Spec\CC$. Given a vector bundle $V$, we obtain a well-defined morphism $S\to \NN$ mapping $s\in S$ to the copy of $\Spec\CC$ indexed by the dimension of the fiber $V_s$ of $V$ at $s$. The scheme $\NN$ is a course moduli space, but apart from the zero dimensional case, it can never be a fine moduli space. Indeed, there is an obvious and essentially unique vector bundle on $\NN$ inducing the identity map $\NN\to \NN$, but a vector bundle on $S$ can never be the pull-back of the one on $\NN$ unless it is constant. Thus, $\NN$ is not a fine moduli space.  
\end{example}
These are also bad news for  our previous examples concerning representations of an algebra $A$ or sheaves on a variety $X$. Indeed, for $A=\CC$ or $X=\Spec \CC$, we are back in the classification problem of finite dimensional $\CC$-vector spaces.  
\begin{example}  \rm
Similar to the previous example, we see that the classification problem for homogeneous $G$-spaces has only a coarse moduli space given by $\Spec\CC$. 
\end{example}
There are several strategies to overcome the difficulty of constructing a fine moduli space.
\begin{example}[rigid families] \rm
One possibility is to rigidify families of objects. For example, instead of considering all vector bundles we could also restrict ourselves to constant vector bundles. In this particular case, $\NN$ is even a fine moduli space. However, in many situations one wants to glue families together to form families of objects on bigger spaces. This is incompatible with the concept of rigidity, and we will not follow this path. 
\end{example}

\begin{example}[weaker equivalence] \rm \label{weaker_equivalence}
Instead of classifying objects up to isomorphism, we could allow weaker equivalences. For example, we could identify to families $V^{(1)}$ and $V^{(2)}$ (over $S$) of vector spaces or representations of an algebra  $A$ if there is a line bundle $L$ on $S$ such that $V^{(2)}=V^{(1)}\otimes_{\mathcal{O}_S} L$. By doing this, we can always replace a rank one bundle with the trivial rank one bundle. Hence, the moduli space $\Spec \CC$ of one-dimensional vector spaces is a fine moduli space.  
\end{example}
\begin{example}[projectivization] \rm \label{projectivization}
 Similar to families of vector spaces of dimension $r$, one could look at locally trivial families $\altP$ of projective spaces $\PP^{r-1}$. The transition functions between local trivializations are regular functions with values in $\Aut(\PP^{r-1})=\PGl(r)$. Every vector bundle $V$ of rank $r$ provides such a bundle by taking $\altP:=\PP(V)$, the bundle of lines or hyperplanes in $V$. Two vector bundles $V^{(1)}$, $V^{(2)}$ define isomorphic bundles $\PP(V^{(1)})\cong \PP(V^{(2)})$ if and only if $V^{(2)}=V^{(1)}\otimes_{\mathcal{O}_S} L$ for some line bundle $L$ on $S$, providing the bridge to the previous example. However, not every $\PP^{r-1}$-bundle $\altP$ can be realized  as $\PP(V)$ for some vector bundle $V$ on $S$. Given a $\PP^{r-1}$-bundle $\altP$, there is an associated locally trivial bundle $\EE_{\altP}$  of $\CC$-algebras isomorphic to $\End_\CC(\CC^r)\cong\Mat_\CC(r,r)$. Conversely, every locally trivial bundle $\EE$ of $\CC$-algebras isomorphic to $\Mat_\CC(r,r)$ defines an associated $\PP^{r-1}$-bundle $\altP_{\EE}$ as the transition functions of $\EE$ must be in $\Aut(\Mat_\CC(r,r))=\PGl(r)$. Thus, we have an equivalence of categories between locally trivial $\PP^{r-1}$-bundles and locally trivial $\Mat_\CC(r,r)$-bundles. If the $\PP^{r-1}$-bundle $\altP$ is given by $\PP(V)$ for a vector bundle $V$ of rank $r$, then $\EE_{\PP(V)}=\EE nd(V)$. Given a $\CC$-algebra $A$, we can study families given by a locally free $\PP^{r-1}$-bundle $\altP$ or equivalently a locally free $\Mat_\CC(r,r)$-bundle $\EE$ and a homomorphism of algebras $A\to \Gamma(S,\EE)$. If $A=\CC$, there is only  a fine moduli space for $r=1$ as every $\PP^0$-bundle must be constant. If the algebra $A$ is more complicated, there are also fine moduli spaces for $r>1$, but only for objects which are simple in a suitable sense. For $A=\CC$ there are no simple vector spaces of dimension $r>1$.     
\end{example}
As we have seen, the construction of fine moduli spaces can only be done in a few cases and severe restrictions. But even if we were only interested in coarse moduli spaces, a standard problem will occur as the following example shows.
\begin{example}  \rm \label{S_equivalence} Instead of looking at representations of  $A=\CC$, we enter the next level of complexity by looking at finite dimensional representations of $A=\CC[z]$. A one-dimensional representation $V$ is determined by the value of $z$ in $\End_\CC(V)\cong \CC$. In other words, a coarse moduli space is given by the complex affine line $\AAA$. Still, we have to face the problem discussed before that a line bundle on $S$ with $z$ acting by multiplication with a fixed number $c\in \CC$ could almost never be the pull-back of a universal family under the constant map $S\to \AAA$ mapping $s\in S$ to $c\in \AAA$. Let us ignore the problem of finding a fine moduli space and continue with two-dimensional representations. Consider the trivial rank 2 bundle on $S=\AAA$ with $z$ acting via the nilpotent matrix
 \[  { 1 \;\; s \choose 0 \;\; 1 } \]
in the fiber over $s\in S=\AAA$. The representations for $s\not=0$ are all isomorphic to each other, and our ``classifying map'' $u:S\to \Msp_2$ to a coarse moduli space $\Msp_2$ of rank 2 representations must be constant on $S\!\setminus\!\!\{0\}$. For $s=0$ we obtain a different representation and $u(0)$ must be another point in $\Msp_2$ if the latter parametrizes isomorphism types. However,  such a discontinuous map $u:S\to \Msp_2$ cannot exist, and we have to abandon the idea of finding a coarse moduli space parameterizing isomorphism classes. One can show that a ``reasonable'' coarse moduli space is given by the GIT-quotient $\Mat_\CC(2,2)/\!\!/\Gl(2)$ which is realized as $\Spec \CC[\Mat_{\CC}(2,2)]^{\Gl(2)}\cong \AA^2$ and similar for higher ranks. The classifying map $S\to \AA^2$ will map $s\in S$ to the unordered pair of eigenvalues of the $z$-action in the fiber over $s$. Such an unordered pair of eigenvalues is determined by the sum (corresponding to the trace) and the product (corresponding to the determinant) of the eigenvalues and similar for higher ranks. Therefore, $\Msp$ will parametrize unordered direct sums of one-dimensional representations. In other words, by passing from $\altC/_\sim$ to $\Msp$, we identify each representation with the (unordered) direct sum of its simple Jordan--H\"older factors. Representations having the same Jordan--H\"older factors, i.e.\ corresponding  to the same point in $\Msp$, are often called S-equivalent\footnote{The ``S'' in ``S-equivalent'' refers to semisimple, i.e.\ sums of simples, and should not be confused with our notation of a base of a family.}.   
\end{example}
Let us summarize the lessons we have learned in the previous examples:
\begin{enumerate}
 \item Constructing coarse moduli spaces has only a chance if we do not parametrize objects up to isomorphism but up to the weaker S-equivalence. In other words, classifying objects up to isomorphism is only possible for simple objects, i.e. objects without subobjects.
 \item The construction of a universal family on the moduli space of simple objects might only work if we identify two families under a weaker equivalence (twist with a line bundle) or pass to some projectivization. 
\end{enumerate}
We suggest to the reader to check these statements in the previous examples.

\subsection{Stability conditions}

Even though the set of objects in  $\altC$ up to isomorphism might be very large, the set of (isomorphism classes of) simple objects can be quite small, even finite. Thus, the ``coarse'' moduli space would not deliver much insight into the set of isomorphism types  in $\altC$. However, there is a simple but clever idea to overcome this problem. Instead of looking at $\altC$, we should ``scan'' $\altC$ by means of a collection $(\altC_\mu)_{\mu\in T}$ of ``small'' full subcategories $\altC_\mu\subseteq \altC$. An object which might be far away from being simple or semisimple (direct sum of simples) can become semisimple or even simple in $\altC_\mu$. By doing this, we can distinguish many S-equivalent objects either because they live in different subcategories or they live in the same subcategory $\altC_\mu$ but have different Jordan--H\"older filtrations taken in $\altC_\mu$. This brilliant idea is the essence of the concept of stability conditions. The following definition is due to Tom Bridgeland. However, there are more general definitions of stability conditions.  
\begin{definition} \hfill
\begin{enumerate}
\item A central charge on a noetherian abelian category $\altC$ is a function $Z$ on the set of objects in $\altC$ with values in $\mathbb{H}_+:=\{r\exp(\sqrt{-1}\phi)\in \CC\mid r\ge 0, \phi\in (0,\pi]\}$ such that $Z(E)=0$ implies $E=0$ and $Z(E)=Z(E')+Z(E'')$ for every short exact sequence $0\to E'\to E\to E''\to 0$. \\
\item Given a central charge $Z$, we call an object $E\in \altC$ semistable if 
\[ \arg Z(E'))\le \arg Z(E)  \;\;  \mbox{ for all subobjects } E'\subset E.                       \]
\item For $\mu\in (-\infty,+\infty]$ we denote with $\altC_\mu$ the full subcategory of all semistable objects $E$ of slope  $-\cot(\arg Z(E))=\mu$ and the zero object. It turns out that $\altC_\mu$ is an abelian subcategory of $\altC$ (cf.\ Exercise \ref{semistable_reps}).
\item A simple object in $\altC_\mu$ is called stable. We assume that  every semistable object of slope $\mu$ has a Jordan--H\"older filtration with stable subquotients of the same slope. Semisimple objects of $\altC_\mu$, i.e.\ sums of stable objects of slope $\mu$, are called polystable.  
\item Every object $E$ in $\altC$ has a unique filtration $0 \subset E_1 \subset E_2 \subset \ldots \subset E_n=E$, the Harder--Narasimhan filtration, with semistable quotients $E_i/E_{i-1}$ of strictly decreasing slopes. 
\end{enumerate}
\end{definition}
\begin{example}[The $r$-Kronecker quiver] \label{Kronecker}  \rm
Let us illustrate this idea with a simple example. Consider the abelian category of $r$-tuples $\bar{x}$ of linear maps $x_i:V_1\to V_2 $ for $1\le i\le r$ between finite dimensional vector spaces $V_1,V_2$. Choosing two complex numbers $\zeta_1,\zeta_2 \in \mathbb{H}_+$, we get a central charge by putting $Z(\bar{x})=\zeta_1\dim V_1 + \zeta_2 \dim V_2$. Assume first that $\arg(\zeta_1)=\arg(\zeta_2)$. Then, all objects are semistable of the same slope $\mu=-\cot( \arg \zeta_1)$, and we have to face the old problems. Choose for instance $\dim V_1=\dim V_2=1$. The isomorphism type of  such objects is determined by the choice of $r$ complex numbers $x_1,\ldots, x_r$ up to rescaling by $(g_1,g_2)\in \Gl(V_1)\times \Gl(V_2)=\CC^\ast\times \CC^\ast$ via $g_1x_ig_2^{-1}$. As the diagonal group $\{(g,g)\mid g\in\CC^\ast\}$ acts trivially, we have the take the GIT quotient of $\AA^r$ by $\CC^\ast$ which is just a point as $\Spec \CC[x_1,\ldots,x_r]^{\CC^\ast}=\Spec \CC$. This corresponds to the fact that all objects have the same Jordan--H\"older factors $x_i=0:V_1 \to 0$ and $x_i=0:0 \to V_2$. Thus, all objects are S-equivalent to ``$V_1\oplus V_2$''$ = V_1\xrightarrow{0} V_2$. If $\arg \zeta_2 > \arg \zeta_1$, non of our objects with $\dim V_1=\dim V_2=1$ are semistable as the central charge $\zeta_2$ of the  subobject $0:0 \to V_2$ has a bigger argument than the central charge $\zeta_1+\zeta_2$ of our given object. If, however, $\arg \zeta_2 < \arg \zeta_1$, all objects except for the semisimple $V_1\oplus V_2$ corresponding to $x_i=0$ for all $1\le i\le r$ are semistable of slope $\mu=- \Re e(\zeta_1+\zeta_2)/\Im m(\zeta_1+\zeta_2)$, and even stable. The moduli space $\Msp^{\zeta_1,\zeta_2}_{(1,1)}=\AA^r\!\setminus\!\!\{0\}/\CC^\ast= \PP^{r-1}$ parameterizing isomorphism classes of simple objects in $\altC_\mu$ of dimension vector $(\dim V_1,\dim V_2)=(1,1)$ is even a fine moduli space if we identify two families of $r$-tuples of line bundle morphisms $x_i:V_1\to V_2$ on a parameter space $S$ as soon as they become isomorphic after twisting  $V_1$ and $V_2$ with some line bundle $L$.  
\end{example}
Note that coarse moduli spaces parameterizing S-equivalence classes of objects in $\altC_\mu$ might not exist for all central charges, but one can show the existence for generic central charges and reasonable abelian categories. \\
We should also keep in mind that we paid a price for getting a refined version of S-equivalence, namely S-equivalence in  subcategories. Indeed, coarse moduli spaces of (S-equivalence classes of) semistable objects can only ``see'' semistable objects but no objects with a non-trivial Harder--Narasimhan filtration. Hence, the construction of (coarse) moduli spaces remains unsatisfying.

\subsection{Moduli stacks}

There is, however, an alternative way to overcome all the problems seen in the previous examples. Following this approach, one can construct a fine moduli ``space'' with a universal family parameterizing all objects - not only simple or stable ones - up to isomorphism. According to the conservative law of mathematical difficulties, we also have to pay a price for getting such a beautiful solution of our moduli problem. It is hidden in the word ``space''. In fact, we have to leave our comfort zone of varieties or schemes and have to dive into the universe of more general spaces known as ``Artin stacks''.\\

Recall that a scheme $X$ is uniquely characterized by its set-valued functor $h_X:S\mapsto \Mor(S,X)$ of points. We have seen many set-valued functors before while studying moduli problems. The general pattern was the following. We considered set-valued contravariant functors  $F:S\longmapsto \{ \mbox{families of objects in }\altC \}/_\sim$ and a fine moduli space would be a scheme $\Msp$ such that $F\cong h_\Msp$, while a coarse moduli space is a scheme $\Msp$ together with a map $F\to h_\Msp$ which is universal with respect to all maps $F\to h_X$ of functors. One possibility of generalizing the concept of a space is to consider set-valued functors satisfying similar properties like the functor $h_X$. Note that if one has a collection of morphisms $U_i\to X$ defined on open subsets $U_i$ covering $S$ such that the maps agree on overlaps, one can glue the maps to form a global morphism $S\to X$. This sheaf property should also be satisfied by a general set-valued functor to be a reasonable generalization of a scheme. Such set-valued functors are also often called ``spaces''. A generalized space is called algebraic if it can be written as the ``quotient'' $X/_\sim$ of a scheme $X$ by an (\'{e}tale) equivalence relation. In other words, algebraic spaces are not to far away from schemes and many results for schemes can be generalized to algebraic spaces. In our situation of forming moduli spaces, this is still not the right approach to take, but shows already into the right direction. Indeed, the problems arising in the construction of universal families are related to the presence of (non-trivial) automorphisms. Thus, we should take automorphisms and isomorphisms more seriously into account. \\
Recall that a set with isomorphisms between points was just a groupoid studied at the beginning of this section. Hence, instead of looking at set-valued functors on the category of schemes, we should consider groupoid-valued contravariant functors. These functors should satisfy some gluing property which looks a bit more complicated than in the set-theoretic context. The best idea of remembering the gluing property is by looking at an example which is - as before - the baby example for all Artin stacks.
\begin{example}  \rm \label{vector_bundle}
Consider the groupoid-valued functor $\Vect$ which maps any scheme $S$ to the groupoid of vector bundles (the objects) and isomorphisms between them (the morphisms). By pulling back vector bundles along morphisms $f:S'\to S$, we get indeed a contravariant functor.\footnote{Strictly speaking, we only get a pseudofunctor as $g^\ast\circ f^\ast$ is only equivalent to $(f\circ g)^\ast$, but we will ignore this technical problem as one can always resolve it.} Given two vector bundles $V, V'$ and an open cover $\cup_{i\in I}U_i=S$ of $S$ together with isomorphisms\footnote{We will always denote the pull-back along an inclusion $U\hookrightarrow S$ of an open subset by $|_U$.} $\phi_i:V|_{U_i} \to V'|_{U_i}$ on the open subsets $U_i$ such that they agree after restriction to the overlaps, i.e.\ $\phi_i|_{U_{ij}}=\phi_j|_{U_{ij}}$ with $U_{ij}=U_i\cap U_j$, we can 
always find a unique global isomorphism $\phi:V\to V'$ such that $\phi_i=\phi|_{U_i}$. On the other hand, if we have vector bundles $V_i$ on $U_i$ and isomorphisms 
$\phi_{ij}:V_i|_{U_{ij}} \to V_j|_{U_{ij}}$ such that the only possible composition $V_i|_{U_{ijk}} \to V_j|_{U_{ijk}} \to V_k|_{U_{ijk}} \to V_i|_{U_{ijk}}$ of their 
restrictions to the triple overlaps $U_{ijk}=U_i\cap U_j\cap U_k$ is the identity (cocycle condition), one can use the transition isomorphisms $\phi_{ij}$ to glue the $V_i$ together, i.e.\ there is a vector bundle $V$ on $S$ and a family of isomorphisms $\phi_i:V|_{U_i}\to V_i$ such that the only possible composition $V|_{U_{ij}}\to V_i|_{U_{ij}} \to V_j|_{U_{ij}} \to V|_{U_{ij}}$ of their restrictions with $\phi_{ij}$ is the identity. This was the gluing property for isomorphisms and objects, and if we replace the word ``vector bundle'' with ``object'', we get the general form of the gluing property for a groupoid-valued functor.       
\end{example}
\begin{definition}
 A stack is a groupoid-valued contravariant functor\footnote{Again, we  ignore the fact that $g^\ast\circ f^\ast$ might only be equivalent to $(f\circ g)^\ast$ for a pair $S''\xrightarrow{g} S' \xrightarrow{f} S$ of composable morphisms.} on the category of schemes satisfying the gluing property for isomorphisms and objects as seen in Example \ref{vector_bundle}
\end{definition}
In that perspective, a stack is like a (generalized) space with set-valued functors replaced with groupoid-valued functors.
\begin{exercise} Thinking of a set as a special groupoid with no nontrivial isomorphisms, show that every generalized space is a stack.
\end{exercise}
\begin{exercise} \label{moduli_algebra}
 Fix a $\CC$-algebra $A$. Show that the functor $A\rep$ associating to every scheme $S$ the groupoid of vector bundles $V$ with algebra homomorphisms $A\to \Gamma(S,\EE nd(V))$ (the objects) and isomorphisms of vector bundles compatible with the algebra homomorphisms (the morphisms) is a stack. Prove the same for bundles $\EE$ of matrix algebras and algebra homomorphisms $A\to \Gamma(S,\EE)$ as in Example \ref{projectivization}. 
\end{exercise}
\begin{exercise}
 Fix a scheme/variety/manifold $X$ over $\CC$. Show that the functor $\Coh^X$ associating to every scheme $S$ the groupoid of coherent sheaves $E$ on $S\times X$ flat over $S$ (the objects) and isomorphisms between them (the morphisms) is  a stack.  
\end{exercise}
\begin{exercise}
 Fix an algebraic group $G$. Show that the functor $\Spec \CC/G$ associating to every scheme the groupoid of principal $G$-bundles (the objects) and isomorphisms between them (the morphisms) is  a stack.   
\end{exercise}
\begin{example}  \rm
 The following example is a generalization of the previous exercise. Fix an algebraic group $G$ and a scheme $X$ with a (right) $G$-action. There is a stack $X/G$ associating to every scheme $S$ the groupoid of pairs $(P\to S, m:P\to X)$, where $P\to S$ is a principal $G$-bundle and $m:P\to X$ is a $G$-equivariant map, with morphisms being given by $G$-bundle isomorphisms $u:P\to P'$ satisfying $m'\circ u =m$. The pull-back along a morphism $f:S'\to S$ is given by $(S'\times_S P \to S', m\circ \pr_P)$. The morphism $m:P\to X$ can also be interpreted as a section of the $X$-bundle $P\times_G X \to S$. The stack $X/G$ is called the quotient stack of $X$ with respect to the $G$-action. 
\end{example}
When is comes to locally trivial families, there is some choice involved, namely the choice of the underlying (Grothendieck) topology. Intuitively, one would start with the Zariski topology, but the \'{e}tale or even the smooth topology have their advantages, too. In fact, the quotient stack $X/G$ defined above is usually taken with respect to the smooth or, equivalently, \'{e}tale topology. However, for so-called ``special'' groups $G$ like $\Gl(n)$ we could equivalently take the Zariski topology as every \'{e}tale locally trivial principal $G$-bundle is then already Zariski locally trivial. Notice that $\PGl(d)$ is not special and we should better take the \'{e}tale topology when it comes to principal $\PGl(d)$-bundles and quotient stacks $X/\PGl(d)$.

\begin{definition}
 A 1-morphism (or morphism for short) from a stack $F$ to a stack $F'$ is a natural transformation $\eta:F\to F'$, i.e.\ a family of functors $\eta_S:F(S) \to F'(S)$ compatible with pull-backs along $f:S'\to S$ up to equivalence of functors. In other words, the functors $F'(f)\circ \eta_S$ and $\eta_{S'}\circ F(f)$ from $F(S)$ to $F'(S')$ are equivalent. A 2-morphism $\alpha:\eta\to \eta'$ between 1-morphisms is an invertible  natural transformation $\alpha_S:\eta_S\to \eta'_S$ for every scheme $S$, compatible with pull-backs. In particular, given two stacks $F,F'$, we get a groupoid of morphisms $\MOR(F,F')$ with 1-morphisms being the objects and 2-morphisms being the morphisms. Hence, the category of stacks is a 2-category. 
\end{definition}
Thinking of a set as being a  groupoid having only identity morphisms, we can associate to every scheme $X$ a contravariant functor $h_X:S\mapsto \Mor(S,X)$. As we can glue morphisms, $h_X$ is indeed a stack. The following lemma is very important.
\begin{lemma}[Yoneda-Lemma] The covariant functor $h:X\mapsto h_X$ from schemes to stacks provides a full embedding of the category of schemes into the (2-)category of stacks. Moreover, there is an equivalence of groupoids $\MOR(h_X,F)\cong F(X)$ for every scheme $X$ and every stack $F$, natural in $X$ and $F$. 
\end{lemma}
\begin{exercise}
 Try to prove the Yoneda-Lemma.
\end{exercise}
The lemma is basically saying that the 2-category of stacks is an enlargement of the category of schemes, and we will drop the functor $h$ from notation. Though the definition of a stack  looks very abstract, the reader should not think of a stack $F$ as a complicated functor, but rather as some object of a bigger 2-category containing the category of schemes. The groupoid-valued functor associated to $F$ can be (re)constructed by taking $X\mapsto \MOR(X,F)$. In other words, assume that you have a 2-category $\altC$ with 2-morphisms being invertible, containing $\Sch_\kk$ as a full subcategory, and such that 1-morphisms starting at schemes and 2-morphisms between such 1-morphisms can be glued in a natural way. To every object $F\in \Obj(\altC)$  we can associate the groupoid-valued functor $\MOR(-,F)|_{\Sch_\kk}$ on the category of schemes. It satisfies the gluing axioms given above, and, hence, defines a stack. Thus, we get a covariant functor from $\altC$ to the category of stacks showing that stacks form some sort of ``natural'' enlargement. 
\begin{exercise} \label{stack_morphisms} \hfill
 \begin{enumerate}
    \item Let $X$ be a scheme with a right action of an algebraic group $G$. Consider the trivial principal $G$-bundle $\pr_X: X\times G\to X$ on $X$  and the $G$-equivariant map $m:X\times G\to X$ given by the group action. According to our definition of a quotient stack, the pair $(X\times G\to X,m)$ defines an element $\rho$ in $X/G(X)$. Check the Yoneda lemma by constructing a morphisms $\rho:X\to X/G$ which is called the ``standard atlas'' of $X/G$. 
  \item Given two schemes $X, Y$ and two algebraic  groups $G, H$ acting on $X$ and $Y$ respectively. Assume that $\phi:G\to H$ is a group homomorphism and that $f:X\to Y$ is a morphism of schemes satisfying $f(xg)=f(x)\phi(g)$ for all $g\in G$ and $x\in X$. Construct a natural 1-morphism $\fst:X/G\to Y/H$ of stacks such that 
  \[ \xymatrix { X\ar[d]^{\rho_X} \ar[r]^f & Y \ar[d]^{\rho_Y} \\ X/G \ar[r]^{\fst} & Y/H} \]
  commutes.\textbf{Warning:} Not every morphism $X/G\to Y/H$ is of this form. 
  \item Consider the special case $H=\{1\}$, and show that $\fst\mapsto f:=\fst\circ \rho_X$ defines an equivalence from $\MOR(X/G,Y)$ to the set $\Mor(X,Y)^G$ of $G$-invariant morphisms, thought as a groupoid.
  \item More general, given a scheme $Y$ and a stack $F$, show that $\MOR(F,Y)$ is essentially a set, i.e.\ the only 2-morphisms are the identity morphisms.
 \end{enumerate}
\end{exercise}
Let us come back to moduli spaces. The moduli problem of classifying $G$-homogeneous spaces $P$ together with $G$-equivariant maps $P\to X$ for some fixed scheme $X$ with an action of an algebraic group $G$, has a natural generalized ``moduli space'', namely the quotient stack $X/G$. This is not a deep insight, but just the definition of the associated moduli functor. Note that the isomorphism classes of the $\kk$-valued points of $X/G$, i.e.\ $X/G(\Spec\kk)/_\sim$, is the set of $G$-orbits in $X$ justifying the notation.\\
Quotient stacks are also very helpful when it comes to other moduli problems as the following example shows, and their usefulness cannot be overestimated.
\begin{example} \label{stack_example_2} \rm
Consider the stack of finite dimensional representations of a  $\kk$-algebra $A$. Assume that $A$ is finitely presented, i.e.\ $A$ is generated by a set\footnote{The notation in this example has been chosen with an eye towards the next section.}  $Q_1$ of finitely many elements $\alpha_1, \ldots, \alpha_n$ satisfying  a finite set of relations $R=\{r_1,\ldots,r_m\}$. Fix a  ``dimension'' $d\in \NN$ and put $X_d =\Hom_\kk(\kk^d,\kk^d)^n=\prod_{\alpha\in Q_1}\Hom_\kk(\kk^d,\kk^d)$ and $X^R_d=\{(M_\alpha)_{\alpha\in Q_1}\mid r_j(M_{\alpha_1},\ldots,M_{\alpha_n})=0 \mbox{ for all }1\le j\le m\}$. We claim that the moduli stack $A\rep_d$ of $d$-dimensional representations of $A$ is equivalent to the quotient stack $X^R_d/\Gl(d)$ with $\Gl(d)$ acting by conjugation on $\Hom_\kk(\kk^d,\kk^d)$. Indeed, a family of  $d$-dimensional representations on $S\in \Sch_\kk$ is uniquely determined by a vector bundle $V$ of rank $d$ on $S$ and vector bundle endomorphisms $\hat{\alpha}$ associated to  $\alpha\in Q_1$ satisfying the relations $r_1,\ldots,r_m$. Consider the frame bundle  $\Fr(V)=\{ (s,\tau) \mid s\in S, \tau\in \Hom_\kk(\kk^d,V_s)\mbox{ is invertible}\}$ of $V$ parameterizing all possible choices of a basis in all possible fibers of $V$. It comes with a projection to $S$ and a right action of $\Gl(d)$ by composition with $\tau\in \Hom_\kk(\kk^d,V_s)$. 
\begin{exercise}
 Show that $\Fr(V)$ is a principal $\Gl(d)$-bundle.
\end{exercise}
There is also a $\Gl(d)$-equivariant map $m(V,(\hat{\alpha})_{\alpha\in Q_1})$ from $\Fr(V)$ into $X^R_d$  mapping a pair $(s,\tau)$ to $(M_\alpha:=\tau^{-1}\circ \hat{\alpha}|_{V_s} \circ \tau)_{\alpha\in Q_1}$. 
\begin{exercise}
 Convince yourself that the map $(V,(\hat{\alpha})_{\alpha\in Q_1})\longmapsto \big(\Fr(V),m(V,(\hat{\alpha})_{\alpha\in Q_1})\big)$ extends to a functor from the groupoid of families of $d$-dimensional $A$-representations into the groupoid $X^R_d/\Gl(d)(S)$. Show furthermore that this functor is compatible with pull-backs, and, thus, defines a morphism $A\rep_d\to X^R_d/\Gl(d)$ of stacks.
\end{exercise}
Conversely, given a principal $\Gl(d)$-bundle $P$ on $S$ and a $\Gl(d)$-equivariant map $m:P\to X^R_d$, we can consider the trivial vector bundle $P\times \kk^d$ on $P$ which comes with a natural $\Gl(d)$-action compatible with the projection to  $P$. Moreover, picking the component of $m$ associated to $\alpha\in Q_1$, we get an endomorphism  $\hat{\alpha}$ of this trivial bundle. $\Gl(d)$-equivariance of $m$ ensures that $\hat{\alpha}$ commutes with the $\Gl(d)$-action on $P\times \kk^d$. By taking the  $\Gl(d)$-quotient, we obtain a vector bundle $V=P\times_{\Gl(d)}\kk^d$ of rank $d$ on $S$ along with vector bundle endomorphisms $\hat{\alpha}$ satisfying the relations $r_1,\ldots,r_m$.
\begin{exercise}
 Show that this construction extends to a functor between groupoids, compatible with pull-backs. Hence, we obtain a morphism from the quotient stack $X^R_d/\Gl(d)$ to $A\rep_d$. Prove that this morphism is an inverse (up to 2-isomorphism) of the morphism constructed above.  
\end{exercise}
Thus, the claim is proven and the stack $A\rep$ of $A$-representations is isomorphic to $\sqcup_{d\in\NN} X^R_d/\Gl(d)$. 
\end{example}
\begin{exercise} \label{stack_projective_reps} Use a similar idea of frame bundles parameterizing tuples $(s\in S,\Mat_\CC(r,r)\xrightarrow{\sim}\EE_s)$ for a locally trivial family $\EE$ on $S$ of $\CC$-algebras isomorphic to $\Mat_\CC(r,r)$ to show that the stack of projective $A$-representations is given by $\sqcup_{d\in\NN} X^R_d/\PGl(d)$. As in Exercise \ref{stack_morphisms}(2), we obtain a morphism $X^R_d/\Gl(d)\longrightarrow X^R_d/\PGl(d)$ by means of the group homomorphism $\Gl(d)\to \PGl(d)$. Show that this morphism is mapping the $A$-representation on $V$ to the projective $A$-representation on $\PP(V)$, in other words, forget $V$ and keep $\EE nd(V)$ together with the algebra homomorphism  $A\to \Gamma(S,\EE nd(V))$. 
\end{exercise}
\begin{example} \rm
The ``geometry'' of the moduli stack $\Coh^{X}$ of coherent sheaves on a smooth projective variety $X$ is more involved. First of all, it decomposes into components $\Coh^{X}_c$ indexed by numerical data like Chern classes similar to the dimension of a representation. Unfortunately, a component can not be written as a quotient stack. However, every component $\Coh^{X}_{c}$ is the nested union of ``open'' substacks $\Coh^X_{c,i}, i\in \NN,$ which can be written as a quotient stack $Y_{c,i}/\Gl(n_{c,i})$. Note that $n_{c,i}$ grows with $i$. More details can be found is section 9 of \cite{JoyceI}.  
\end{example}

The following definition of a fiber product is very important.
\begin{definition}[fiber product]
 Given two morphisms $\fst:\XXX\to \ZZZ$ and $\gst:\YYY\to \ZZZ$ of groupoid-valued functors, we define the fiber product $\XXX\times_\ZZZ \YYY$ as the groupoid-valued functor such that 
\[ \Obj\XXX\times_\ZZZ \YYY(S)=\{ (x,y,w)\mid x\in \Obj\XXX(S), y\in  \Obj\YYY(S), w\in \Mor_{\ZZZ(S)}(\fst_S(x),\gst_S(y)) \}, \] and
\begin{eqnarray*} \lefteqn{\Mor_{\XXX\times_\ZZZ\YYY(S)}\Big((x,y,w),(x',y',w')\Big)} \\&=& \{ (u,v)\in \Mor_{\XXX(S)}(x,x')\times \Mor_{\YYY(S)}(y,y') \mid \xymatrix @C=0.6cm @R=0.6cm{ \fst_S(x) \ar[r]^w \ar[d]_{\fst_S(u)} & \gst_S(y) \ar[d]^{\gst_S(v)} \\ \fst_S(x') \ar[r]^{w'} & \gst_S(y') } \mbox{commutes } \}
  \end{eqnarray*}
for every $S\in \Sch_\kk$.  
\end{definition}
\begin{exercise} Show the main properties of the fiber product.
\begin{enumerate} 
 \item Prove that $\XXX\times_\ZZZ \YYY$ is a stack if $\XXX,\YYY,\ZZZ$ were stacks.
 \item Construct two morphisms $\pr_\XXX: \XXX \times_\ZZZ \YYY \longrightarrow \XXX$ and $\pr_\YYY: \XXX \times_\ZZZ \YYY \longrightarrow \YYY$ of groupoid-valued functors and a 2-morphism $\omega:\fst\circ\pr_\XXX\to \gst\circ\pr_\YYY$. Show that the following universal property holds. Given a groupoid-valued functor $\mathfrak{T}$, two morphisms $\mathfrak{p}:\mathfrak{T}\to \XXX$, $\mathfrak{q}:\mathfrak{T}\to \YYY$ and a 2-morphism $\eta:\fst\circ\mathfrak{p} \to \gst\circ\mathfrak{q}$, there is a unique morphism $\mathfrak{r}:\mathfrak{T}\to \XXX\times_\ZZZ \YYY$ such that $\pr_\XXX\circ\mathfrak{r}=\mathfrak{p}$ and $\pr_\YYY\circ\mathfrak{r}=\mathfrak{q}$.
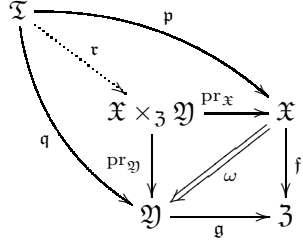
 \[ \xymatrix { \mathfrak{T} \ar@{.>}[dr]^{\mathfrak{r}} \ar@/^1pc/[drr]^{\mathfrak{p}} \ar@/_1pc/[ddr]_{\mathfrak{q}} & & \\ & \XXX\times_\ZZZ \YYY \ar[r]^{\pr_\XXX} \ar[d]_{\pr_\YYY} & \XXX \ar[d]^\fst \ar@{=>}[dl]^\omega \\ & \YYY \ar[r]_\gst & \ZZZ } \]
\end{enumerate} 
\end{exercise}
When it comes to quotient stacks, the following examples are very useful.
\begin{exercise} \hfill
\begin{enumerate}
 \item Assume $\XXX=X, \YYY=Y\in \Sch_\kk$ and $\ZZZ=Z/G$ for some algebraic group $G$ acting on a scheme $Z$. The morphisms $\fst:X\to Z/G$ and $\gst:Y \to Z/G$ are given by principal $G$-bundles $P\to X$ and $Q\to Y$ together with $G$-equivariant morphisms $f:P\to Z$ and $g:Q\to Z$ respectively. Show that the fiber product $X\times_{Z/G} Y$ is given by the scheme $Iso_{f,g}(P,Q)\subseteq Iso(P,Q)$ over $X\times Y$ given by $\{(x,y,w)\mid x\in X,y\in Y, w:P_x\xrightarrow{\sim}Q_y \;G\mbox{-equivariant such that }f|_{P_x}=g|_{Q_x}\circ w\}$.
\item Assume furthermore $X=Z$ and $P=X\times G\xrightarrow{\pr_X}X$ with $f:X\times G\to X$ being the group action. Hence, $\fst$ is the standard atlas $\rho:X\to X/G$. Show that $Iso_{f,g}(P,Q)$ is isomorphic to  $Q$, and 
\[ \xymatrix { Q \ar[r]^g \ar[d] & X \ar[d]^\rho \\ Y \ar[r]^{\gst} & X/G }\]
is the fiber product diagram, i.e.\ a cartesian square. Hence, $\rho:X\to X/G$ is the universal principal $G$-bundle. 
\end{enumerate} 
\end{exercise}
\begin{example} \label{example_fiber_product} \rm
Let $\phi:G\to K$ and $\psi:G\to K$ be homomorphisms between algebraic groups $G,H,K$ acting on $X,Y$ and $Z$ respectively. Moreover, let $f:X\to Z$ and $g:Y\to Z$ be two morphisms such that $f(xg)=f(x)\phi(g)$ and $g(yh)=g(y)\psi(h)$ for all $x\in X,y\in Y,g\in G,h\in H$. As we have seen in Exercise \ref{stack_morphisms}, this induces morphisms $\fst:X/G\to Z/K$ and $\gst:Y/H\to Z/K$. Then, $X/G\times_{Z/K} Y/H$ is the quotient stack $(X\times Y)\times_{(Z\times Z)} (Z\times K)/(G\times H)$ using the group actions $(x,y)(g,h)=(xg,yh)$, $(z,k)(z\phi(g),\phi(g)^{-1}k\psi(h))$, $(z_1,z_2)(g,h)=(z_1\phi(g),z_2\psi(h))$ of $G\times H$ on $X\times Y$, $Z\times K$, $Z\times Z$ and the $G\times H$-equivariant morphisms $X\times Y\ni (x,y)\mapsto (f(x),g(y))\in Z\times Z$, $Z\times K\ni(z,k) \mapsto (z,zk) \in Z\times Z$.
\end{example}
\begin{exercise} \label{fiber_projectivization}
Use the previous example to show that every fiber of the morphism $X^R_d/\Gl(d)\longrightarrow X^R_d/\PGl(d)$ constructed in Exercise \ref{stack_projective_reps} is isomorphic to $\Spec\CC/\GG_m$. Interpret this result in terms of (projective) $A$-representations. (cf.\ Example \ref{projectivization})
\end{exercise}

The following definition is slightly stronger than the one used in the literature as we do not have algebraic spaces at our disposal. However, it will be sufficient for our purposes. 
\begin{definition} A morphism $\fst:\XXX\to \ZZZ$ is called representable if for every morphisms $\gst:Y\to \ZZZ$ from a scheme $Y$ into $\ZZZ$, the fiber product $\XXX\times_\ZZZ Y$ is (represented by) a scheme. In such a situation, we call $\fst$ smooth, surjective etc.\ if $\XXX\times_\ZZZ Y \xrightarrow{\pr_Y} Y$ is smooth, surjective etc.
\end{definition}

\begin{exercise} \label{representable} Here are some examples of representable morphisms.
 \begin{enumerate}
  \item Show that every morphism between schemes is representable.
  \item Prove that the standard atlas $\rho:X\to X/G$ is representable, smooth and surjective. Hint: Every algebraic group is a smooth scheme. 
  \item Use Example \ref{example_fiber_product} to show that $\fst:X/G\to Z/K$ is representable if $\phi:G\to K$ is injective. Give a counterexample for the converse statement.
  \item Prove that the diagonal $\Delta_\ZZZ:\ZZZ \to \ZZZ\times_{\Spec\kk}\ZZZ$ is representable if and only if every morphism $\fst:X\to \ZZZ$ from a scheme $X$ is representable. Hint: $X\times_\ZZZ Y= (X\times Y)_{(\ZZZ\times \ZZZ)} \ZZZ$.
 \end{enumerate}
\end{exercise}
\begin{definition}
 A stack $\XXX$ is called algebraic or an Artin stack if 
 \begin{enumerate}
  \item[(i)] $\Delta_\XXX:\XXX\to \XXX\times \XXX$ is representable (cf.\ Exercise \ref{representable}(4)) and
  \item[(ii)] there is a smooth, surjective morphism $\rho:X\to \XXX$ from a scheme $X$.
 \end{enumerate}
In such a situation, we call $\rho:X\to \XXX$ an atlas of $\XXX$. 
\end{definition}
In a suitable sense, the algebraic stack $\XXX$ is a quotient of its atlas $X$ similar to the concept of an algebraic space. However, the quotient is taken in the category of groupoids and not in the category of sets as before. As we have seen in Exercise \ref{representable}, every quotient stack is an Artin stack with standard atlas $\rho:X\to X/G$. By taking $X^R=\sqcup_{d\in \NN}X^R_d\to A\rep$, the  moduli stack of finite dimensional representations of a  finitely represented $\CC$-algebra $A$ is also algebraic. Finally, using $Y=\sqcup_{c,i} Y_{c,i}\to \Coh^X$, we see that the moduli stack of coherent sheaves on a smooth projective variety $X$ is also an Artin stack.

\section{Quiver representations and their moduli}

\subsection{Quivers and $\kk$-linear categories}
Recall that a groupoid is a category generalizing groups and  sets. Similarly, there is a  categorical concept interpolating between $\kk$-algebras and sets. 
These are the so-called $\kk$-linear categories. A category $\mathcal{A}$ is called $\kk$-linear if the morphism sets $\Mor_{\mathcal{A}}(x,y)$ have the structure of a $\kk$-vector space such that the composition of morphisms is $\kk$-bilinear. 
As usual, we write $\Hom_{\mathcal{A}}(x,y)$ for the $\kk$-vector space of morphisms from $x$ to $y$ and $\End_{\mathcal{A}}(x)=\Hom_{\mathcal{A}}(x,x)$ for the $\kk$-algebra of endomorphisms of $x\in \Obj(\mathcal{A})$.  A $\kk$-linear category with one object is just a $\kk$-algebra. On the other hand, $\kk$-linear categories with as less morphisms as possible are uniquely classified by their set of objects since any morphism must be zero or a multiple of the identity of some object. 
Another standard example of a $\kk$-linear category is given by the category $\Vect_\kk$ of finite dimensional $\kk$-vector spaces. A finite dimensional representation of a $\kk$-linear category $\mathcal{A}$ is simply given by a functor $V:\mathcal{A}\to \Vect_\kk$. Indeed, if the category $\altA$ has only one object $\star$, $V(\star)$ is just a finite dimensional representation of the endomorphism algebra $\End_\altA(\star)$. As we have seen in the previous section, generators of algebras are very useful when it comes to the construction of moduli stacks. The  analogue in the context of $\kk$-linear categories is called a quiver. A quiver consists of a set of objects $Q_0$ and a set of ``arrows'' $Q_1$ along with  maps $s,t:Q_1 \to Q_0$ indicating the \underline{s}ource and the \underline{t}arget of an arrow. We do not require a composition law nor identity morphisms. 
Given a $\kk$-linear category $\altA$, a quiver in $\altA$ satisfies $Q_0\subseteq\Obj(\altA), Q_1\subseteq \Mor(\altA)$ and $s,t$ are given by restriction of the corresponding maps on $\Mor(\altA)$ to $Q_1$. We say that $\altA$ is generated by a quiver $Q$, if the smallest $\kk$-linear category containing $Q$ is $\altA$ which implies $Q_0=\Obj(\altA)$. There is a biggest $\kk$-linear category generated by a given quiver $Q$, the so-called path category $\kk Q$ of $Q$. A morphism of $\kk Q$ from $x\in Q_0$ to $y\in Q_0$ is a $\kk$-linear combination of chains $x=x_1\to x_2\to \ldots \to x_{n-1}\to x_n=y$ of composable arrows in $Q_1$. We also need to add an identity morphism and its $\kk$-linear multiples. 
\begin{exercise} Construct a category of quivers such that $Q\mapsto \kk Q$ is a functor from this category to the category of (small) $\kk$-linear categories. Construct a right adjoint of this functor. 
\end{exercise}
\begin{exercise} Show that there is a bijection between representations of $\kk Q$ and representations $V$ of  $Q$ associating to every $i\in Q_0$ a vector space $V_i$ and to every arrow $\alpha:i\to j$ in $Q_1$ a $\kk$-linear map $V(\alpha):V_i\to V_j$.
\end{exercise}
Given a $\kk$-linear category $\altA$ and a generating quiver $Q$ in $\altA$, we get a full functor $\kk Q \twoheadrightarrow \altA$ which is a bijection on the set of objects. The kernel is a $\kk$-linear subcategory $\mathcal{I}$ in $\kk Q$ which has the property $a\circ b\in \Mor(\mathcal{I})$ if $a\in \Mor(\mathcal{I})$ or $b\in \Mor(\mathcal{I})$ categorifying the concept of an ideal. A generating quiver for $\mathcal{I}$ is uniquely determined by its set of arrows
$R\subseteq \Mor(\kk Q)$ which are called relations. Conversely, every quiver $Q$ with relations give rise to a $\kk$-linear category $\kk Q/(R)$ uniquely defined up to isomorphism. Conversely, every $\kk$-linear category $\altA$ can be written like this (up to isomorphism) in many ways.\\

Throughout this paper we will only consider finite quivers, i.e.\ $|Q_0|<\infty$ and $|Q_1|<\infty$ and similarly for the relations. Hence, the $\kk$-linear categories $\altA$ which can be described by a finite quiver with finitely many relations are exactly the finitely presented $\kk$-linear categories. 
\begin{exercise}
Show that the category of $\kk$-linear categories $\altA$ with finite set of objects is equivalent to the category of $\kk$-algebras together with a distinguished finite set $\{e_i\}_{i\in I}$ of mutually orthogonal idempotent elements $e_i$ such that $1=\sum_{i\in I}e_i$. Hint: Put $A:=\oplus_{i,j\in \Obj(\altA)} \Hom_\altA(i,j)$ and $e_i=\id_i$ for all $i\in \altA$. Moreover, prove that the category of representations of such a $\kk$-linear category is isomorphic to the category of representations of the associated algebra.  
\end{exercise}
Using the last exercise, we can also talk about the path $\kk$-algebra of a quiver with finite set $Q_0$ and its representations. Note that the path $\kk$-algebra has a distinguished family $(e_i)_{i\in Q_0}$ of mutually orthogonal idempotent elements summing up to $1$. 

\subsection{Quiver moduli spaces and stacks}

Generalizing the moduli functor $A\rep$ of finite dimensional representations of a given $\kk$-algebra $A$ (see Example \ref{moduli_algebra}), we define the moduli functor $\altA\rep$ of finite dimensional representations of a $\kk$-linear category $\altA$ as follows. To every scheme $S$ over $\kk$ we associate the isomorphism groupoid $\altA\rep(S)$ of the category of functors $\altA\to \Vect_S$, where $\Vect_S$ is the category of vector bundles on $S$. 
\begin{exercise}
 Show that $\altA\rep$ is a stack, i.e.\ satisfies the gluing axiom for groupoid-valued functors.
\end{exercise}
If $\altA$ is represented by a quiver $Q$ with  relations $R\subseteq \Mor(\kk Q)$, the category $\altA\rep(S)$ is equivalent to the category of families $(V_i)_{i\in Q_0}$ of vector bundles on $S$ together with vector bundle morphisms $\hat{\alpha}=V(\alpha):V_i\to V_j$ such that $V(r)=r\big((\hat{\alpha})_{\alpha\in Q_1}\big)=0$ for all $r\in R$, where we extended $V$ from $Q_1$ to $\Mor(\kk Q_1)\supseteq Q_1$. 

Let us assume that $Q_0,Q_1$ and $R$ are finite sets. Using Example \ref{stack_example_2}, it should not come as a surprise that $\altA\rep$ is isomorphic to  a disjoint union $\Mst^R:=\sqcup_{d\in \NN^{Q_0}} \Mst^R_d$ of quotient stacks $\Mst^R_d:=X^R_d/G_d$ with 
\[X^R_d:=\{ (M_{\alpha})_{\alpha\in Q_1}\in X_d \mid r\big((M_\alpha)_{\alpha\in Q_1}\big)=0 \,\forall \, r\in R\} \subseteq X_d:= \prod_{Q_1\ni\alpha:i\to j} \Hom_\kk(\kk^{d_i},\kk^{d_j})\]  and $G_d=\prod_{i\in Q_0}\Gl(d_i)$ acting on $X_d$ by simultaneous conjugation. The ``dimension vector'' $d\in \NN^{Q_0}$ is fixing $\dim \altV=(\rk \altV_i)_{i\in Q_0}$. \\
Similarly, given a sequence of dimension vectors $d^{(1)},\ldots,d^{(r)}$ we denote with $X_{d^{(1)},\ldots,d^{(r)}}\subseteq X_{d^{(1)}+\ldots+d^{(r)}}$ the affine subvariety parameterizing linear maps preserving the standard flag $0 \subseteq \kk^{d^{(1)}_i} \subseteq \kk^{d^{(1)}_i}\oplus\kk^{d^{(2)}_i} \subseteq \ldots \subseteq \kk^{d^{(1)}_i}\oplus\ldots\oplus \kk^{d^{(r)}_i}$ for every $i\in Q_0$. The subgroup $G_{d^{(1)},\ldots,d^{(r)}}\subseteq G_{d^{(1)}+\ldots+d^{(r)}}$ is defined in the same way. Finally, we put $X^R_{d^{(1)},\ldots,d^{(r)}}:=X_{d^{(1)},\ldots,d^{(r)}}\cap X^R_{d^{(1)}+\ldots+d^{(r)}}$.
\begin{exercise}
 Show that the stack of all successive extensions 
 \begin{eqnarray*}
 &0\to V^{(1)}\to \hat{V}^{(2)} \to V^{(2)} \to 0, &\\
&0\to \hat{V}^{(2)}\to \hat{V}^{(3)} \to V^{(3)} \to 0, &\\  
& \vdots & \\
&0\to \hat{V}^{(r-1)}\to \hat{V}^{(r)} \to V^{(r)} \to 0 &
 \end{eqnarray*}
of quiver representations satisfying the relations $R$ and with $\dim V^{(j)}=d^{(j)}$ for all $1\le j\le r$ is given by the quotient stack $\Mst^R_{d^{(1)},\ldots,d^{(r)}}=X^R_{d^{(1)},\ldots,d^{(r)}}/G_{d^{(1)},\ldots,d^{(r)}}$. Hint: The standard flag introduced above defines a standard successive extension of $Q_0$-graded vector spaces of prescribed dimension vectors. Given a family of successive extensions, consider the principal $G_{d^{(1)},\ldots,d^{(r)}}$-bundle parameterizing all isomorphism from the standard extension to the fibers of the family, and proceed as usual.  
\end{exercise}
We are mainly interested in the following type of relations. A potential $W$ is an element of the vector space $\kk Q/[\kk Q,\kk Q]$, where $[\kk Q,\kk Q]$ denotes the $\kk$-linear span (and not the spanned ideal) of all commutators. Note that $\kk Q/[\kk Q,\kk Q]$ is the $0$-th Hochschild homology of the $\kk$-linear category $\kk Q$. Convince yourself that $W$ is essentially just a $\kk$-linear combination of equivalence classes of cycles in $Q$ with two cycles being equivalent if they can be transformed into each other by a cyclic permutation. 
\begin{example} \rm
 The three elements $[x,y]z=xyz-yxz$, $[z,x]y$ and $[y,z]x$ in $\kk Q^{(3)}$ of the 3-loop quiver $Q^{(3)}$
 
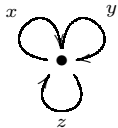
\[ \xymatrix { \bullet \ar@(u,r)^y \ar@(dr,dl)^z \ar@(l,u)^x }\]
 define the same potential $W$.
\end{example}
For a fixed potential $W=\sum_{l=1}^L a_l \cdot[C_l]$ we define relations $\partial W/\partial \alpha\in \Hom_{\kk Q}(j,i)$ for every $\alpha:i\to j$ in $Q_1$ as follows.
\[ \frac{\partial W}{\partial \alpha}:=\sum_{l=1}^L a_l \cdot\sum_{C_l=u\alpha v} vu \]
with $a_l\in \kk$, where the second sum is over all occurrences of $\alpha$ in a fixed representative of an equivalence class $[C_l]$ of cycles in $Q$.
\begin{exercise}
 Show that the definition of $\partial W/\partial \alpha$ is independent of the choice of the representative $C_l\in [C_l]$ for all $1\le l\le L$. 
\end{exercise}

\begin{example} \rm
 Using the potential $W=[x,y]z=xyz-yxz$ from the previous example, we compute
 \begin{eqnarray*}
  \frac{\partial W}{\partial x} & = & yz-zy\; =\;[y,z], \\
  \frac{\partial W}{\partial y} & = & zx-xz\; =\;[z,x], \\
  \frac{\partial W}{\partial z} & = & xy-yx\; =\;[x,y].
 \end{eqnarray*}
Convince yourself that $W=[z,x]y$ and $W=[y,z]x$ provide the same relations. 
\end{example}

Given a dimension vector $d\in \NN^{Q_0}$ and a potential $W=\sum_{l=1}^L a_l \cdot[C_l]$ with $C_l=\alpha_l^{(1)}\circ \ldots \circ\alpha^{(n_l)}_l$, we define the following function
\[ \Tr(W)_d:X_d\ni (M_\alpha)_{\alpha\in Q_1} \longmapsto \sum_{l=1}^L a_l\cdot \Tr\big(M_{\alpha_l^{(1)}}\cdot \ldots \cdot M_{\alpha^{(n_l)}_l}\big) \in \AAA \]
which is independent of the choice of the representative $C_l\in [C_l]$ as the trace is invariant under cyclic permutation. By the same argument, $\Tr(W)_d$ is $G_d$-invariant, and induces a function $\WWW_d:\Mst_d\to \AAA$ on the quotient stack. 
\begin{exercise}
 Let us take the relations $R=\{\partial W/\partial \alpha \mid \alpha\in Q_1\}$. Show that $X^R_d=\Crit(\Tr(W)_d)$ is the critical locus of $\Tr(W)_d$, and similarly $\Mst^R_d=\Crit(\WWW_d)$. 
\end{exercise}
Throughout the paper we will use the superscript $W$ instead of the superscript $R$ for $R=\{\partial W/\partial \alpha \mid \alpha\in Q_1\}$, and no superscript if $W=0$. We will also use the notation $\Jac(Q,W)$ for the so-called Jacobi algebra $\kk Q/(R)$.\\

The moduli stack $\Mst^R_d$ has a coarse moduli space $\Msp^{W,ssimp}_d$ parameterizing semisimple (direct sums of simple) representations of dimension vector $d$. It is an affine scheme given by $\Spec \kk[X^R_d]^{G_d}$ with $\kk[X^R_d]^{G_d}$ denoting the $G_d$-invariant regular functions on the affine scheme $X^R_d$. 
\begin{example} \rm
 For the 3-loop quiver $Q^{(3)}$ with potential $W=[x,y]z$, the scheme $X^W_d$ parametrizes triples of commuting $d\times d$-matrices $M_x,M_y,M_z$. Hence, a simple representation of the Jacobi algebra $\Jac(Q^{(3)},W)=\kk[x,y,z]$ is one-dimensional and determined by $(M_x,M_y,M_z)\in \AA^3$. Therefore, $\Msp^W_d=\Sym^d(\AA^3)=(\AA^3)^d/S_d$.
\end{example}
Let us finally introduce a stability condition by choosing a tuple $\zeta\in \mathbb{H}_+^{Q_0}$ of complex numbers in the (extended) upper half plane $\mathbb{H}_+$ giving rise to the ``central charge'' $Z(V):=\zeta\cdot \dim V=\sum_{i\in Q_0}\zeta_i\dim V_i \in \mathbb{H}_+$ for every representation $V$ of $Q$. 
\begin{definition} A representation $V\not=0$ of a quiver $Q$ (with relations) is called $\zeta$-semistable if
 \[ \arg Z( V') \le \arg Z( V)  \]
 for all proper subrepresentations $V'\subset V$. If the inequality is strict, $V$ is called $\zeta$-stable. The real number $\mu(V):=-\cot(\arg Z(V))$  is called the slope of $V$. Hence, $V$ is semistable if and only if $\mu(V')\le \mu(V)$ for all proper subrepresentations $V'\subset V$. 
\end{definition}
\begin{exercise} \label{semistable_reps} Let us show that semistable representations of the same slope $\mu$ form a nice full subcategory.
 \begin{enumerate}
  \item Consider a morphism $f:V^{(1)}\to V^{(2)}$ of semistable representations of slopes $\mu(V^{(1)})>\mu(V^{(2)})$. Show that $f=0$. Hint: Relate the slope of $V/\ker(f)=\im(f)$ to $\mu(V^{(1)})$ and to $\mu(V^{(2)})$ by drawing the central charges of all objects involved. 
  \item Using the notation of the first part, let us assume $\mu(V^{(1)})=\mu(V^{(2)})$ for the semistable representations $V^{(1)},V^{(2)}$. Show that $\ker(f)$ and $\coker(f)$ are also semistable of the same slope $\mu(V^{(1)})$. In particular, the semistable representations of a fixed slope $\mu$ form a full abelian subcategory.
  \item Show that the stable objects of slope $\mu$ are the simple objects in the full abelian subcategory of semistable representations of slope $\mu$.
  \item Prove that the extension of two semistable representations of slope $\mu$ is again semistable of the same slope.
 \end{enumerate}
\end{exercise}
Every representation $V$ of a quiver (with relations) has a unique Harder--Narasimhan filtration, i.e.\ a finite filtration $0\subset V^{(1)} \subset \ldots\subset V^{(r)}=V$ such that the subquotients $V^{(i)}/V^{(i-1)}$ are semistable of slope $\mu^{(i)}$ satisfying $\mu^{(1)}>\ldots> \mu^{(r)}$.
\begin{exercise} Let us prove the last statement in three steps.
 \begin{enumerate}
  \item Show that $V$ has a maximal nonzero subrepresentation of maximal slope. Hint: Show that the set of slopes of subrepresentations of $V$ has a maximal element. Use Exercise \ref{semistable_reps}(4) to construct a maximal subrepresentation of maximal slope.
  \item Use  Exercise \ref{semistable_reps}(1) to construct  a Harder--Narasimhan filtration. Hint: Let $V^{(1)}$ be the subrepresentation constructed in the first step, and let $V^{(2)}$ be the preimage of a maximal subrepresentation in $V/V^{(1)}$ of maximal slope. Proceed in this way, and use the previous exercise to estimate the slopes. 
  \item Prove the uniqueness of this filtration by applying  Exercise \ref{semistable_reps}(1) once more. 
 \end{enumerate}

\end{exercise}

We denote by $X^{R,\zeta-ss}_d$ the subscheme of linear maps $(M_\alpha)_{\alpha\in Q_1}$ such that the induced quiver representation on $(\kk^{d_i})_{i\in Q_0}$ is $\zeta$-semistable. It is open and stable under the $G_d$-action. Hence, we can form the quotient stack $\Mst^{R,\zeta-ss}_d=X^{R,\zeta-ss}_d/G_d$ of $\zeta$-semistable representations of dimension vector $d$. The open subscheme $X^{R,\zeta-st}_d\subseteq X^{R,\zeta-ss}_d$ and the open substack $\Mst^{R,\zeta-st}_d\subset \Mst^{R,\zeta-ss}_d$ of $\zeta$-stable representations are defined accordingly.  \\
If $\zeta_i=-\theta_i+\sqrt{-1}$ with $\theta_i\in \ZZ$ for all $i\in Q_0$, one can linearize the $G_d$-action on the trivial line bundle over $X_d$ using the character \[G_d\ni (g_i)_{i\in Q_0} \longmapsto \prod_{i\in Q_0}\det(g_i)^{\theta\cdot d-|d|\theta_i} \in \GG_m\]
with $\theta\cdot d=\sum_{i\in Q_0}\theta_id_i$ and $|d|:=\sum_{i\in Q_0}d_i$. A.\ King showed in \cite{King} that $X^{R,\zeta-ss}_d$ is the subscheme of semistable points in $X^R_d$ with respect to this linearization. Hence, a GIT-quotient $X^{R,\zeta-ss}_d/\!\!/G_d=\Msp^{R,\zeta-ss}_d$ with stable sublocus $X^{R,\zeta-st}_d/G_d=\Msp^{R,\zeta-st}_d$ exists. Using this, one can show that  all moduli stacks $\Mst^{R,\zeta-ss}_d$ have a coarse moduli space $\Msp^{R,\zeta-ss}_d$ parameterizing S-equivalence classes of $\zeta$-semistable objects, or, equivalently, isomorphisms classes of $\zeta$-polystable objects of dimension vector $d$ if $\zeta$  is in the complement of a countable union of real hypersurfaces in $\mathbb{H}_+^{Q_0}$. (see \cite{Meinhardt4} Example 3.32) A stability condition $\zeta$ having coarse moduli spaces $\Msp^{\zeta-ss}_d$ for all $d\in \NN^{Q_0}$ is called geometric. In case $\zeta_i=\sqrt{-1}$ for all $i\in Q_0$, i.e.\ $\theta=0$, we write $\Msp^{R,ssimp}_d$ for $\Msp^{R,\zeta-ss}_d$ as its points correspond to isomorphism classes of semisimple $\CC Q$-representations satisfying the relations $R$.
\begin{remark} \rm
Notice that $\GG_m$, embedded into $G_d$ diagonally, acts trivially on $X_d$, and $G_d$ induces a $PG_d:=G_d/\GG_m$-action on $X_d$. The character given above descends to a character on $PG_d$, and $X^{R,\zeta-ss}_d$ is also the semistable locus of $X^{R}_d$ with respect to the $PG_d$-linearization. Hence, $\Msp^{R,\zeta-ss}_d=X^{R,\zeta-ss}_d/\!\!/PG_d$ is also the coarse moduli space for $X^{R,\zeta-ss}_d/PG_d$, the stack of ``$d$-dimensional'' projective semistable quiver representations satisfying the relations $R$. It is not difficult to see that $\Msp^{R,\zeta-st}_d$ is in fact a fine moduli space for $X^{R,\zeta-st}_d/PG_d$, in other words, there is an isomorphism $X^{R,\zeta-st}_d/PG_d\cong \Mst^{R,\zeta-st}_d$ of stacks. In particular, $\Msp^{R,\zeta-st}$ carries a universal family $\altP$ of projective stable quiver representations satisfying our relations $R$. The reader should compare this with our final remarks in  Example \ref{projectivization} and the two lessons we have mentioned after Example \ref{S_equivalence}. The morphism $X^{R,\zeta-st}_d/G_d\longrightarrow X^{R,\zeta-st}_d/PG_d\cong \Msp^{R,\zeta-st}_d$ is not an isomorphism. It is not hard to see that this map has a right inverse, i.e.\ a section, if $\gcd(d):=\gcd(d_i:i\in Q_0)=1$. (See \cite{Reineke5}, Section 5.4 for more details) Such a section is nothing else than a family $V=\bigoplus_{i\in Q_0}V_i$ of stable quiver representations on $\Msp^{R,\zeta-st}_d$ such that $\altP=\PP(V)$. The section and within the family is not unique. Any two sections corresponding to $V^{(1)}$ and $V^{(2)}$ differ (up to isomorphism) by a line bundle $L$ on $\Msp^{R,\zeta-st}$ with $V^{(2)}\cong V^{(1)}\otimes_{\mathcal{O}_{\Msp^{R,\zeta-st}_d}} L$ as in Example \ref{projectivization}. (See also Exercise \ref{fiber_projectivization}) Therefore, $V$ on $\Msp^{R,\zeta-ss}$ is only universal up to this weaker equivalence. It has been shown in \cite{Reineke6}, Thm.\ 3.4 that under some mild conditions on the pair $(d,\zeta)$ the moduli space $\Msp^{\zeta-st}_d$ has no ``universal family'' $V$, i.e.\ $X^{\zeta-st}_d/G_d\to \Msp^{\zeta-st}_d$ has no section, if $\gcd(d)>1$.  
\end{remark}
\begin{definition} \label{generic_stability}
 A stability condition $\zeta$ is called generic if $\langle d,d'\rangle=0$ for all $d,d'\in \Lambda^\zeta_\mu:=\{e\in \NN^{Q_0}\mid e=0 \mbox{ or }e\mbox{ has slope }\mu\}$ and all $\mu\in \mathbb{R}$, where $\langle d,d' \rangle=(d,d')-(d',d)$ denotes the antisymmetrized Euler pairing
 \[ (d,d')=\sum_{i\in Q_0}d_id'_i - \sum_{\alpha:i\to j}d_id'_j \]
 satisfying $(\dim V,\dim V')=\dim \Hom_{\kk Q}(V,V')-\dim \Ext^1_{\kk Q}(V,V')$ for all $\kk Q$-representations $V,V'$. 
\end{definition}

\section{From constructible functions to motivic theories}

\subsection{Constructible functions}

Let us start by recalling some facts about constructible functions. A constructible function is a function $a:X(\CC) \to \ZZ$ on the set of (closed) points of  a  scheme/variety/manifold $X$ over $\CC$ with only finitely many values on each connected component of $X$ and such that the level sets of $a$ are the (closed) points of locally closed subsets of $X$.  We denote with $\Con(X)$ the group of constructible functions on $X$.
\begin{exercise}
 Assume that $X$ is connected. Show that the map associating to every irreducible closed subset $V$ of $X$ its characteristic function extends to an isomorphism $\oplus_{x\in X}\ZZ x \cong \oplus_{V\subset X} \ZZ V\xrightarrow{\sim} \Con(X)$, where the first sum is over all not necessarily closed points $x\in X$, and the second sum is taken over all irreducible closed subsets $V\subset X$. 
\end{exercise}
Apparently, we can pull back constructible functions and also multiply them pointwise. Contrary to the usual notation, we denote the pointwise product with $a\cap b$, i.e.\ $(a\cap b)(x)=a(x)b(x)$. The constant function $\unit_X(x)=1$ for all $x\in X(\CC)$ is the unit for the $\cap$-product. There is another product, the external product $a\boxtimes b =\pr_X^\ast(a)\cap\pr_Y^\ast(b)$ of two functions $a\in \Con(X)$ and $b\in \Con(Y)$ on $X\times Y$ such that $a\cap b=\Delta_X^\ast (a\boxtimes b)$ if $Y=X$. The unit for the $\boxtimes$-product is $1\in \ZZ=\Con(\Spec\CC)$. Moreover, we can define a push-forward of a constructible function 
$a\in \Con(X)$ along a morphism $u:X\to Y$ of finite type by\footnote{For every scheme $X$ locally  of finite type over $\CC$, we denote with $X^{an}$ the ``analytification'' of $X$ which is an analytic space locally isomorphic to the vanishing locus of holomorphic functions on $\CC^n$. If $X$ is smooth, $X^{an}$ is a complex manifold. In any case $X^{an}$ carries the analytic topology which is much finer than the Zariski topology on $X$.} $u_!(a)(y):=\int_{u^{-1}(y)^{an}} a\, d\chi_c:= \sum_{m\in \ZZ}m\chi_c\{x\in X\mid u(x)=y, a(x)=m\}^{an}$ for $y\in Y$. Here $\chi_c$ denotes the Euler characteristic with compact support, i.e.\ the alternating sum of the dimensions of the compactly supported cohomology. One can think of $\chi_c$ as a signed measure $\chi_c^X$ on $X$, even though it is only additive and not $\sigma$-additive. Given a constructible function $a$ on $X$, we get a new measure $a\cdot\chi_c^X$ on $X$ of density $a$ with respect to $\chi_c^X$. A push-forward of a measure is well-defined and $u_!(a\chi_c^X)$ has density $u_!(a)$ with respect to $\chi_c^Y$. Using the push-forward, one can define a third product for constructible functions on a monoidal scheme, i.e.\ a scheme $X$ with two maps $0:\Spec\kk\to X$ and $+:X\times X\to X$ of finite type satisfying an associativity and unit law. The convolution product is given by $a b=+_!(a\boxtimes b)$ and is commutative if $+$ is commutative. The unit is given by $0_!(1)$ with $1\in \Con(\Spec \kk)= \ZZ$ being  the unit for the $\boxtimes$-product. If we had taken $X=\AAA$, the convolution product is just the ``constructible version'' of the usual convolution product of integrable functions. The free commutative monoid generated by a scheme $X$ is given by $\Sym(X) = \sqcup_{n\in \NN} \Sym^n(X)$ with $\Sym^n(X)=X^n/\!\!/S_n$, and $\oplus:\Sym(X) \times \Sym(X)\longrightarrow \Sym(X)$ is just the concatenation of unordered tuples of (geometric) points of $X$. The unit $0:\Spec\CC=:\Sym^0(X) \hookrightarrow \Sym(X)$ is given by the ``empty tuple''. We can apply the definition of the convolution product $ab:=\oplus_!(a\boxtimes b)$ to $\Con(\Sym(X))$ making it into a commutative ring. This (convolution) ring has even more structure. Indeed, there is a family of maps $\sigma^n:\Con(X) \to \Con(\Sym^n(X))$ mapping the characteristic function of $V\subseteq X$ to the characteristic function of $\Sym^n(V)\subseteq \Sym^n(X)$.   

\begin{example}\rm
Consider the example $X=\Spec \kk$. Then $\Sym(X)\cong\NN$, and $\Con(X)\cong \ZZ[[t]]$ follows. The convolution product is just the ordinary product of power series and $\Sym^n(at)={ a+n-1 \choose n}t^n$. The pointwise $\cap$-product is known as the Hadamard product of two power series.   
\end{example}
 Let us collect the main properties of the structures described above.
 
\begin{proposition} \label{motivic_theories}
By taking pull-backs and push-forwards of constructible functions, we obtain a functor $\Con$ from the category $\Sch_\kk$ to the category of abelian groups which is both contravariant with respect to all morphisms and covariant with respect to all morphisms of finite type, i.e.\ for every morphism $u:X\to Y$
there is a group homomorphism $u^\ast:\Con(Y)\longrightarrow \Con(X)$, and if $u$ is of finite type, there is also a group homomorphism $u_!:\Con(X)\longrightarrow \Con(Y)$. Moreover, there is  an ``exterior'' product
\[ \boxtimes: \Con(X)\otimes \Con(Y) \longrightarrow \Con(X\boxtimes Y) \]
defined for every pair $X,Y\in \Sch_\kk$ which is associative, symmetric\footnote{If $\tau:X\boxtimes Y \stackrel{\sim}{\to} Y\boxtimes X$ is the transposition, being symmetric means $\tau_!(a\boxtimes b )=\tau^\ast(a\boxtimes b) = b\boxtimes a$} and has a unit $1\in \Con(\Spec \kk )$. Finally, there are also  operations 
\[ \sigma^n: \Con(X) \longrightarrow \Con(\Sym^n( X)) \]
for $n\in \NN$ such that $\sigma^n(1)=1$ holds for all $n\in \NN$. Additionally, we have the following properties.
\begin{enumerate}
\item[(i)] Considered as a functor from $\Sch_\kk^{op}$ to abelian groups, $\Con$ commutes with all (not necessarily finite) products, i.e.\ the morphism
\[ \Con(X) \longrightarrow \prod_{X_i\in\pi_0(X)} \Con(X_i) \]
given by restriction to connected components is an isomorphism for all $X\in \Sch_\kk$.
\item[(ii)] ``Base change'' holds, i.e.\ for every cartesian diagram 
\[ \xymatrix { X\times_Z Y \ar[r]^{\tilde{v}} \ar[d]_{\tilde{u}}&  X \ar[d]^u  \\  Y \ar[r]_v   & Z  } \]
with $u$ and, therefore, also $\tilde{u}$ of finite type, we have $\tilde{u}_!\circ \tilde{v}^\ast = v^\ast \circ u_!$.
\item[(iii)] The functor $\Con$ commutes with exterior products and $\sigma^n$, i.e.\ \\ $(u\boxtimes v)^\ast(a\boxtimes b)=u^\ast(a)\boxtimes v^\ast(b)$ 
for all $u:X\to X', v:Y\to Y', a\in \Con(X'), b\in \Con(Y')$. If $u, v$ are of finite type, then $(u\boxtimes v)_!(a\boxtimes b)=u_!(a)\boxtimes v_!(b)$ and $\Sym^n(u)_!(\sigma^n(a))=\sigma^n(u_!(a))$ for all $a\in \Con(X), b\in \Con(Y), n\in \NN$. 
\item[(iv)] Using the convolution product $ab=\oplus_!(a\boxtimes b)$ and thinking of $\Con(\Sym^n(X))$ as being a subgroup of $\Con(\Sym(X))$ by means of $\big(\Sym^n(X)\hookrightarrow\Sym(X)\big)_!$, we have 
\[ \sigma^n(a+b)=\sum_{l=0}^n \sigma^l(a)\sigma^{n-l}(b) \]
with $\sigma^1(a)=a$ and $\sigma^0(a)=1\in \Con(\Spec\kk)\hookrightarrow \Con(\Sym(X))$ for all $a,b\in \Con(X)$.
\item[(v)] The ``motivic property'' holds, i.e.\ for every $X$ and every closed subscheme $Z\subseteq X$ giving rise to inclusions $i:Z \hookrightarrow X$ and $j:X\!\setminus\! Z \hookrightarrow X$, we have $i^\ast i_!=\id_{\Con(Z)}, j^\ast j_!=\id_{\Con(X\!\setminus\! Z)}, j^\ast i_!=i^\ast j_!=0$ and 
\[ a=i_!i^\ast(a) + j_!j^\ast(a) \quad\forall a\in \Con(X).\]
\item[(vi)] The equation $\sigma^n(\unit_X)=\unit_{\Sym^n(X)}$ holds for all $X$ and all $n\in \NN$ with $\unit_X=(X\to \Spec\CC)^\ast(1)$ and similarly for $\unit_{\Sym^n(X)}$.
\end{enumerate}
\end{proposition}
\begin{exercise}
 Show that the projection formula $u_!(a\cap u^\ast(b))=u_!(a)\cap b$ holds for all $u:X\to Y$ of finite type and all $a\in R(X), b\in R(Y)$ by using the properties mentioned in Proposition \ref{motivic_theories} and $a\cap b=\Delta_X^\ast(a\boxtimes b)$. Hint: Consider the diagram
 \[ \xymatrix { X \ar[d]_u \ar[r]^{\Delta_X} & X\times X \ar[r]^{\id_X \times u} & X\times Y \ar[d]^{u\times \id_Y}   \\ Y \ar[rr]^{\Delta_Y} & & Y\times Y.} \]
\end{exercise}

\subsection{Motivic theories for schemes}

Generalizing constructible functions, we define a motivic theory\footnote{Motivic theories are special cases of reduced motivic $\lambda$-ring $(\Sch,ft)$-theories defined in \cite{DavisonMeinhardt3}. Every reduced motivic $\lambda$-ring $(\Sch,ft)$-theory is a motivic theory in our sense if $\sigma^n(\unit_X)=\unit_{\Sym^n(X)}$ holds for all $X$ and all $n\in \NN$. (cf.\ Proposition \ref{motivic_theories}.(vi))}   to be a rule associating to every scheme $X$ an abelian group $R(X)$, like $\Con(X)$, along with pull-backs $u^\ast:R(Y)\to R(X)$ for all morphisms $u:X\to Y$ and push-forwards $u_!:R(X)\to R(Y)$ if $u$ is of finite type. Moreover, there should be some associative, symmetric exterior product $\boxtimes:R(X)\times R(Y) \to R(X\times Y)$ with unit element $1\in R(\Spec\kk)$, and some operations $\sigma^n:R(X) \to R(\Sym^n(X))$ for all $n\in \NN$, satisfying exactly the same properties as $\Con(X)$ given in Proposition \ref{motivic_theories}. Similar to the case $\Con$, we can construct a $\cap$-product $a\cap b=\Delta_X^\ast(a\boxtimes b)$ with unit $\unit_X=(X\to \Spec\CC)^\ast(1)$ and a convolution product $ab=+_!(a\boxtimes b)$ with unit $1_0=0_!(1)$ if, additionally, $(X,+,0)$ is a (commutative) monoid with $+$ being of finite type. Note that all these products coincide on $R(\Spec\kk)$. 

\begin{exercise} \label{Euler_characteristics}
Given a motivic theory $R$ and a scheme $X$ with morphism $c:X\to \Spec\kk$, we define $[X]_R:=c_!c^\ast(1)\in R(\Spec\kk)$.
 \begin{enumerate} 
 \item[(i)]  Show that $[X]_R=[Z]_R+[X\!\setminus\! Z]_R$ (cut and paste relation) for every closed subscheme $Z\subseteq X$ and $[X\times Y]_R=[X]_R[Y]_R$  by applying the defining properties of Proposition \ref{motivic_theories}. In particular, $X\mapsto [X]_R\in R(\Spec\kk)$ is a generalization of the classical Euler characteristic $\chi_c:\Sch_\kk \to \ZZ=\Con(\Spec\kk)$. 
  \item[(ii)] Use {\rm (i)} to show $[\PP^n]_R=\LL_R^n+\ldots+\LL_R+1=(\LL^{n+1}_R-1)/(\LL_R-1)$ with $\LL_R:=[\AAA]_R.$
 \end{enumerate}
\end{exercise}
\begin{exercise} \label{fiber_bundle} We use the notation introduced in the previous exercise.
\begin{enumerate}
 \item[(i)] Assume $Y\to X$ is a Zariski locally trivial fiber bundle with fiber $F$. Use the cut and paste relation to prove $[Y]_R=[F]_R[X]_R$.
 \item[(ii)] Use {\rm (i)} applied to the  projection onto the first column and induction over $n\in \NN$ to show $[\Gl(n)]_R=\prod_{i=0}^{n-1} (\LL_R^n-\LL_R^i)$.
 \item[(iii)] Use {\rm (i)} to prove $[\Gr(k,n)]_R=[\Gl(n)]_R/[\Gl(k)]_R[\Gl(n-k)]_R$.
\end{enumerate} 
\end{exercise}

A morphism $\eta:R\to R'$ between motivic theories is a collection of group homomorphisms $\eta_X:R(X)\to R'(X)$ commuting with pull-backs, push-forwards and exterior products. It is called a $\lambda$-morphism, if it additionally commutes with the $\sigma^n$-operations. Thus, we obtain a category of motivic theories containing the subcategory of motivic theories with $\lambda$-morphisms.\\
The rule $X\mapsto R(X)=0$ is the terminal object in the category of motivic theories. Moreover, the following holds.
\begin{lemma} \label{initial_object}
 The category of motivic theories has an initial object given by the (completed) relative Grothendieck group $\underline{\Ka}_0(\Sch):X\mapsto \underline{\Ka}_0(\Sch_X)$ as constructed below. The unique morphisms starting at $\underline{\Ka}_0(\Sch)$ are even $\lambda$-morphisms. 
\end{lemma}
Instead of giving an ad hoc definition of $\underline{\Ka}_0(\Sch)$, let us motivate the construction by looking at constructible functions. Starting with the constant function $1$ on $\Spec\kk$, we take the pull-back $c^\ast(1)=:\unit_X$ for the constant map $c:X \to \Spec\kk$ which is the constant function with value $1$ on $X$, but more importantly, it is the unit object for the $\cap$-product. For any morphism $v:V\to X$ of finite type, consider the function $v_!(\unit_V)\in \Con(X)$ and denote it by\footnote{We already introduced the shorthand $[V]_{\Con}$ for $[V\to \Spec\kk]_{\Con}$ in Exercise \ref{Euler_characteristics}.} $[V\xrightarrow{v} X]_{\Con}$. If $v:V\hookrightarrow X$ is the embedding of a locally closed subscheme, $[V\hookrightarrow X]_{\Con}$ is just the characteristic function of $V$. \\
Using Proposition \ref{motivic_theories}, we get
\begin{eqnarray}
&& \label{eq1} [V\to X]_{\Con}=[Z\to X]_{\Con} + [V\!\setminus\! Z \to X]_{\Con} \mbox{ for all closed }Z\subset V, \\
&& \label{eq2} 1=[\Spec\kk\xrightarrow{\id} \Spec \kk]_{\Con}, \\
&& u^\ast([W\to Y]_{\Con})=[X\times_Y W \to X]_{\Con} \mbox{ for all }u:X\to Y, \\
&& u_!([V\xrightarrow{v} X]_{\Con})=[V\xrightarrow{u\circ v} Y]_{\Con} \mbox{ if }u:X\to Y\mbox{ is of finite type}, \\
&& \label{eq5} [V\xrightarrow{v}X]_{\Con}\boxtimes[W\xrightarrow{w}Y]_{\Con}= [V\times W \xrightarrow{v\times w} X\times Y]_{\Con}, \\
&& \label{eq6} \sigma^n([V\to X]_{\Con})= [\Sym^n(V) \to \Sym^n(X)]_{\Con}.\\
&& \label{eq8} [V\xrightarrow{v}X]_{\Con}=[V'\xrightarrow{v'}X]_{\Con} \mbox{ if there is an isomorphism } \\
&& \nonumber v'':V\to V' \mbox { such that }v=v'\circ v'', 
\end{eqnarray}
Obviously, the same must hold in every motivic category as they share the same properties, and so the same applies to the initial object if it exists. Moreover, for connected $X$ the group $\Con(X)$ is generated by all classes $[V\to X]_{\Con}$ satisfying relation (\ref{eq1}). The same must be true for the initial motivic theory since otherwise the subgroup spanned by the elements $[V\to X]_{init}$ for connected $X$ and extended by Property \ref{motivic_theories}(i) for non-connected $X$ is a proper subtheory of the initial theory which will lead to a contradiction. However, there are more relations in $\Con(X)$ as for example $[\GG_m\xrightarrow{z^d}\GG_m]=d[\GG_m\xrightarrow{\id} \GG_m]$ which might not hold in other motivic theories as for example the initial one. Dropping the subscript ``init'' we will, therefore, define our (hopefully) initial theory by associating to every connected scheme $X$ the group $\Ka_0(\Sch_X)$ generated by symbols $[V\xrightarrow{v}X]$ for every isomorphism class (due to equation (\ref{eq8})) of morphisms $v:V\to X$ of finite type, subject to the relation (\ref{eq1}). For non-connected $X$ we simply put 
\[ \underline{\Ka}_0(\Sch_X):=\prod_{X_i\in \pi_0(X)} \Ka_0(\Sch_{X_i}). \]
To obtain a motivic theory $\underline{\Ka}_0(\Sch)$, we must define $1\in\Ka_0(\Sch_\kk), u^\ast, u_!,\boxtimes$ and $\sigma^n$ as in equations (\ref{eq2})--(\ref{eq6}), at least over connected components.  It has been shown in \cite{GLMH1} that $\sigma^n$-operations satisfying these properties do indeed exist. Moreover, the authors prove that $\sigma^n(a\LL)=\sigma^n(a)\LL^n$ holds for every $a\in \underline{\Ka}_0(X)$ and every $n\in \NN$, were $\LL=c_!c^\ast(1)=[\AA^1\xrightarrow{c} \Spec\CC]\in \Ka_0(\Sch_\kk)$ is considered as an element of $\underline{\Ka}_0(\Sch_{\Sym(X)})$ via the embedding $0_!:\Ka_0(\Sch_\kk)\hookrightarrow \underline{\Ka}_0(\Sch_{\Sym(X)})$.  
\begin{exercise} Use the properties of a morphism between motivic theories to show that $[V\to X] \mapsto [V\to X]_R$ defines a homomorphism $\eta_X:\Ka_0(\Sch_X)\to R(X)$ for connected $X$ which extends to a morphism $\eta:\underline{\Ka}_0(\Sch)\to R$ of motivic theories. Prove that this morphism is the only possible one. Hence, $\underline{\Ka}_0(\Sch)$ is the initial object in the category of motivic theories. Moreover, show that $\eta$ is a $\lambda$-morphism. 
\end{exercise}
The initial property of $\underline{\Ka}_0(\Sch)$ is just a generalization of the well-known property that $X\to [X]\in \Ka_0(\Sch_\kk)$ is the universal Euler characteristics.\\ 
Let $R^{gm}(X)\subset R(X)$ be the subgroup generated by all elements $[V\to X]_R$ if $X$ is connected and $R^{gm}(X):=\prod_{X_i\in \pi_0(X)}R^{gm}(X_i)$ for general $X$. One should think of elements in $R^{gm}(X)$ as ``geometric'' since they are $\ZZ$-linear combinations of elements  obtained by geometric constructions, namely pull-backs and push-forwards of the unit $1\in R(\Spec\CC)$.
\begin{exercise} \label{geometric_part} Show that $X\mapsto R^{gm}(X)$ defines a subtheory of $R$. By construction, it is the image of the $\lambda$-morphism $\eta:\underline{\Ka}_0(\Sch)\to R$ obtained in the previous exercise. Show that $\Con^{gm}=\Con$ and $\underline{\Ka}_0(\Sch)^{gm}=\underline{\Ka}_0(\Sch)$. Prove that $\sigma^n(a\LL_R)=\sigma^n(a)\LL_R^n$ holds for every $a\in R^{gm}(X)$ and every $n\in \NN$.
\end{exercise}

\subsection{Motivic theories for quotient stacks}

In the previous section we generalized constructible functions and the classical Euler characteristic to more refined ``functions'' and invariants. When it comes to moduli problems, we should also be able to compute refined invariants of quotient stacks as they occur naturally in moduli problems. Hence, we need to extend motivic theories to disjoint unions of quotient stacks $X/G$ for schemes $X$ locally of finite type over $\CC$ and linear algebraic groups $G$. This is the topic of this subsection.
\begin{exercise} \label{special}
Given a closed embedding $G\hookrightarrow \Gl(n)$ of a linear algebraic group $G$ and a $G$-action on a scheme $X$. Show that the morphism $X/G \to (X\times_G \Gl(n))/\Gl(n)$ induced by $X\ni x\mapsto (x,1)\in X\times_G\Gl(n)$ and $G\hookrightarrow \Gl(n)$ as in Exercise \ref{stack_morphisms} is in fact an isomorphism of quotient stacks. Hint: Given a principal $\Gl(n)$-bundle $P\to S$ and a $\Gl(n)$-equivariant morphism $\psi:P\to X\times_G\Gl(n)$, show that  $\psi^{-1}(X)\to S$ for $X\hookrightarrow X\times_G\Gl(n)$ is a principal $G$-bundle over $S$. To prove local triviality of $\psi^{-1}(X)\to S$ one has to construct  local sections of $\psi^{-1}(X)\to S$. For this, one can take a local section $\nu:U\to P$ of $P\to S$ and a lift $(f,g):\tilde{U} \to X\times \Gl(n)$ of $\psi\circ \nu:U\to X\times_G \Gl(n)$ on a possibly smaller \'{e}tale neighborhood $s\in \tilde{U}\subset U$ of $s\in S$. Then $\tilde{\nu} :\tilde{U}\ni t \longmapsto\to \nu(t)g(t)^{-1}\in \psi^{-1}(X)$ is a local section of $\psi^{-1}(X)\to S$. Note that if $G$ is special, these \'{e}tale neighborhoods $U$ and $\tilde{U}$ can even be replaced with Zariski neighborhoods.   
\end{exercise}

\begin{definition} A stacky motivic theory $R$ is a  rule associating to every disjoint union $\XXX=\sqcup_{i\in I}X_i/G_i$ of quotient stacks with linear algebraic groups $G_i$ an abelian group $R(\XXX)$ along with pull-backs $u^\ast:R(\YYY)\to R(\YYY)$ for all (1-)morphisms $u:\XXX \to \YYY$ and push-forwards $u_!:R(\XXX)\to R(\YYY)$ if $u$ is of finite type. Moreover, there should be some associative, symmetric exterior product $\boxtimes:R(\XXX)\times R(\YYY) \to R(\XXX\times \YYY)$ with unit element $1\in R(\Spec\kk)$, and some operations $\sigma^n:R(X) \to R(\Sym^n(X))$ for all $n\in \NN$ and all schemes $X$, satisfying the stacky analogue of the  properties of $\Con(-)$ given in Proposition \ref{motivic_theories}.
\end{definition}
\begin{remark} \rm
There are two technical difficulties to overcome when we try to generalize Proposition \ref{motivic_theories}, which serves as our definition of a (stacky) motivic theory,  to disjoint unions of quotient stacks. First of all, we need to explain what the correct generalization of a finite type morphism ought to be. For us, this is a (1-)morphism $u:\XXX\to \YYY$ of algebraic stacks such the preimage of each ``connected component'' $Y/H$ (with connected $Y$) consists of only finitely many connected components $X_i/G_i$ of $\XXX$. Secondly, we need to define $\Sym^n(X/G)$ for quotient stacks. There is an obvious candidate given by the quotient stack $X^n/(S_n\ltimes G^n)$. However, if $G=\{1\}$ is the trivial group, we get the quotient stack $X^n/S_n$ which is different from its coarse ``moduli space'' $\Sym^n(X)=X^n/\!\!/S_n$. To avoid these problems, we only require the existence of $\sigma^n$-operations for schemes $\XXX=X$ and not for general disjoint unions of quotient stacks.     
\end{remark}

\begin{example}\rm 
There is no stacky motivic theory $R$ with $R|_{\Sch_\kk}=\Con$ such that the pull-back
$\rho^\ast:R(X/G) \to R(X)=\Con(X)$ is an embedding.  Indeed, consider the case $X=\Spec\kk$ and $G=\GG_m$. Then $X \xrightarrow{\rho} X/G \to \Spec\kk$ is the identity, and $\rho_!(\unit_X)$ cannot be zero. By assumption, $\rho^\ast\rho_!(\unit_X)$ is also nonzero. However, applying base change to the diagram 
\[ \xymatrix { X\times G \ar[r]^m \ar[d]_{\pr_X} & X \ar[d]^\rho \\ X \ar[r]_\rho & X/G }\]
with $m:X\times G \to X$ denoting the (trivial) group action, $\rho^\ast\rho_!(\unit_X)=\pr_{X\, !}(\unit_{X\times G})=\chi_c(G)\unit_X =0$, a contradiction.  
\end{example}
Applying the functoriality of the pull-back to the previous diagram, we obtain $\pr_X^{\ast}(b)=m^\ast(b)$ for $b=\rho^\ast(a)$. In other words, for every stacky motivic theory $R$, the image of $\rho^\ast$ is contained in the subgroup $R(X)^G:=\{a\in R(X)\mid \pr_X^{\ast}(b)=m^\ast(b)\}$ of ``$G$-invariant'' elements.
Despite the negative result given by the previous example, we will provide a functorial construction which associates to every motivic theory $R$ satisfying 
\begin{equation} \label{eq7}
 \sigma^n(a\LL_R)=\sigma^n(a)\LL_R^n \quad\forall \;a\in R(X)
\end{equation}
another motivic theory $R^{st}$ such that $R^{st}$ extends to a stacks motivic theory, also denoted with $R^{st}$, for which   $\rho^\ast:R^{st}(X/G)\to R^{st}(X)^G$ is an isomorphism. Moreover, there is a morphism $R\to R^{st}|_{\Sch_\kk}$ of motivic theories satisfying the property that every morphism $R\to R'|_{\Sch_\kk}$ with $R'$ being a stacky motivic theory satisfying $\rho^\ast:R'(X/G) \xrightarrow{\sim} R'(X)^G$  must factorize though $R\to R^{st}$. In particular, the restriction functor from the category of stacky motivic theories $R'$ satisfying (\rm \ref{eq7}) and  $\rho^\ast:R'(X/G) \xrightarrow{\sim} R'(X)^G$ has a left adjoint given by $R\to R^{st}$. As we will see, $\Con^{st}=0$. 

Recall that a linear algebraic group $G$ was called special if every \'{e}tale locally trivial principal $G$-bundle is already Zariski locally trivial. In particular, given a closed embedding $G\hookrightarrow \Gl(n)$, the map $\Gl(n)\to \Gl(n)/G$ must be a Zariski locally trivial principal $G$-bundle. On can show that this property is already sufficient for being special. Hence, $\Gl(n)$ is special for every $n\in \NN$. As a result of Exercise \ref{fiber_bundle} we get $[\Gl(n)]_R=[G]_R[\Gl(n)/G]_R$ in $R(\Spec\kk)$ for every motivic theory $R$. In particular, $[G]_R$ is invertible for every special group $G$ if and only if $[\Gl(n)]_R$ is invertible for every $n\in \NN$.
\begin{definition}
 Given a group $G$, a (1-)morphism $u:\PPP\to \XXX$ of stacks is called a principal $G$-bundle on $\XXX$ if $u$ is representable and the pull-back $\tilde{u}:X\times_\XXX \PPP \longrightarrow X$ of  $u$ along every morphism $X\to \XXX$ with $X$ being a scheme is a principal $G$-bundle on $X$.  
\end{definition}

\begin{exercise}
Given a stacky motivic theory $R$. We want to show in several steps that the condition $\rho^\ast: R(X/G)\to R(X)^G$ being an isomorphism for every special group $G$ is equivalent to the condition that $[\PPP\xrightarrow{u} \XXX]_R:=u_!(\unit_\PPP)=[G]\unit_\XXX$ for every special group $G$ and every principal $G$-bundle $u:\PPP\to \XXX$ in the category of disjoint unions of quotient stacks. 
\begin{enumerate}
 \item Given a principal $G$-bundle $\PPP\to \XXX$ and assume for simplicity $\XXX=X/\Gl(n)$. Consider the cartesian diagram 
 \[ \xymatrix {P \ar[r]^{\tilde{\rho}} \ar[d]^{\tilde{u}}  & \PPP \ar[d]^u \\ X \ar[r]^\rho & \XXX.} \]
 By assumption, $\tilde{u}:P\to X$ is a principal $G$-bundle. Use injectivity of $\rho^\ast$ to prove $u_!(\unit_\PPP)=[G]_R\unit_\XXX$ (Hint: Use base change and Exercise \ref{fiber_bundle}.) Extend this result to arbitrary disjoint unions of quotient stacks. 
 \item Conversely, assume  that $u_!(\unit_\PPP)=[G]_R\unit_\XXX$ for every principal $G$-bundle in the category of quotient stacks. 
  Show first that $[G]_R$ is invertible in $R(\Spec\kk)$ with inverse $[\Spec\kk/G]_R$. (Hint: Consider the principal $G$-bundle $\Spec\kk \to \Spec\kk/G$.)
 \item Secondly, prove that $\rho^\ast:R(X/G)\to R(X)^G$ is invertible by showing that $\rho_!(-)/[G]_R$ is an inverse. (see \cite{DavisonMeinhardt3}, Lemma 5.13 if you need help)
\end{enumerate}

\end{exercise}

For connected $X$ we define $R^{st}(X):=R(X)[ [\Gl(n)]^{-1}_R \mid n\in \NN]$ using the $R(\Spec\kk)$-module structure of $R(X)$ by means of the $\boxtimes$-product. We extend it via $R^{st}(X)=\prod_{X_i\in \pi_0(X)}R^{st}(X_i)$ to non-connected $X$. The morphism $R(X)\to R^{st}(X)$ is obvious, and it is also easy to see how to extend $u_!, u^\ast$ for $u:X\to Y$ and the $\boxtimes$-product. The only nontrivial part is the extension of $\sigma^n$. For $X\in \Sch_\kk$ and $a\in R(X)$ define
\[ \sigma_t(a):=\sum_{n\in \NN} \sigma^n(a)t^n \in R(\Sym(X)[[t]] \]
The Adams operations $\psi^n:R(X)\to R(\Sym^n(X))\subseteq R(\Sym(X))$ are defined by means of the series
\[ \psi_t(a):=\sum_{n\ge 1} \psi^n(a)t^{n} := \frac{d\log \sigma_t(a)}{d\log t}=t\sigma_t(-a)\frac{d\sigma_t(a)}{dt},\]
where the product is the convolution product in $R(\Sym(X))$. Using the properties of $\sigma^n$, we observe $\sigma_t(0)=1$ and $\sigma_t(a+b)=\sigma_t(a)\sigma_t(b)$. Thus, $\psi_t(0)=0$ as well as $\psi_t(a+b)=\psi_t(a)+\psi_t(b)$ follows. Property {\rm (\ref{eq7}) } implies $\psi^n(a P(\LL_R))=\psi^n(a)P(\LL_R^n)$ for every polynomial $P(x)\in \ZZ[x]$. Due to Exercise \ref{fiber_bundle}(ii), we have to extend $\psi^n$ to $R^{st}(X)$ by means of 
\[ \psi^n\left( \frac{a}{\prod_{i\in I} [\Gl(m_i)]_R }\right):= \frac{\psi^n(a)}{\prod_{i\in I} P_{m_i}(\LL^n_R)} \]
using the polynomial $P_m(x)=\prod_{j=0}^{m-1}(x^m-x^j)$ satisfying $[\Gl(m)]_R=P_m(\LL_R)$. Having extended the Adams operations, we can also extend the $\sigma^n$-operations by putting
\[ \sigma_t(a)=\exp\left(\int \psi_t(a)\frac{dt}{t}\right). \]
Note that the last expression involves rational coefficient, but one can show that the rational coefficients disappear in the expression for 
\[ \sigma_t\left(\frac{a}{\prod_{i\in I}[\Gl(m_i)]_R}\right)=\exp\left(\sum_{n\ge 1} \frac{\psi^n(a)t^n}{n\prod_{i\in I} P_{m_i}(\LL_R^n)} \right) \]
if we express $\psi^n(a)$ in terms of $\sigma^m(a)$ for $1\le m\le n$. (See \cite{DavisonMeinhardt3}, Appendix B for more details.) 
\begin{exercise}
 Show that $\Con^{st}(X)=0$ for all $X$.
\end{exercise}

Now, as we have constructed $R^{st}$ on schemes, we will put $R^{st}(\XXX):=R^{st}(X)^G$ for a  quotient stack $\XXX=X/G$ with special group $G$ and connected $X$. We have to show that this definition is independent of the presentation of the quotient stack. For this let $\XXX\cong Y/H$ be another presentation with a special group $H$. Let us form the cartesian square
\[ \xymatrix @C=1.5cm @R=1.5cm { X\times_\XXX Y\times G\times H  \ar@/_1pc/[d] \ar@/^1pc/[d] \ar@/^1pc/[r] \ar@/_1pc/[r] &  X\times_\XXX Y \times H  \ar@/^1pc/[d] \ar@/_1pc/[d] \ar[r]^{\rho''} & Y\times H \ar@/^1pc/[d]^{m_Y} \ar@/_1pc/[d]_{\pr_Y} \\ 
X \times_\XXX Y \times G  \ar[d]^{\tau''} \ar@/^1pc/[r] \ar@/_1pc/[r] & X\times_\XXX Y  \ar[d]^{\tau'} \ar[r]^{\rho'} & Y  \ar[d]^\tau \\
X\times G \ar@/^1pc/[r]^{m_X} \ar@/_1pc/[r]_{\pr_X} & X  \ar[r]^\rho & \XXX }
\] 
with  $\rho',\rho''$ and $\tau',\tau''$ being $G$- respectively $H$-principal bundles. The other maps are either projections or actions of $G$ or $H$. Applying $R^{st}$, we get the following diagram with exact rows and columns by construction of $R^{st}$, where $K$ denotes the kernel of say $\pr_Y^\ast-m_Y^\ast$. 
\[ \xymatrix @C=1.3cm { & 0 \ar[d] & 0 \ar[d] & 0 \ar[d] \\ 
 & K \ar[r] \ar[d] &  R^{st}(X) \ar[r]^{\pr_X^\ast-m_X^\ast} \ar[d]^{\tau'^\ast} & R^{st}(X\times G) \ar[d]^{\tau''^\ast} \\
0 \ar[r] & R^{st}(Y) \ar[r]^{\rho'^\ast} \ar[d]^{\pr_Y^\ast-m_Y^\ast} & R^{st}(X\times_\XXX Y) \ar[r] \ar[d] & R^{st}(X\times_\XXX Y \times G) \ar[d] \\
0 \ar[r] & R^{st}(Y\times H) \ar[r]_{\rho''^\ast} & R^{st}(X\times_\XXX Y \times H) \ar[r] & R^{st}(X\times_\XXX Y \times G\times H) } 
\]
Hence, $R^{st}(X)^{G}\cong K\cong R^{st}(Y)^{H}$ showing that $R^{st}(\XXX)$ is independent of the choice of a presentation.  
Given a morphism $u:X/G \to Y/H$ of quotient stacks, we form the cartesian diagram
\[ \xymatrix { X\times_{Y/H} Y \ar[d]_{\tilde{\tau}} \ar[r]^(0.4){\tilde{\rho}} & X/G\times_{Y/H} Y \ar[r]^(0.6){\tilde{u}} \ar[d] & Y \ar[d]^\tau \\
 X \ar[r]^\rho & X/G \ar[r]^u & Y/H. } \]
For $a\in R^{st}(X/G)\cong R^{st}(X)^G$ and $b\in R^{st}(Y)^H$ we put 
\begin{eqnarray*} 
u_!(a)&:=&(\tilde{u}\circ \tilde{\rho})_!\tilde{\tau}^\ast(a)/[G]_R\in R^{st}(Y)^H\mbox{ and }\\
 u^\ast(b)&:=&\tilde{\tau}_!(\tilde{u}\circ\tilde{\rho})^\ast(b)/[H]_R\in R^{st}(X)^G\\
 a\boxtimes b&:=& a\boxtimes b\in R^{st}(X\times Y)^{G\times H}.
\end{eqnarray*}
For disjoint unions $\XXX=\sqcup_{i\in I}X_i/G_i$ of connected quotient stacks, we can always assume that $G_i$ is special for all $i\in I$ due to Exercise \ref{special}. Then, we need to define $R^{st}(\XXX):=\prod_{i\in I}R^{st}(X_i/G_i)$ according to Proposition \ref{motivic_theories}(i), and extend $u_!, u^\ast$ and $\boxtimes$ in the natural way.

Given a morphism $\eta:R\to R'|_{\Sch_\kk}$ with $\rho^\ast:R'(X/G)\to R'(X)^G$ being an isomorphism, we define $\eta_{X/G}:R^{st}(X/G)=R(X)^G \xrightarrow{\eta_X} R'(X)^G\xrightarrow{\rho^{\ast\, -1}} R'(X/G)$ with $\rho^{\ast\, -1}(a)=\rho_!(a)/[G]_R$ which is the only possible choice to extend $\eta$ to a morphism $R^{st}\to R'$ of stacky motivic theories. More details in a more general context are given in \cite{DavisonMeinhardt3}, Section 5.

\section{Vanishing cycles}

The aim of this section is to introduce the notion of a vanishing cycle taking values in a (stacky) motivic theory $R$. We start by considering vanishing cycles of morphisms $f:X\to \AAA$ defined on smooth schemes $X$.

\subsection{Vanishing cycles for schemes}

\begin{definition}
 Given a motivic theory $R$, a vanishing cycle\footnote{The definition given here differs from the one given in \cite{DavisonMeinhardt3} for the sake of simplicity. We do not require the support property but a blow-up formula.} (with values in $R$) is a rule  associating to every regular function $f:X\to \AAA$ on a smooth scheme/variety or complex manifold $X$ an element $\phi_f\in R(X)$ such that the following holds.
 \begin{enumerate}
  \item If $u:Y\to X$ is a smooth, then $\phi_{f\circ u}=f^\ast(\phi_f)$.
  \item Let $X$ be a smooth variety containing a smooth closed subvariety $i:Y\hookrightarrow X$. Denote by $j:E\hookrightarrow \Bl_Y X$ the exceptional divisor of the blow-up $\pi:\Bl_Y X \to X$ of $X$ in $Y$. Then the  formula  \[ \pi_!\big( \phi_{f\circ \pi} - j_! \phi_{f\circ \pi\circ j}\big) =  \phi_f - i_!\phi_{f\circ i}\]     holds for every $f:X\to \AAA$.
  \item Given two morphisms $f:X\to \AAA$ and $g:Y\to \AAA$ on smooth $X$ and $Y$, we introduce the notation $f\boxtimes g:X\times Y\xrightarrow{f\times g} \AAA\times \AAA\xrightarrow{+} \AAA$. Then $\phi_{f\boxtimes g}=\phi_f\boxtimes \phi_g$ in $R(X\times Y)$. Moreover, $\phi_{\Spec\kk\xrightarrow{0}\AAA}(1)=1$.
 \end{enumerate}
 
\end{definition}
\begin{lemma} \label{vanishing_cycle_morphism}
A collection of elements $\phi_f\in R(X)$ for regular functions $f:X\to \AAA$ on smooth schemes $X$ satisfying the properties (1),(2) and (3) is equivalent to a collection of group homomorphisms\footnote{The collection of group homomorphisms $\phi_f$  is what is called a morphisms of (motivic) ring $(Sm,prop)$-theories over $\AAA$ in \cite{DavisonMeinhardt3}.} $\phi_f:\underline{\Ka}_0(\Sch_X)\to R(X)$ for all regular functions $f:X\to \AAA$ on arbitrary schemes $X$ such that the following diagrams commute
\[ \xymatrix @C=1.5cm{ \underline{\Ka}_0(\Sch_X) \ar[r]^{\phi_f} \ar[d]^{u^\ast} & R(X) \ar[d]^{u^\ast} \\ \underline{\Ka}_0(\Sch_Y) \ar[r]^{\phi_{f\circ u}} & R(Y) } \qquad\mbox{if }u:Y\to X\mbox{ is smooth,} \]
\[ \xymatrix @C=1.5cm{ \underline{\Ka}_0(\Sch_Y) \ar[r]^{\phi_{f\circ u}} \ar[d]^{u_!} & R(Y) \ar[d]^{u_!} \\ \underline{\Ka}_0(\Sch_X) \ar[r]^{\phi_{f}} & R(X) } \qquad\mbox{if }u:Y\to X\mbox{ is proper,} \]
\[ \xymatrix @C=2.5cm { \underline{\Ka}_0(\Sch_X)\otimes\underline{\Ka}_0(\Sch_Y) \ar[r]^{\phi_{f}\otimes \phi_g} \ar[d]^{\boxtimes} & R(X)\otimes R(Y) \ar[d]^{\boxtimes} \\ \underline{\Ka}_0(\Sch_{X\times Y}) \ar[r]^{\phi_{f\boxtimes g}} & R(X\times Y) }  \]
and $\phi_{\Spec\CC\xrightarrow{0}\AAA}(1)=1$.
\end{lemma}
\begin{exercise}
 Proof the lemma using the following fact (see \cite{Bittner04}, Thm.\ 5.1). The group $\Ka_0(\Sch_Z)$ can also be written as the abelian group generated by symbols $[X\xrightarrow{p} Z]$ with smooth $X$ and proper $p$ subject  the ``blow-up relation'': If $i:Y\hookrightarrow X$ is a smooth  subvariety and $\pi:\Bl_Y X \to X$ the blow-up of $X$ in $Y$ with exceptional divisor $j:E\hookrightarrow \Bl_Y X$, then $[\Bl_Y X \xrightarrow{p\pi} Z] - [E \xrightarrow{p\pi j} Z] = [X\xrightarrow{p}Z]-[Y\xrightarrow{p i} Z]$. Hint: Given a function $f:Z\to \AAA$, try the Ansatz $\phi_f([X\xrightarrow{p} Z]):=p_!\phi_{f\circ p}$ for a proper morphism $X\xrightarrow{p}Z$ on a smooth scheme $X$, where $\phi_{f\circ p}\in R(X)$ on the right hand side is given by our family of elements. In particular, $\phi_f(\unit_X)=\phi_f\in R(X)$ for a regular function $f$ on a smooth scheme $X$.
\end{exercise}
We need to apologize for using the same symbol $\phi_f$ with two different meanings. However, with a bit of practice it should be clear from the context which interpretation is used.  

\begin{example}\rm
 For every motivic theory there is a canonical vanishing cycle  such that $\phi^R_{can,f}=\unit_{X}\in R(X)$ for every $f:X\to \AAA$. Hence, it does not depend on $f$, and the map $\phi_f:\underline{\Ka}_0(\Sch_X)\to R(X)$ is just the morphism constructed in Lemma \ref{initial_object}.
\end{example}

Let us look at the following more interesting examples.
\begin{example} \rm
 Let $R=\Con$. For $x\in X$ we fix a metric on  a an analytic neighborhood of $x\in X^{an}$ for example by embedding such a neighborhood into $\CC^n$. We form the so-called Milnor fiber $\MF_{f}(x):=f^{-1}(f(x)+\delta)\cap B_\varepsilon(x)$, where $B_\varepsilon(x)$ is a small open ball around $x\in X$ and $0<\delta \ll \varepsilon \ll 1$ are small real parameters. Notice, that the Milnor fiber depends on the choice of the metric and the choice of $\delta,\varepsilon$. However, its reduced cohomology and its Euler characteristic $\chi(\MF_f(x))$ are independent of the choices made. We finally define $\phi^{con}_f(x):=1-\chi(\MF_f(x))$ for sufficiently small $\delta,\varepsilon$. One can show that the properties listed above are satisfied.  Moreover,  $\phi^{con}_f$   agrees with the Behrend function of $\Crit(f)$ up to the sign $(-1)^{\dim X}$.  
\end{example}
\begin{example} \rm
The previous example has a categorification. For a closed point $t\in \AAA(\CC)$ consider the cartesian diagram
\[ \xymatrix @C=1.5cm { X^{an} \ar[d]^{f} & X^{an}\times_\CC \CC \ar[d]^{\tilde{f}} \ar[l]^{q_t} \\ \CC & \ar[l]^{t+\exp(-)} \CC.}\]
Denoting the inclusion of the fiber $X_t=f^{-1}(t)$ into $X$ by $\iota_t$, the (classical) vanishing cycle $\phi^{perv}_f$ is defined via 
\[ \oplus_{t\in \AAA(\CC)} \iota_{t\,!}\Cone(\QQ_{X_t}\longrightarrow \iota_t^\ast q_{t\,\ast} q_t^\ast \QQ_X), \]
in the derived category $D^b(X,\QQ)$ of sheaves of $\QQ$-vector spaces on $X$. Here, $\QQ_X$ respectively $\QQ_{X_t}=\iota_t^\ast \QQ_X$ denotes the locally constant sheaf on $X$ respectively $X_t$ with stalk $\QQ$. The morphism is the restriction to $X_t$ of the adjunction $\QQ_X \longrightarrow q_{t\,\ast} q_t^\ast \QQ_X$.  Spelling out the definition we see that  the stalk of $\phi^{perv}_f$ at $x$ is given by the reduced cohomology of the Milnor fiber $\MF_f(x)$ shifted by $-1$. \\
Associating to every connected $X$ the Grouthendieck group $\Ka_0(D^b_{con}(X,\QQ))=\Ka_0(\Perv(X))$ of the triangulated subcategory $D^b_{con}(X,\QQ)\subset D^b(X,\QQ)$ consisting of complexes of sheaves of $\QQ$-vector spaces with constructible cohomology, we obtain a motivic theory $\underline{\Ka}_0(D^b_{con}(-,\QQ))$ with $\underline{\Ka}_0(D^b_{con}(X,\QQ)):=\prod_{X_i\in \pi_0(X)} \Ka_0(D^b_{con}(X_i,\QQ))$. Since $\phi_f^{perv}$ turns out to be a complex with constructible cohomology, we can take its class  in $\underline{\Ka}_0(D^b_{con}(X,\QQ))$ and get a vanishing cycle satisfying all required properties. 
\end{example}
\begin{example} \rm
The previous example has a refinement $\phi_f^{mhm}\in D^b(\MHM(X)_{mon})$ involving (complexes of) ``monochromatic mixed Hodge modules with monodromy groups of the form $\mu_n$, the group of n-th roots of unity, for some $n\in \NN$. Forgetting the Hodge and the monodromy structure, we get a functor $D^b(\MHM(X)_{mon})\longrightarrow D^b_{con}(X,\QQ)$ mapping $\phi_f^{mhm}$ to $\phi_f^{perv}$. By passing to Grothendieck groups, we get a vanishing cycle $\phi^{mhm}_f$ with values in $\underline{\Ka}_0(D^b(\MHM(X))_{mon})=\underline{\Ka}_0(\MHM(X)_{mon}):=\prod_{X_i\in \pi_0(X)} \Ka_0(\MHM(X_i)_{mon})$. 
\end{example}

In the remaining part of this subsection we will construct vanishing cycles depending functorially on $R$. First of all we need to enlarge $R$ by defining a new motivic theory $R(-\times \AAA)$ mapping $X$ to $R(X\times \AAA)$ and using the exterior product 
\[ R(X \times \AAA)\otimes R(Y\times \AAA) \xrightarrow{\boxtimes} R( X\times Y \times \AA^2) \xrightarrow{(\id\times +)_!} R(X\times Y\times \AAA) \]
with unit $1':=0_!(1)\in R(\Spec\CC\times\AAA)=R(\AAA)$ and the $\sigma^n$-operations
\[ R(X\times \AAA) \xrightarrow{\sigma^n} R(\Sym^n(X\times \AAA))\longrightarrow R(\Sym^n(X)\times \Sym^n(\AAA)) \xrightarrow{(\id\times +)_!} R(\Sym^n(X)\times\AAA). \]
\begin{exercise}
 Check the properties of a motivic theory given in Proposition \ref{motivic_theories}.
\end{exercise}

Given a scheme $X$, let  $\GG_m$ act on $X\times \AAA$ via $g(x,z)=(x,gz)$. For connected $X$ we denote with $R^{gm}_{\GG_m}(X\times \AAA)$ the subgroup of $R(X\times \AAA)$ generated by elements $[Y\xrightarrow{f} X\times \AAA]_R$ such that $Y$ carries a good\footnote{An action of $\GG_m$ on $Y$ is called good if every point $y\in Y$ has an affine $\GG_m$-invariant neighborhood.} $\GG_m$-action for which $f$ is homogeneous of some degree $d> 0$, i.e.\ $f(gy)=g^df(y)$. Notice, that such a $Y$ will carry many actions for which $f$ is homogeneous. Indeed, given $0\not=n\in \NN$, let $\GG_m$ act on $Y$ via $g\star y:=g^n y$ using the old action on the right hand side. Then, $f$ is homogeneous of degree $dn$ with respect to the new action. In particular, given finitely many generators $[Y_i\xrightarrow{f_i} X\times \AAA]$, we can always assume that the degrees of $f_i$ are equal. Finally, we put $R^{gm}_{\GG_m}(X\times \AAA):=\prod_{X_i\in \pi_0(X)} R^{gm}_{\GG_m}(X_i\times \AAA)$ for non-connected $X$. 

\begin{lemma} \label{monodromy_extension}
 The subgroup $R^{gm}_{\GG_m}(X\times \AAA)\subset R(X\times \AAA)$ is invariant under pull-backs, push-forwards, exterior products and the $\sigma^n$-operations. Moreover, $\pr_X^\ast$ maps $R^{gm}(X)$  onto a ``$\lambda$-ideal'' $I^{gm}_X\subseteq R^{gm}_{\GG_m}(X\times \AAA)$, i.e.\ $a\boxtimes b \in I^{gm}_{X\times Y}$ for $a\in I^{gm}_X, b\in R^{gm}_{\GG_m}(Y\times \AAA)$ and $\sigma^n(a)\in I^{gm}_{\Sym^n(X)}$ for $a\in I^{gm}_X$. 
\end{lemma}
\begin{exercise}
 Check the first sentence of the previous lemma.
\end{exercise}
\begin{proof}
 To show that $I^{gm}_X$ is a $\lambda$-ideal, it suffices to look at generators $[V\times\AAA\xrightarrow{f\times \id_{\AAA}} X\times\AAA]$ and $[W \xrightarrow{(g,h)} Y\times \AAA]$  of $I^{gm}_X$ and $R^{gm}_{\GG_m}(Y\times\AAA)$  respectively with $[V\xrightarrow{f} X]\in R^{gm}(X)$. (cf.\ Exercise \ref{generators}) For the $\boxtimes$-product 
\[ [V\times \AAA\to X\times \AAA] \boxtimes [W \to Y\times \AAA] = [V\times\AAA\times W\xrightarrow{u} X\times Y\times \AAA] \]
with $u(v,z,w)=(f(v),g(w),z+h(w))$ we use the isomorphism 
\[ V\times  \AAA \times W \ni(v,z,w)\longmapsto (v,w,z+h(w))\in V\times W\times \AAA\] 
to show $[V\times\AAA\times W\xrightarrow{u} X\times Y\times \AAA]=\pr^\ast_{X\times Y}([V\times W\xrightarrow{f\times g} X\times Y])\in I^{gm}_{X\times Y}$. 
 We also have
\[ \sigma^n[V\times \AAA^1 \xrightarrow{f\times\id_{\AAA}} X\times \AAA]=[ \Sym^n(V\times \AAA) \xrightarrow{\tilde{p}} \Sym^n(X)\times \AAA ] \] 
with $\tilde{p}$ being induced by the $S_n$-invariant morphism $p:(V\times \AAA)^n\to \Sym^n(X)\times \AAA$ with $p\big((v_1,z_1),\ldots,(v_n,z_n)\big)=\big((f(v_1),\ldots,f(v_n)),z_1+\ldots + z_n\big)$. We define 
\[ (V^n\times \AA^n)^0:=\{(v_1,\ldots,v_n,z_1,\ldots,z_n) \in V^n\times \AA^n \mid z_1 + \ldots z_n=0 \} \]
for all $n>0$. There is a $S_n$-equivariant isomorphism $\psi:(V^n\times \AA^n)^0\times \AAA \longrightarrow (V \times \AAA)^n$ sending $((v_1,\ldots,v_n,z_1,\ldots,z_n),z)$ to $\big((v_1,z_1+z/n),\ldots,(v_n, z_n + z/n)\big)$. Then, $p\circ \psi=q\times \id_\AAA: (V^n\times \AA^n)^0\times \AAA\longrightarrow \Sym^n(X)\times \AAA$ for the $S_n$-invariant morphism $q:(V^n\times \AA^n)^0\twoheadrightarrow (V^n\times \AA^n)^0/\!\!/S_n\xrightarrow{\tilde{q}} \Sym^n(X)$ with $q(v_1,\ldots,v_n,z_1,\ldots,z_n)=(f(v_1),\ldots,f(v_n))$. Modding out the $S_n$-action, we see that 
\[ [ \Sym^n(V\times \AAA) \xrightarrow{\tilde{p}} \Sym^n(X)\times \AAA ]=\pr^\ast_{\Sym^n}([(V^n\times \AAA^n)^0/\!\!/S_n \xrightarrow{\tilde{q}} \Sym^n(X)])\] 
is indeed in $I^{gm}_{\Sym^n(X)}$.
\end{proof}
\begin{exercise} \label{generators} Convince yourself using the formula $\sigma^n(a+b)=\sum_{l=0}^n\pi^{(l)}_!(\sigma^l(a)\boxtimes \sigma^{n-l}(b))$ for the motivic theory $R(-\times \AAA)$ with $\pi^{(l)}:\Sym^l(X)\times \Sym^{n-1}(X)\longrightarrow \Sym^n(X)$ being the natural map, that $a\in I^{gm}_X$ implies $\sigma^n(a)\in I^{gm}_{\Sym^n(X)}$ is indeed true if it  already holds for generators $a=\pr^\ast_X([V\xrightarrow{f} X])$ of $I^{gm}_X$. 
\end{exercise}

Due to Lemma \ref{monodromy_extension}, we can form the quotient $R^{gm}_{mon}(X)=R^{gm}_{\GG_m}(X\times \AAA)/\pr_X^\ast R^{gm}(X)$ and obtain a new motivic theory together with a morphism $R^{gm}\to R^{gm}_{mon}$ of motivic theories (cf.\ Exercise \ref{geometric_part}) given by $R^{gm}(X) \xrightarrow{(\id\times 0)_!} R^{gm}_{\GG_m}(X\times \AAA) \twoheadrightarrow R^{gm}_{mon}(X)$.

\begin{exercise}
Fix a motivic theory $R$ and consider the map $R^{gm}_{\GG_m}(X\times \AAA) \longrightarrow R^{gm}(X)$ given by $a\longmapsto (\id\times 0)^\ast(a)-(\id\times 1)^\ast(a)$, where $0,1:\Spec\kk \to \AAA$ are the obvious maps. As $\pr_X^\ast R^{gm}(X)$ is in the kernel, we obtain a well-defined group homomorphism $R^{gm}_{mon}(X)\to R^{gm}(X)$. Note that the composition $R^{gm}(X) \rightarrow R^{gm}_{mon}(X) \rightarrow R^{gm}(X)$ is the identity. Hence, $R^{gm}(X)$ is a direct summand of $R^{gm}_{mon}(X)$. However, show that the retraction $R^{gm}_{mon}(X)\to R^{gm}(X)$ is not a morphism of motivic theories and $R^{gm}$ is not a direct summand of the motivic theory $R^{gm}_{mon}$.  
\end{exercise}
\begin{exercise}
Using the notation of the previous exercise, show that the kernel of $\Con^{gm}_{mon}(X) \longrightarrow \Con^{gm}(X)=\Con(X)$ is trivial, i.e.\ $\Con(X)=\Con^{gm}_{mon}(X)$. On the other hand, show that the kernel is nonzero for $X=\Spec\kk$ and $R=\underline{\Ka}_0(D^b(-,\QQ))$, $R=\underline{\Ka}_0(\MHM(-)_{mon})$ and $R=\underline{\Ka}_0(\Sch)$.  
\end{exercise}
If $R=R^{gm}$, we suppress the superscript ``gm'' from notation. This applies for instance to $\underline{\Ka}_0(\Sch)$ but also to $R^{gm}$ as $(R^{gm})^{gm}=R^{gm}$. 
\begin{exercise}
 Check that  $R^{gm}_{mon}=(R^{gm})^{gm}_{mon}=(R^{gm})_{mon}$ using our convention for the last equation. 
\end{exercise}
If $R^{gm}\subsetneq R$ is a proper subtheory, as for example for $\underline{\Ka}_0(D^b_{con}(-,\QQ))$ or for $\underline{\Ka}_0(\MHM(-)_{mon})$, we can nevertheless define a theory $R_{mon}$ under the assumption that the formula 
 \[ \sigma^n(a\boxtimes b) =\sum_{\lambda \dashv n} \pi^{(n)}_!\Big(P^\lambda(\sigma^1(a),\ldots,\sigma^n(a))\boxtimes  P^\lambda(\sigma^1(b),\ldots,\sigma^n(b))\Big) \]
 holds in $R(\Sym^n(X\times Y))$ for all $0\not=n\in \NN$, all $a\in R(X),b\in R(Y)$ and all $X,Y$, where the sum is taken over all partitions $\lambda=(\lambda_1\ge\ldots  \ge \lambda_n\ge 0)$ of $n$ and 
 \[ P^\lambda(x_1,\ldots,x_n)=\det(x_{\lambda_i}+j-i)_{1\le i,j\le n}=\left|\begin{array}{cccc} x_{\lambda_1} & x_{\lambda_1+1} & \ldots & x_{\lambda_{1}+n-1} \\ x_{\lambda_2-1} & x_{\lambda_2} & \ldots & x_{\lambda_2+n-2} \\ \vdots & \vdots & \ddots & \vdots \\ x_{\lambda_n-n+1} & x_{\lambda_n-n+2} & \ldots & x_{\lambda_n} \end{array} \right| \]
is the polynomial from the Jacobi-Trudi formula with the convention $x_0=1$ and $x_m=0$ for $m<0$ or $m>n$. Here, $\pi^{(n)}:\Sym^n(X)\times \Sym^n(Y)\longrightarrow \Sym^n(X\times Y)$ is the obvious map. The expression $P^\lambda(\sigma^1(a),\ldots,\sigma^n(a))$ is computed in $R(\Sym(X))$ with respect to the convolution product, and is an element of $R(\Sym^n(X))$ since $\lambda_1 +\ldots + \lambda_n=n$. Similarly for $P^\lambda(\sigma^1(b),\ldots,\sigma^n(b))$. It can be show that this formula holds\footnote{The formula is a direct consequence of the assumption that $R(\Sym(X\times Y))$ is a special $\lambda$-ring which is true for any ``decategorification''.}  whenever $R$ has a ``categorification'' as for example $\underline{\Ka}_0(D^b(\MHM_{mon}))$ or $\underline{\Ka}_0(D^b_{con}(-,\QQ))$. In this case, we may replace $R^{gm}_{\GG_m}(X\times \AAA)$ with the $R(X)$-submodule $R_{\GG_m}(X\times \AAA)$ of $R(X\times \AAA)$ generated by $R^{gm}_{\GG_m}(X\times \AAA)$. Note that $R(X\times \AAA)$ is an $R(X)$-module using the convolution product and the embedding $R(X)\hookrightarrow R(X\times \AAA)$ provided by the ``zero section'' $0_X=\id_X\times 0:X\hookrightarrow X\times \AAA$. The formula for $\sigma^n(a\boxtimes b)$ ensures that $R_{\GG_m}(X\times \AAA)$ is still closed under taking $\sigma^n:R(X\times \AAA)\longrightarrow R(\Sym^n(X)\times \AAA)$ and, thus, defines another motivic theory containing $R^{gm}_{\GG_m}(-\times \AAA)$ as a subtheory. Moreover, $\pr_X^\ast(R(X))=:I_X$ is a $\lambda$-ideal and the quotient $R_{mon}(X):=R_{\GG_m}(X\times \AAA)/I_X$ is a well-defined motivic theory which contains $R$ as a subtheory such that $R(X)\hookrightarrow R_{mon}(X)$ is a retract for every $X$. Moreover, the following diagram is cartesian
\[ \xymatrix { R^{gm} \ar@{^{(}->}[r] \ar@{^{(}->}[d] & R^{gm}_{mon} \ar@{^{(}->}[d] \\ R \ar@{^{(}->}[r] & R_{mon}.}\]
\begin{exercise} Show that $R^{gm}_{\GG_m}(X\times \AAA)$ is already an $R(X)$-module under the assumption $R=R^{gm}$. Also $I^{gm}_X=\pr^\ast_X (R(X))$ in this case. Hence, we do not get anything new by the previous construction whenever it applies, and putting $R_{mon}:=R^{gm}_{mon}$ for theories $R=R^{gm}$ will not cause any confusion.
\end{exercise}

\begin{example}  \rm
There is a morphism $\underline{\Ka}_0(\MHM_{mon}) \longrightarrow \underline{\Ka}_0(\MHM)_{mon}$ of motivic theories. Roughly speaking, $\MHM_{mon}(X)$ is obtained by a categorification of the construction just described. One can show that for a regular function $f:X\to \AAA$ on a smooth scheme $X$ the image of $\phi^{mhm}_f\in \underline{\Ka}_0(\MHM_{mon}(X))$ under this morphism is already contained in the subgroup $\underline{\Ka}_0(\MHM(X))^{gm}_{mon}$. A similar statement holds for $\underline{\Ka}_0(D^b_{con}(-,\QQ))$.
\end{example}

\begin{example} \rm 
There is a morphism $\Ka^\muu_0(\Sch_X)\to \underline{\Ka}_0(\Sch_X)_{mon}$ for every (connected) $X$ with $\Ka^\muu_0(\Sch_X)$ being defined in \cite{DenefLoeser2}. In a nutshell, $\Ka^\muu_0(\Sch_X)$ is constructed very similar to $\underline{\Ka}_0(\Sch_X)_{mon}$ by considering generators $[Y\xrightarrow{f} X\times \AAA,\rho]$ with $Y$ carrying a $\GG_m$-action $\rho:\GG_m\times Y \to Y$ such that $f$ is homogeneous of degree $d> 0$. In contrast to our previous definition, the $\GG_m$-action $\rho$ is part of the data, and $\Ka^\muu_0(\Sch_X)\to \underline{\Ka}_0(\Sch_X)_{mon}$ forgets the $\GG_m$-action $\rho$. In particular, given another homogeneous map $Y'\xrightarrow{f'} X\times \AAA$ with $Y'$ carrying a $\GG_m$-action $\rho'$ and an isomorphism $\theta:Y'\xrightarrow{\sim} Y$ such that $f\theta=f'$ then $[Y\xrightarrow{f}X\times\AAA]=[Y'\xrightarrow{f'} X\times \AAA]$ in $\underline{\Ka}_0(\Sch_{X})_{mon}$, but the generators $[Y\xrightarrow{f}X\times \AAA,\rho]$ and $[Y'\xrightarrow{f'}X\times \AAA,\rho']$ of $\Ka^\muu_0(\Sch_X)$ might be different unless $\theta$ is $\GG_m$-equivariant. The relations in $\Ka^\muu_0(\Sch_X)$ are the cut and paste relation for $\GG_m$-invariant closed subschemes $Z\subset Y$ and $[Y\times \AAA \xrightarrow{u\times \id_\AAA} X\times \AAA,\rho]=0$ for every $\GG_m$-invariant morphism $u:Y\to X$ from a scheme $Y$ with a good $\GG_m$-action. Here, $\rho$ is given by $g(y,z)=(gy,gz)$ using the $\GG_m$-action on $Y$. There is a third relation dealing with linear actions of $\mu_d\subset \GG_m$, the group of $d$-th roots of unity, which is also fulfilled in $\underline{\Ka}_0(\Sch_X)_{mon}$ as \'{e}tale locally trivial vector bundles are already Zariski locally trivial. 
\end{example}
\begin{exercise} \label{functoriality}
 Show that the construction $R\mapsto R^{gm}_{mon}$ is functorial in $R$.
\end{exercise}

The motivic theory $R^{gm}_{mon}$ will be the target of our vanishing cycle which we are going to construct now. For this let $f:X\to \AAA$ be a regular function on a smooth connected scheme and let $\altL_n(X)$ be the scheme parameterizing all arcs of length $n$ in $X$, i.e.\ the scheme representing the set-valued functor $Y\mapsto \Mor(Y\times \Spec \kk[z]/(z^{n+1}), X)$. The standard action of $\GG_m$ on $\AAA=\Spec \kk[z]$ given by $z\mapsto gz$ induces an action on $\Spec \kk[z]/(z^{n+1})$ and, hence, also on $\altL_n(X)$. By functoriality applied to $f:X\to \AAA$, we also get a $\GG_m$-equivariant morphism $\altL_n(X)\xrightarrow{\altL_n(f)} \altL_n(\AAA)\cong \AA^{n+1}$, where $\GG_m$ acts coordinatewise on the latter arc space with weights $0,1,\ldots,n$.  Fix $t\in\AAA(\CC)$ and consider the map $f_n=\pr_{n+1}\circ \altL_n(f):\altL_n(X)|_{X_t} \longrightarrow  \AAA$, a $\GG_m$-equivariant map of degree $n$, and the projection $\pi_n:\altL_n(X)\to X$ mapping an arc to its base point. By $\GG_m$-equivariance, $[\altL_n(X)|_{X_t} \xrightarrow{\pi_n\times f_n}X\times \AAA]$ is in $R^{gm}_{\GG_m}(X_t\times \AAA)$ and defines an element in $R^{gm}_{mon}(X_t)$. We form the generating series
\[ Z_{f,t}^R(T):=\sum_{n\ge 1} \LL_R^{-n\dim X} [\altL_n(X)|_{X_t} \xrightarrow{\pi_n\times f_n}X\times \AAA] \,T^n \mbox{ in }R^{gm}_{mon}(X_t)[[T]].\]
The following result is a consequence of  Thm.\ 3.3.1 in the article \cite{DenefLoeser2} of Denef and Loeser or   of  Thm.\ 5.4 in Looijenga's paper \cite{Looijenga1}.
\begin{theorem}
 The series $Z_{f,t}^R(T)$ is a Taylor expansion of a rational function in $T$. The latter has a regular value at $T=\infty$. 
\end{theorem}
\begin{definition}
 Using the same notation for the rational function, we define $\phi_{f,t}^R:=\unit_{X_t}+Z_{f,t}^R(\infty)\in R^{gm}_{mon}(X_t)$ and, finally, $\phi_f^R:=\sum_{t\in \AAA(\CC)} \iota_{t\,!}\phi_{f,t}^R \in R^{gm}_{mon}(X)$ to be the vanishing cycle of $f:X\to \AAA$. For non-connected $X$ we apply the definition to every connected component $X_i$ and define $\phi^R_f$ to be the family $(\phi^R_{f|_{X_i}})_{X_i\in \pi_{0}(X)}$ in $R^{gm}_{mon}(X)=\prod_{X_i\in \pi_0(X)}R^{gm}_{mon}(X_i)$.  
\end{definition}
\begin{example} \rm One can show $\phi^{\Con}_f=\phi^{con}_f$ for all  $f:X\to \AAA$ on smooth $X$.
\end{example}
\begin{example}\rm Let $f:X\to \AAA$ be a regular function on a smooth connected scheme $X$. Using the morphism $\Ka^{\muu}_0(\Sch_X)\to \underline{\Ka}_0(\Sch_X)_{mon}$, the vanishing cycle $\phi^{mot}_f$ constructed by Denef and Loeser maps to $\phi^{\underline{\Ka}_0(\Sch)}_f$ up to the normalization factor $(-1)^{\dim X}$, and we will keep the shorter notation $\phi^{mot}_f$ for $\phi^{\underline{\Ka}_0(\Sch)}_f$.
\end{example}
\begin{example} \rm Let $f$ be a regular function on a smooth scheme as before. Using the map $\underline{\Ka}_0(\MHM(X)_{mon})\longrightarrow \underline{\Ka}_0(\MHM(X))_{mon}$ discussed earlier, the vanishing cycle $\phi^{mhm}_f$ maps to $\phi^{\underline{\Ka}_0(\MHM)}_f$, and we will keep the shorter notation $\phi^{mhm}_f$. Similarly for $\phi^{perv}_f$.
\end{example}
\begin{exercise} \label{functoriality2} Prove that $\phi^R_f$ is functorial in $R$. In particular, the diagram 
 \[ \xymatrix{ \underline{\Ka}_0(\Sch_X) \ar@{=}[r] \ar[d]^{\phi^{mot}} & \underline{\Ka}_0(\Sch_X) \ar[d]^{\phi^R_f} \\ \underline{\Ka}_0(\Sch_X)_{mon} \ar[r] & R^{gm}_{mon}(X) } \]
 commutes where we used Lemma \ref{vanishing_cycle_morphism}, Lemma \ref{initial_object} and Exercise \ref{functoriality} to construct the corresponding morphisms. Conclude that $\phi^R_f$ is a vanishing cycle using the known fact that $\phi^{mot}$ is a vanishing cycle.
\end{exercise}

In order to compute the vanishing cycle in practice, we choose an embedded resolution of $X_t\subset X$, i.e.\ a smooth variety $Y$ together with a proper morphism $\pi:Y \rightarrow X$ such that $Y_t=(f\circ\pi)^{-1}(a)=\pi^{-1}(X_t)$ is a normal crossing divisor and $\pi: Y\!\setminus\! Y_t \xrightarrow{\sim} X\!\setminus\! X_t$. Denote the irreducible components of $Y_t$ by $E_i$ with $i\in J$ and let $m_i>0$ be the multiplicity of $f\circ \pi$ at $E_i$. Since $f\circ \pi$ is a section in $\mathcal{O}_Y(-\sum_{i\in J}m_iE_i)$, it induces a regular map to $\AAA$ from the total space of $\mathcal{O}_Y(\sum_{i\in I} m_iE_i)$ for any $\emptyset \neq I \subset J$. The latter space restricted to $E_I^\circ:=\cap_{i\in I}E_i \!\setminus\! \cup_{i\not\in I} E_i$ is just $\otimes_{i\in I} N_{E_i|Y}^{\otimes m_i}|_{E_I^\circ}$. By composition with the tensor product we get a regular map $f_I:N_I:=\prod_{i\in I}( N_{E_i|Y}\!\setminus\! E_i) |_{E_I^\circ} \longrightarrow  \AAA$ which is obviously homogeneous of degree $m_i$ with respect to the $\GG_m$-action on the factor $(N_{E_i|Y}\!\setminus\! E_i)|_{E_I^\circ}$ and homogeneous of degree $m_I:=\sum_{i\in I}m_i$ with respect to the diagonal $\GG_m$-action.  By composing with $\pi:Y \rightarrow X$, the projection $N_I \rightarrow E_I^\circ$  induces a  map $\pi_I:N_I \rightarrow X_t$.

\begin{theorem}[\cite{DenefLoeser2} or \cite{Looijenga1}] \label{resolution}
Let $f:X\rightarrow \AA_k^1$ be a regular map and $\pi:Y \rightarrow X$ be an embedded resolution of $X_t$. In the notation just explained we have\footnote{This formula seems to differs from the one given in \cite{DenefLoeser2} or \cite{Looijenga1} by a sign. However, this is not true as the authors work with schemes  over $X$ with good $\mu_d$-action. Given such a scheme $Y\xrightarrow{u} X$ with $\mu_d$-invariant $u$, the associated generator of $R^{gm}_{mon}(X)$ is $-[Y\times_{\mu_d} \GG_m \ni (y,z)\mapsto (u(y),z^d)\in X\times \AAA]_R$. The sign here is chosen in such a way that if $Y$ carries a trivial $\mu_d$-action, then the generator is equivalent to $[Y\ni y \mapsto (u(y),0)\in X\times \AAA]$ in $R^{gm}(X)\subset R^{gm}_{mon}(X)$.}  
\[ Z_{f,t}^R(\infty) \;=  \sum_{\emptyset \neq I\subset J} (-1)^{|I|-1}[N_I\xrightarrow{\pi_I\times f_I} X_t\times \AAA] \,\in\, R^{gm}_{mon}(X_t). \]
\end{theorem}
\begin{corollary}[support property]\label{support_property}
Given a regular function $f:X\to \AAA$ on a smooth scheme $X$, the vanishing cycle $\phi^R_f$ is supported on $\Crit(f)$, i.e.\ $\phi^R_f$ is in the image of the embedding $ R(\Crit(f))^{gm}_{mon}\hookrightarrow R(X)^{gm}_{mon}$, where $\Crit(f)=\{df=0\}\subset X$ is the critical locus of $f$.  
\end{corollary}

The following result might also be helpful when it comes to actual computations. Let $\int_X a \in R(\Spec\kk)$ be the short notation for $(X\to \Spec \kk)_!(a)$ for $a\in R(X)$. 
\begin{theorem}[\cite{DaMe1}] \label{equivariant_case}
 Let $X$ be a smooth variety with $\GG_m$-action such that every point $x\in X$ has a neighborhood $U\subset X$ isomorphic to $\AA^{n(x)}\times U^{\GG_m}$ with $\GG_m$ acting by multiplication (with weight one) on $\AA^{n(x)}$. Let $f:X\to \AAA$ be a homogeneous function of degree $d>  0$. Then, $\int_X \phi^R_f=[X\xrightarrow{f}\AAA]$ in $R^{gm}_{mon}(\Spec\kk)$.    
\end{theorem}
\begin{exercise} \label{square_root} \hfill
\begin{enumerate} 
\item Prove that the scheme $\{(x,y)\in \AA^2\mid xy\not= 0\} \ni (x,y)\mapsto xy\in \AAA$ over $\AAA$ is isomorphic to the scheme $\GG_m\times \GG_m\ni (x,y) \to y \in \AAA$ over $\AAA$ and conclude $[ \{xy\not=0 \} \xrightarrow{xy} \AAA]_R=-[\GG_m]_R=1-\LL_R$ in $R^{gm}_{mon}(\Spec\CC)$.
\item  Use this and any of the two previous theorems to show $\phi^R_{f,0}=\int_{\AA^2}\phi^R_f =\LL_R$ for the function $f:\AA^2\ni (x,y)\mapsto xy\in \AAA$. Note that $\Crit(f)\cong\Spec\CC$ is the origin $0\in \AA^2$. Hence, $\phi^R_f$ is located at the origin. 
\item Use the result of the second part and the product formula for vanishing cycles to prove that $\phi_{z^2}^R$ for the function $\AAA\ni z\mapsto z^2\in \AAA$ is a square root $\LL^{1/2}_R$ of $\LL_R$.
\item Show that $\sigma^n(\LL^{1/2}_R)=0$ for all $n\ge 2$. Hint: It is a well-known fact that $\AA^n\ni (z_1,\ldots,z_n)\longmapsto \big(\sum_{i=1}^n z_i^k\big)_{k=1}^n \in \AA^n$ has a factorization $\AA^n \twoheadrightarrow \Sym^n(\AAA) \xrightarrow{\sim} \AA^n$ into the quotient map for the natural $S_n$-action on $\AA^n$ and an isomorphism. Remember that for every scheme $X$ the equation $[X\times \AAA \xrightarrow{\pr_{\AAA}} \AAA]=0$ holds in $R(\Spec\CC)_{mon}^{gm}$ by construction. 
\end{enumerate}
\end{exercise}

\subsection{Vanishing cycles for quotient stacks}

The theory of vanishing cycles for regular functions $f:X\to \AAA$ on smooth schemes generalizes straight forward to functions $\fst:\XXX\to \AAA$ on (disjoint unions of) smooth quotient stacks. Note that a quotient stack $X/G$ is called smooth if $X$ is smooth. A closed substack $\PPP\subset \XXX$ is given by $Y/G$ for a $G$-invariant closed subset $Y\subset X$. The blow-up of $\XXX$ in $\PPP$ is then simply given by the quotient stack $\Bl_\PPP \XXX=\Bl_Y X /G$ having exceptional divisor $\mathfrak{E}=E/G$. Given a quotient stack $X/G$ with smooth $X$ and a regular function $\fst:X/G\to \AAA$, we denote with $\Crit(\fst)$ the quotient stack $\Crit(\fst\rho)/G$ with $\rho:X\to X/G$. The generalization to disjoint unions of quotient stacks is at hand.
\begin{definition}
 Given a stacky motivic theory $R$, a stacky vanishing cycle (with values in $R$) is a rule  associating to every regular function $\fst:\XXX\to \AAA$ in a disjoint union of smooth quotient stacks an element $\phi_{\fst}\in R(\XXX)$ such that the following holds.
 \begin{enumerate}
  \item If $u:\PPP\to \XXX$ is a smooth, then $\phi_{\fst\circ u}=u^\ast(\phi_{\fst})$.
  \item Let $\XXX$ be a disjoint union of smooth quotients containing a smooth closed substack $i:\PPP\hookrightarrow \XXX$. Denote by $j:\mathfrak{E}\hookrightarrow \Bl_{\PPP} \XXX$ the exceptional divisor of the blow-up $\pi:\Bl_\PPP \XXX \to \XXX$ of $\XXX$ in $\PPP$. Then the  formula  \[ \pi_!\big( \phi_{\fst\circ \pi} - j_! \phi_{\fst\circ \pi\circ j}\big) =  \phi_{\fst} - i_!\phi_{\fst\circ i}\]     holds for every $\fst:\XXX\to \AAA$.
  \item Given two morphisms $\fst:\XXX\to \AAA$ and $\mathfrak{g}:\YYY\to \AAA$ with smooth $\XXX$ and $\YYY$, we introduce the notation $\fst\boxtimes \mathfrak{g}:\XXX\times \PPP\xrightarrow{\fst\times \mathfrak{g}} \AAA\times \AAA\xrightarrow{+} \AAA$. Then $\phi_{\fst\boxtimes \mathfrak{g}}=\phi_{\fst}\boxtimes \phi_{\mathfrak{g}}$ in $R(\XXX\times \PPP)$. Moreover, $\phi_{\Spec\kk\xrightarrow{0}\AAA}(1)=1$.
 \end{enumerate}
\end{definition}
Recall, that we constructed a correspondence between motivic theories $R$ with $\LL_R^{-1},(\LL^n_R-1)^{-1}\in R(\Spec\kk)$ for all $0\not= n\in \NN$ and stacky motivic theories satisfying $\rho^\ast:R(X/G)\xrightarrow{\sim} R(X)^G$ for every special group $G$. 
\begin{lemma} \label{stacky_vanishing_cycles}
 Let $R$ be a motivic theory such that $\LL^{-1}_R, (\LL^n_R-1)^{-1}\in R(\Spec\kk)$ for all $0\not= n\in \NN$. The restriction to schemes provides a bijection between stacky vanishing cycles with values in $R^{st}$ and vanishing cycles with values in $R$. 
\end{lemma}
\begin{proof}
By applying the first property of a stacky vanishing cycle to the smooth map $\rho:X\to X/G$, we see that $\phi_{\fst}$ is uniquely determined by $f=\fst\circ\rho:X\to \AAA$ as $\rho^\ast:R(X/G)\to  R(X)$ is injective for special $G$. 
Moreover, applying the first property once more to 
\[ \xymatrix { X\times G \ar[r]^m \ar[d]^{\pr_X} & X \\ X} \]
we observe that the vanishing cycle $\phi_f$ of a $G$-invariant function $f:X\to \AAA$ is $G$-invariant. Hence, given a vanishing cycle with values in $R$, we can define $\phi^{st}_{\fst}$ to be the unique element in $R^{st}(X/G)$ mapping to $\phi_f=\phi_{\fst\circ \rho}$ under the isomorphism $\rho^\ast:R^{st}(X/G)\cong R(X)^G$. Alternatively, we can write $\phi^{st}_\fst=\rho_!(\phi_f)/[G]_R$ since $[G]_R^{-1}\rho_!$ is the inverse of $\rho^\ast$ on $R(X)^G$. 
If $\fst:\XXX\to \AAA$ is a function on a disjoint union of quotient stacks $\XXX_i=X_i/G_i$, we may assume that $G_i$ is special for all $i$ (see Exercise \ref{special}) and define $\phi^{st}_\fst$ by means of the family $\phi^{st}_{\fst|_{\XXX_i}}$ using the first property of a stacky motivic theory.   
\end{proof}
\begin{exercise}
 Complete the proof of the previous lemma by checking that $\phi^{st}_\fst$ on a quotient stack $\XXX=X/G$ is independent of the choice of a presentation, i.e.\ if $X/G\cong Y/H$ for special groups $G$ and $H$, then $X$ is smooth if and only if $Y$ is smooth, and $\phi_{\fst\circ\rho_X}$ corresponds to $\phi_{\fst\circ\rho_Y}$ under the isomorphism $R(X)^G\cong R(Y)^H$ in this case. Moreover, show that $\phi^{st}$ satisfies the properties of a stacky vanishing cycle. 
\end{exercise}
\begin{example} \rm
 Given a motivic theory $R$ satisfying equation (\ref{eq7}) for all $a\in R(X), n\in \NN$ and a vanishing cycle $\phi$ with values in $R$, we can adjoin inverses of $\LL_R$ and $\LL^n_R-1$ for all $n>0$. By applying the previous lemma to the new motivic theory and the induced vanishing cycle, we get a motivic theory $R^{st}$ and a vanishing cycle $\phi^{st}$ such that $\eta_X(\phi_f)=\phi^{st}_f$ for every $f:X\to \AAA$, where $\eta_X:R(X)\to R^{st}(X)$ is the adjunction morphism $R\to R^{st}|_{\Sch_\kk}$ from the previous section. In particular, we can apply this to $\phi^{R^{gm}}_{can}$ with values in $R^{gm}$ and also to  $\phi^R$ with values in $R^{gm}_{mon}$ as equation (\ref{eq7}) holds in both cases (cf.\ Exercise \ref{geometric_part}). If $R$ satisfies equation (\ref{eq7}), we can also extend $\phi^R_{can}$ with values in $R$. Note that $\phi^{R,st}_{can, \fst} =\unit_{\XXX}\in R^{st}(\XXX)$ for all $\fst:\XXX\to \AAA$.
\end{example}

\section{Donaldson--Thomas theory}
 After introducing a lot of technical notation, we are now in the position to provide the definition of Donaldson--Thomas functions and to state a couple of results in Donaldson--Thomas theory. We close this section by given a list of examples. There are basically three approaches to define Donaldson--Thomas functions (see \cite{DavisonMeinhardt3}). The one given here is due to Kontsevich and Soibelman. 
 
 \subsection{Definition and main results}
 
 We start by fixing a stacky motivic theory $R$ satisfying $R(X/G)\cong R(X)^G$ for every quotient stack $X/G$ with special group $G$. Moreover, let $\phi$ be a stacky vanishing cycle with values in $R$ which is completely determined by its restriction to functions $f:X\to \AAA$ on smooth schemes $X$. (cf.\ Lemma \ref{stacky_vanishing_cycles}) As shown before, we could start with any vanishing cycle with values in a motivic theory and pass to the ``stackification''. Let us also assume that a square root $\LL_R^{1/2}$ of $\LL_R$ is contained in $R(\Spec\kk)$ such that $\sigma^n(\LL_R^{1/2})=0$ for all $n\ge 2$. 
 \begin{example}\rm
  Assume that $\sigma^n(a\LL_R)=\sigma^n(a)\LL^n_R$ holds for all $a\in R(X)$, all $n \in \NN$ and all $X$. As shown in \cite{DavisonMeinhardt3}, Appendix B, one  can extend the $\sigma^n$-operations to $R(X)[\LL^{1/2}_R]$ such that  $\sigma^n(-\LL_R^{1/2})=(-\LL_R^{1/2})^n$ for all $n\in \NN$ or equivalently $\sigma^n(\LL^{1/2}_R)=0$ for all $n\ge 2$. Thus, $R[\LL^{1/2}_R]$ is a new motivic theory having the required square root of $\LL_R$. We can apply the stackification to the canonical vanishing cycle $\phi^{R[\LL_R^{1/2}]}_{can}$ of $R[\LL_R^{1/2}]$ and obtain a pair $(R[\LL^{1/2}_R]^{st}, \phi_{can}^{R[\LL^{1/2}_R], st})$ satisfying our requirements. Note that the assumption on $R$ is always fulfilled if we replace $R$ with $R^{gm}$ (cf.\ Exercise \ref{geometric_part}).
 \end{example}
\begin{example} \rm
 For every motivic theory $R$ the element $[\AAA\ni z \mapsto z^2\in \AAA]_R\in R^{gm}_{mon}(\Spec \CC)$ is the required square root of $\LL_{R^{gm}_{mon}}$ as shown in Exercise \ref{square_root}. The stackification of $\phi^R$ with values in $R^{gm}_{mon}$ will match our requirements. This applies in particular to $\phi^{mot}$ and $\phi^{mhm}$. Note that the stackification of $\Con=\Con^{gm}_{mon}$ and of $\underline{\Ka}_0(D^b_{con}(-,\QQ))^{gm}_{mon}$ is zero as $[\Gl(n)]_R=0$ in both cases.  
\end{example}

We fix a quiver $Q$ with potential $W$ and a geometric stability condition $\zeta$. Recall that $\Mst^{\zeta-ss}$ was the stack of $\zeta$-semistable quiver representations with coarse moduli space $\Msp^{\zeta-ss}$ parameterizing polystable representations. Similarly $\Mst$ was the stack of all quiver representations and $\Msp^{ssimp}$ its coarse moduli space parameterizing semisimple representations. There are various maps between these spaces as shown in the following diagram
\[ \xymatrix @C=2cm { \Mst^{\zeta-ss} \ar@{^{(}->}[r] \ar[d]^{p^\zeta} & \Mst \ar[d]^p & \\ \Msp^{\zeta-ss} \ar[r]^{q^\zeta} & \Msp^{ssimp} \ar[r]^{\dim\times \WW} & \NN^{Q_0}\times \AAA.}\]
Note that the maps in the lower horizontal row are homomorphisms of monoids with respect to (direct) sums. Moreover, $q^\zeta$ is proper. Denote the composition $\Mst^{\zeta-ss} \hookrightarrow \Mst \xrightarrow{\WW\circ p} \AAA$ with $\WWW^\zeta$. For a fixed slope $\mu\in \mathbb{R}$, let us introduce the short hand $\phi_{\WWW^\zeta}(\ICS_{\Mst^{\zeta-ss}_\mu})$ for the object in $R(\Mst^{\zeta-ss})\cong \prod_{d\in \NN^{Q_0}} R(\Mst^{\zeta-ss}_d)$ having components 
\[ \LL_R^{(d,d)/2}\phi_{\WWW^\zeta|_{\Mst^{\zeta-ss}_d}}=\frac{\LL_R^{(d,d)/2}}{[G_d]_R}\rho_{d!}\phi_{\Tr(W)_d|_{X_d^{\zeta-ss}}} \]
if $d$ has slope $\mu$ or $d=0$ and $0$ for the remaining dimension vectors $d$. The idea behind the notation is the following. The vanishing cycle $\phi$ defines a map $\underline{\Ka}_0(\Sch_{\Mst^{\zeta-ss}})^{st} \longrightarrow R(\Mst^{\zeta-ss})$ mapping $\unit_{\Mst^{\zeta-ss}}=(\Mst^{\zeta-ss}\to \Spec\kk)^\ast(1)$ to $\phi_{\WWW^\zeta}$. If we define $\ICS_{\Mst^{\zeta-ss}_\mu}$ to be the object in $\underline{\Ka}_0(\Sch_{\Mst^{\zeta-ss}})[\LL^{1/2}]^{st}$ whose restriction to $\Mst^{\zeta-ss}_d$ is $\LL^{(d,d)/2}\unit_{\Mst^{\zeta-ss}_d}$ if $d$ has slope $\mu$ or $d=0$ and zero else, then $\phi_{\WWW^\zeta}(\ICS_{\Mst^{\zeta-ss}_\mu})$ is just the image of $\ICS_{\Mst^{\zeta-ss}_\mu}$ under the induced map $\underline{\Ka}_0(\Sch_{\Mst^{\zeta-ss}})[\LL^{1/2}]^{st} \longrightarrow R(\Mst^{\zeta-ss})$. One should think of $\ICS_{\Mst^{\zeta-ss}_\mu}$ as (the class of) the motivic intersection complex of $\Mst^{\zeta-ss}_\mu\subseteq \Mst^{\zeta-ss}$.

Let us define the convolution product on $R(\Msp^{\zeta-ss})$ by means of 
\[ R(\Msp^{\zeta-ss})\otimes R(\Msp^{\zeta-ss}) \xrightarrow{\boxtimes} R(\Msp^{\zeta-ss}\times \Msp^{\zeta-ss}) \xrightarrow{\oplus_!} R(\Msp^{\zeta-ss}) \]
and operations $\Sym^n:R(\Msp^{\zeta-ss}) \to R(\Msp^{\zeta-ss})$ for $n\in \NN$ via
\begin{equation}\label{lambda_operations} R(\Msp^{\zeta-ss})\xrightarrow{\sigma^n} R(\Sym^n \Msp^{\zeta-ss}) \xrightarrow{\oplus_!} R(\Msp^{\zeta-ss}). \end{equation}
\begin{lemma}
 For $a\in R(\Msp^{\zeta-ss})$ with $a_0:=a|_{\Msp^{\zeta-ss}_0}=0$ the infinite sum $\Sym(a):=\sum_{n\in \NN} \Sym^n(a)$ has only finitely many nonzero summands after restriction to $\Msp^{\zeta-ss}_d$ and, hence, defines a well defined element in $R(\Msp^{\zeta-ss})\cong\prod_{d\in \NN^{Q_0}} R(\Msp^{\zeta-ss}_d)$. Conversely, every element $b\in R(\Msp^{\zeta-ss})$ with $b_0=1\in R(\Msp^{\zeta-ss}_0)=R(\Spec\kk)$ can be written uniquely as $\Sym(a)$. The map $\Sym(-)$ is a group homomorphism form the additive group $\{a\in R(\Msp^{\zeta-ss})\mid a_0=0\}$ to the multiplicative group $\{b\in R(\Msp^{\zeta-ss})\mid b_0=1\}$. The same holds true if we replace $\Msp^{\zeta-ss}$ with $\Msp^{ssimp}$ or with $\NN^{Q_0}$. 
\end{lemma}
\begin{proof} Fix $d\not= 0$. Since $\oplus$ maps $\Sym^n \Msp^{\zeta-ss}_e$ to $\Msp^{\zeta-ss}_{ne}$, we get $\Sym^n(a)|_{\Msp^{\zeta-ss}_d}=0$ for all $n>|d|=\sum_{i\in Q_0}d_i$, and the infinite sum is finite after restriction to $\Msp^{\zeta-ss}_d$. Conversely, given $b$ we set $a_0=0$. Suppose $a_e\in R(\Msp^{\zeta-ss}_e)$ has been constructed for all 
dimension vectors\footnote{We write $e<d$ if $d=e+e'$ with $0\not=e'\in \NN^{Q_0}$.} $e<d$. We put 
\[ a_d:=b_d- \sum_{{n_1e_1+\ldots n_re_r=d\atop 0\not=n_j\in \NN, 0\not=e_i\not=e_j\in \NN^{Q_0}}} \prod_{j=1}^r \Sym^{n_j}(a_{e_j}).\]
Then $b=\Sym(a)$ for $a\in R(\Msp^{\zeta-ss})$ with $a|_{\Msp^{\zeta-ss}_d}=a_d$. Using the properties of $\sigma^n:R(\Msp^{\zeta-ss})\longrightarrow R(\Sym^n\Msp^{\zeta-ss})$ we get $\Sym(0)=1$ and $\Sym(a+b)=\Sym(a)\Sym(b)$. In particular, $\{b\in R(\Msp^{\zeta-ss})\mid b_0=1\}\cong\{a\in R(\Msp^{\zeta-ss})\mid a_0=0\}$ as groups. The  proof for $\Msp^{ssimp}$ and $\NN^{Q_0}$ is similar.
\end{proof}
\begin{exercise} \label{master_equation}
Let $\iota:\Msp^{\zeta-st}\hookrightarrow \Msp^{\zeta-ss}$ be the inclusion of the moduli space of $\zeta$-stable representations. Show that $\unit_{\Msp^{\zeta-ss}}=\Sym\big(\iota_!\unit_{\Msp^{\zeta-st}}\big)$ holds in $\underline{\Ka}_0(\Sch_{\Msp^{\zeta-ss}})$. Hint: Proof that the strata of the Luna stratification of $\Msp^{\zeta-ss}$ and the strata of the natural stratification of $\sqcup_{{n_1, \ldots, n_r\atop  d_i\not=d_j \forall i\not= j}}\prod_{i=1}^r(\Msp_{d_i}^{\zeta-ss})^{n_i}/\!\!/S_{n_i}$ given by the conjugacy types of the $S_{n_i}$-stabilizers coincide.  In other words, the canonical map $\oplus:\Sym(\Msp^{\zeta-st}) \longrightarrow \Msp^{\zeta-ss}$ is a ``constructible'' isomorphism. 
\end{exercise}
\begin{exercise}
 Show  that the restriction of $p^\zeta_!\big(\phi_{\WWW^\zeta}(\ICS_{\Mst^{\zeta-ss}_\mu})\big)$ to $\Msp^{\zeta-ss}_0$ is $1$.
\end{exercise}
Using the last exercise and the previous lemma, the following definition makes sense.
\begin{definition}
 The Donaldson--Thomas function $\DTS(Q,W)^\zeta\in R(\Msp^{\zeta-ss})$ is the unique element with $\DTS(Q,W)^\zeta|_{\Msp^{\zeta-ss}_0}=0$ such that $\DTS(Q,W)^\zeta_\mu:=\DTS(Q,W)|_{\Msp^{\zeta-ss}_\mu}$ solves the equation
 \[ p^\zeta_!\big(\phi_{\WWW^\zeta}(\ICS_{\Mst^{\zeta-ss}_\mu})\big)=\Sym\left( \frac{\DTS(Q,W)^\zeta_\mu}{\LL^{1/2}_R-\LL^{-1/2}_R} \right) \]
 in $R(\Msp^{\zeta-ss})$ for all $\mu\in (-\infty,+\infty]$. We also use the notation $\DTS(Q,W)^\zeta_d:=\DTS(Q,W)^\zeta|_{\Msp^{\zeta-ss}_d}$. The element $\int_{\Msp^{\zeta-ss}_d} \DTS(Q,W)^\zeta_d=:\DT(Q,W)^\zeta_d\in R(\Spec \kk)$ is called the Donaldson--Thomas invariant of $(Q,W)$ with respect to $\zeta$ for dimension vector $d$. If $W=0$, we simply write $\DTS(Q)^\zeta_d$ and $\DT(Q)^\zeta_d$. 
\end{definition}
In view of Exercise \ref{master_equation}, one might hope that $\DTS(Q,W)^\zeta_d/(\LL_R^{1/2}-\LL_R^{-1/2})$ is something like $p^\zeta_! \phi_{\WWW^\zeta}(j_!\ICS_{\Mst^{\zeta-st}_\mu})$, where $j:\Mst^{\zeta-st}_\mu\hookrightarrow \Mst^{\zeta-ss}_\mu$ denotes the inclusion, and $\ICS_{\Mst^{\zeta-st}_\mu}$ is defined similarly to $\ICS_{\Mst^{\zeta-ss}_\mu}$. Let us assume, we were allowed to commute $p^\zeta_!$ with $\phi_{\WWW^\zeta}$ which is a priori not clear as $p^\zeta$ is not proper. Then 
\[ p^\zeta_! \phi_{\WWW^\zeta}(j_!\ICS_{\Mst^{\zeta-st}_\mu})= \phi_{\WW\circ q^\zeta}\left(\frac{\iota_! \ICS_{\Msp^{\zeta-ss}_\mu}}{\LL^{1/2}_R-\LL_R^{-1/2}}\right). \]
for $\ICS_{\Msp^{\zeta-st}_d}=\LL_R^{((d,d)-1)/2}\unit_{\Msp^{\zeta-st}_d}$, and $\DTS(Q,W)^\zeta_d=\phi_{\WW\circ q^\zeta_d} (\iota_!\ICS_{\Msp^{\zeta-ss}_d})$ follows. This is not quite true. It turns out that the extension $\iota_!\ICS_{\Msp^{\zeta-st}_\mu}$ of $\ICS_{\Msp^{\zeta-st}_d}$  by zero has to be replaced with the ``correct'' extension $\ICS_{\overline{\Msp^{\zeta-st}_d}}$ which restricts to $\ICS_{\Msp^{\zeta-st}_d}$, but might also be nonzero on the boundary of $\Msp^{\zeta-st}_d$ inside $\Msp^{\zeta-ss}_d$. However, we have not defined $\ICS_{\overline{\Msp^{\zeta-st}_d}}$ yet.
\begin{definition}
We denote with $\ICS^{mot}_{\overline{\Msp^{\zeta-st}}}\in \underline{\Ka}_0(\Sch_{\Msp^{\zeta-ss}})[\LL^{1/2}]^{st}$ the Donaldson--Thomas function $\DTS(Q)^\zeta$ computed with respect to the stackification of the canonical vanishing cycle $\phi^{\underline{\Ka}_0(\Sch)[\LL^{1/2}]}_{can}$. Note that $\overline{\Msp^{\zeta-st}_d}=\Msp^{\zeta-ss}_d$ if $\Msp^{\zeta-st}_d\not=\emptyset$ and $\overline{\Msp^{\zeta-st}_d}=\emptyset$ else.
\end{definition}
The following result justifies the definition.
\begin{theorem}[\cite{MeinhardtReineke}] \label{main_result_1}
 If $\zeta$ is generic (see Definition \ref{generic_stability}), the element $\ICS^{mot}_{\overline{\Msp^{\zeta-st}}}$  maps to the classical intersection complex\footnote{Strictly speaking one has to normalize the class of the classical (shifted) intersection complex of $\Msp^{\zeta-ss}_d$ by multiplication with $(-\LL^{1/2})^{(d,d)-1}$ which does not change the underlying perverse sheaf.} $\ICS_{\overline{\Msp^{\zeta-st}}}$ of the closure of $\Mst^{\zeta-st}$ inside $\Msp^{\zeta-ss}$ under the map 
 $\underline{\Ka}_0(\Sch_{\Msp^{\zeta-ss}})[\LL^{1/2}]^{st}\longrightarrow \underline{\Ka}_0(\MHM(\Msp^{\zeta-ss}))[\LL^{1/2}]^{st}$ constructed in Lemma \ref{initial_object}. 
\end{theorem}
\begin{example} \rm
 Assuming equation (\ref{eq7}) for all $a\in R(X)$ and all $n\in \NN$ so that $\phi^R_{can}$ has a stacky extension $\phi^{R[\LL^{1/2}_R]}_{can}$.  In this case, $\DTS(Q,W)^\zeta=\DTS(Q)^\zeta$ is just the image of $\ICS^{mot}_{\overline{\Msp^{\zeta-st}}}$ under the canonical map $\underline{\Ka}_0(\Sch_{\Msp^{\zeta-ss}})[\LL^{1/2}]^{st} \longrightarrow R(\Msp^{\zeta-ss})[\LL^{1/2}]^{st}$ of Lemma \ref{initial_object}, and we may define $\ICS_{\overline{\Msp^{\zeta-st}}}:=\DTS(Q)^\zeta\in R(\Msp^{\zeta-ss})[\LL^{1/2}]^{st}$. As for mixed Hodge modules one can show that $\ICS_{\overline{\Msp^{\zeta-st}}}$ has a lift in $R(\Msp^{\zeta-ss})[\LL^{-1/2}]$ if the canonical map $\underline{\Ka}_0(\Sch)\to R$ factorizes trough $\underline{\Ka}_0(\MHM)$ or if $R$ is a sheaf in the \'{e}tale topology with $[\PP^r]$ acting as a nonzero divisor in each group $R(X)$ for all $r\in \NN$.  Note that we can replace $R$ with $R^{gm}$ to ensure equation (\ref{eq7}).
\end{example}
One can also prove the following result which should be seen as the analogue of Exercise \ref{master_equation}.
\begin{theorem}[\cite{DavisonMeinhardt3}] \label{main_result_2}
 Recall that every stacky vanishing cycle with values in $R$ satisfying our assumptions defines a map $\phi_{\WW\circ q^\zeta}:\underline{\Ka}_0(\Sch_{\Msp^{\zeta-ss}})^{st}\longrightarrow R(\Msp^{\zeta-ss})$. If this map commutes with the $\Sym^n$-operations of equation (\ref{lambda_operations}) for every $n\in \NN$, and if $\zeta$ is generic, then $\DTS(Q,W)^\zeta=\phi_{\WW\circ q^\zeta}\big(\ICS^{mot}_{\overline{\Msp^{\zeta-st}}}\big)$. 
\end{theorem}
\begin{example} \rm The assumption on $\phi_{\WW q^\zeta}$ is true for $\phi=\phi^{mhm}$. Hence, if $\zeta$ is generic, $\DTS(Q,W)^\zeta_d=\phi^{mhm}_{\WW\circ q^\zeta_d}\big(\ICS_{\Msp^{\zeta-st}_d}\big)$ if $\Msp^{\zeta-st}_d\not=\emptyset$ and zero else. It is not known yet whether or not $\phi^{mot}_{\WW\circ q^\zeta}$ commutes with the $\Sym^n$-operations, and we cannot apply the theorem. However, a counterexample is also not known, and conjecturally the theorem also holds for $\phi^{mot}$.  
\end{example}
\begin{exercise} Use the support property of $\phi^R$ (see Corollary \ref{support_property}) to show that $\DTS(Q,W)^\zeta_d$ is supported on $\Msp^{W,\zeta-ss}_d$, i.e.\ is an element in $R(\Msp^{W,\zeta-ss}_d)$, where $\Msp^{W,\zeta-ss}_d$ the moduli space parameterizing $\zeta$-polystable $\CC Q$-representations $V$ of dimension $d$ such that $\partial W/\partial \alpha=0$  on $V$ for all $\alpha\in Q_1$.
\end{exercise}

\subsection{Examples}
The aim of this section is to provide some examples of (motivic) Donaldson--Thomas invariants. 

\subsubsection{\rm \textbf{The m-loop quiver}}
 Let us consider the quiver $Q^{(m)}$ with one vertex and $m$ loops. 
 The choice of a stability condition is irrelevant as $\Mst^{\zeta-ss}=\Mst$. We take the canonical vanishing cycle of $\underline{\Ka}_0(\Sch)[\LL^{1/2}]$ and are only interested in Donaldson--Thomas invariants. As $\underline{\Ka}(\Sch_\NN)[\LL^{1/2}]^{st}\cong \Ka(\Sch_\kk)[\LL^{-1/2},(\LL^n-1)^{-1}:n\in \NN_\ast][[t]]$, we end up with the following power series
 \[ A^{(m)}(t):=(\dim \circ p)_!\big(\phi_{\WWW}(\ICS_{\Mst})\big)=\sum_{d\ge 0}\frac{\LL^{(m+1)d^2/2}}{[\Gl(d)]}t^d=\sum_{d\ge 0} \frac{\LL^{(md^2+d)/2}}{\prod_{i=1}^d(\LL^i-1)}t^d.\]
Note that the series is also well-defined for $m\in \ZZ$.
\begin{exercise}
 Prove the identity $A^{(m)}(\LL t)-A^{(m)}(t)=\LL^{(m+1)/2}\,t\, A^{(m)}(\LL^mt)$ for all $m\in \ZZ$.
\end{exercise}
For $m\in \NN$ we introduce the series 
\[ B^{(m)}(t):=A^{(m)}(\LL t)/A^{(m)}(t)=\Sym\big( \sum_{d\ge 1} \LL^{1/2}[\PP^{d-1}]\DT(Q^{(m)})_d\,t^d\big), \]
where we used the properties of $\Sym$ and the fact that $\dim_!$ commutes with $\Sym$. Moreover, due to the previous exercise
\[ B^{(m)}(t)=1+\LL^{(m+1)/2}t \prod_{i=0}^{m-1}B^{(m)}(\LL^it). \]
For $m=0$ the empty product on the right hand side is $1$, and we obtain $B^{(0)}(t)=1+\LL^{1/2}t$ as well as
\[ \DT(Q^{(0)})_d=\begin{cases} 1 & \mbox{for }d=1, \\ 0&\mbox{else.}   \end{cases} \]
This is in fully agreement with Theorem \ref{main_result_1} as $\Msp^{simp}_d=\Spec\kk$ for $d=1$ and $\Msp^{simp}_d=\emptyset$ else.\\
For $m=1$, we get $B^{(1)}(t)=1+\LL tB^{(1)}(t)$, and $B^{(1)}(t)=1/(1-\LL t)=\sum_{d\in \NN} \LL^dt^d$ follows. Hence,
\[ \DT(Q^{(1)})_d=\begin{cases} \LL^{1/2} & \mbox{for }d=1, \\ 0&\mbox{else.}   \end{cases} \]
Again, this is in fully agreement with Theorem \ref{main_result_1} as $\Msp^{simp}_d=\AAA$ for $d=1$ and $\Msp^{simp}_d=\emptyset$ else.\\
Solving the pseudo-algebraic equation for $m\ge 2$ is much more complicated, but the answer is given as follows. Note that $\ZZ/(d)=:C_d$ acts on the set $U_d:=\{ (a_1,\ldots,a_d)\in \NN^d\mid a_1+\ldots+a_d=(m-1)d\}$ by cyclic permutation. We call $a=(a_1,\ldots,a_d)$ primitive if $\Stab_{C_d}(a)=\{0\}$, and almost primitive if $a$ is primitive or $m\equiv 0 (2), d\equiv 2 (4)$ and $a=(b_1,\ldots,b_{d/2},b_1,\ldots,b_{d/2})$ for some primitive $(b_1,\ldots,b_{d/2})$. The subset $U^{ap}_d:=\{ a\in U_d\mid a\mbox{ is almost primitive}\}$ is obviously stable under the $C_d$-action. Define $\deg(a)=\sum_{i=1}^d(d-i)a_i$ and $\deg(C_d\cdot a)=\min\{ \deg(a')\mid a'\in C_d\cdot a\}$. 
\begin{theorem}[\cite{Reineke4}]
 Let $d\ge 1$ and $m\ge 2$. Then $\dim(\Msp^{simp}_d)=(m-1)d^2+1$ and 
 \[\DT(Q^{(m)})_d=\LL^{\frac{(m-1)d^2+1}{2}}\frac{1-\LL^{-1}}{1-\LL^{-d}}\sum_{C_d\cdot a\in U^{ap}_d/C_d} \LL^{-\deg(C_d\cdot a)}. \]
 In particular, $\chi_c(\DT(Q^{(m)})_d)=\frac{(-1)^{(m-1)d^2+1}}{d}|U^{as}_d/C_d|$.
\end{theorem}
\begin{exercise}
 By taking the Euler characteristic of every coefficient of $B^{(m)}(t)$, we get a series $\beta^{(m)}(t)\in \ZZ[[t]]$ satisfying $\beta^{(m)}(t)=1+(-1)^{m+1}\beta^{(m)}(t)^m$, but also $\beta^{(m)}(t)=\prod_{d\ge 1}(1-t^d)^{d\Omega^{(m)}_d}$ using the shorthand $\Omega^{(m)}_d:=\chi_c(\DT(Q^{(m)})_d)$. Prove that 
 \[ \Omega^{(2)}_d=\chi_c(\DT(Q^{(2)})_d)=\frac{1}{2d^2}\sum_{n|d} (-1)^{n+1}\mu(d/n){2n \choose n} \]
 for $d\ge 1$, where $\mu(-)$ denotes the M\"obius function. Hint: Solve the quadratic equation for $\beta^{(2)}(t)$ and use the  logarithmic derivative to prove the formula
 \[ \sum_{d\ge 1}2d^2\Omega^{(2)}_d \frac{t^d}{1-t^d}=\sum_{n\ge 1} \Big(\sum_{d|n} 2d^2\Omega^{(2)}_d\Big)t^n= 1-\frac{1}{\sqrt{1+4t}}. \]
 The Taylor expansion of $1/\sqrt{1+4t}$ is $\sum_{n\in \NN}(-1)^n{2n\choose n}t^n$. Finally, use the  M\"obius function to solve for $\Omega^{(2)}_d$. \\
 One can show $\Omega^{(2)}_d=F(d)$ for $F(d)$ being the coefficients introduced in \cite{JoyceMF}, equation (53). 
\end{exercise}

\subsubsection{\rm\textbf{Dimension reduction}}
 Let $Q$ be a quiver with potential $W$. Let $\zeta$ be a stability condition and $\mu\in [-\infty,\infty)$. Assume $\Mst^{\zeta-ss}_d=\Mst_d$ for all $d\in \Lambda_\mu$. As an example, we may take the King stability condition $\theta=0$ and $\mu=0$. Given a motivic theory $R$, the following arguments apply to $\phi^{R,st}$ with values in $R^{gm,st}_{mon}$. Without loss of generality, we consider the case $R=\underline{\Ka}_0(\Sch)$, i.e.\ $\phi=\phi^{mot,st}$. As in the previous example, we are only interested in Donaldson--Thomas invariants. A subset $C\subset Q_1$ such that every cycle in $W$ contains exactly one arrow in $C$ is called a cut of $W$. Let $\GG_m$ act on $X_d=\prod_{\alpha:i\to j} \Hom(\kk^{d_i},\kk^{d_j})$ by multiplying a linear map corresponding to $\alpha\in C$ with $g\in \GG_m$. By assumption, $\Tr(W)_d$ is homogeneous of degree one, and $\int_{X_d}\phi_{\Tr(W)_d}^{mot}$ is the residue class of $[X_d\xrightarrow{\Tr(W)_d}\AAA] $ in $\Ka_0(\Sch_\kk)_{mon}$ according to Theorem \ref{equivariant_case}. Consider the projection $\tau_d:X_d\to Y_d$ with $Y_d=\prod_{C\not\ni \alpha:i\to j}\Hom(\kk^{d_i},\kk^{d_j})$ which is a trivial vector bundle with fiber $F=\prod_{C\ni \alpha:i\to j}\Hom(\kk^{d_i},\kk^{d_j})$. As $\Tr(W)_d$ is linear along the fibers, we can think of it as being a section $\sigma^W_d$ of the dual bundle with fiber $F^\vee=\prod_{C\ni \alpha:i\to j}\Hom(\kk^{d_j},\kk^{d_i})$ using the trace pairing. Indeed, it maps a point $M=(M_\alpha)_{\alpha\not\in C}$ to $(\frac{\partial W}{\partial \alpha}(M))_{\alpha\in C}\in F^\vee$. 
 \begin{exercise}
  Show that $[\tau_d^{-1}\{\sigma^W_d\not=0\} \xrightarrow{\Tr(W)_d}\AAA]$ is in $\pr_\kk^\ast \Ka_0(\Sch_\kk)\subset \Ka_{0,\GG_m}(\Sch_\AAA)$. Hint: Choose an open cover  $\cup_{i\in I} U_i=\{\sigma^W_d\not= 0\}\subseteq Y_d$ such that $\tau_d: \tau_d^{-1}U_ i \to U_i$ splits into the kernel of $\Tr(W)_d$ and a complement of rank one. 
 \end{exercise}
Using the exercise, we obtain 
\begin{eqnarray*} [X_d\xrightarrow{\Tr(W)_d}\AAA]&=&[\tau_d^{-1}\{\sigma^W_d=0\}\xrightarrow{0}\AAA]\;=\;[F][\{\sigma^W_d=0\}]\\ &=&\LL^{\sum_{C\ni \alpha:i\to j}d_id_j}\left[\left\{\ M\in Y_d\mid \frac{\partial W}{\partial \alpha}(M)=0 \,\forall\, \alpha\in C\right\}\right]. \end{eqnarray*}
in $\Ka_0(\Sch_\CC)_{mon}$. Therefore,
\begin{eqnarray*} \lefteqn{ \Sym\left( \frac{\DT(Q,W)^\zeta_\mu}{\LL^{1/2}-\LL^{-1/2}} \right)\; =\; (\dim\circ p)_!\big(\phi_{\WWW}(\ICS_{\Mst_\mu})\big)} \\&=&\sum_{d\in \Lambda_\mu} \LL^{(d,d)/2+\sum_{C\ni \alpha:i\to j}d_id_j}\frac{\left[\left\{M\in Y_d\mid \frac{\partial W}{\partial \alpha}(M)=0\, \forall \,\alpha\in C\right\}\right]}{[G_d]} \, t^d, \end{eqnarray*}
is actually in $\Ka_0(\Sch_\kk)[\LL^{-1/2},(\LL^n-1)^{-1}:n\in \NN_\ast][[t_i\mid i\in Q_0]]$, where we used 
\[\Ka_0(\Sch_\kk)[\LL^{-1/2},(\LL^n-1)^{-1}:n\in \NN_\ast][[t_i: i\in Q_0]]=\underline{\Ka}(\Sch_{\NN^{Q_0}})[\LL^{1/2}]^{st} \subset \underline{\Ka}(\Sch_{\NN^{Q_0}})^{st}_{mon}.\] 
This reduction process is usually called dimension reduction, and \\
$\left\{M\in Y_d\mid \frac{\partial W}{\partial \alpha}(M)=0\, \forall \,\alpha\in C\right\}/G_d $ 
is the stack of $d$-dimensional representations of the algebra $\kk Q/(\alpha,\partial W/\partial \alpha\mid \alpha\in C)$.

\subsubsection{\rm\textbf{0-dimensional sheaves on a Calabi--Yau 3-fold}}

Let us illustrate the concept of dimension reduction using the quiver $Q^{(3)}$ with one vertex and three loops $x,y,z$. The choice of the stability condition does not matter. We take the potential $W=[x,y]z=xyz-yxz$, and $\kk Q/(\partial W/\partial\alpha\mid \alpha\in Q^{(3)}_1)=\kk[x,y,z]$ follows. Hence, representations of this algebra are just 0-dimensional sheaves of finite length $d$ on the Calabi--Yau 3-fold $\AA^3$. We can take $C=\{z\}$ and obtain $\kk Q/(z,\partial W/\partial z)=\kk[x,y]$, and representations of this algebra are 0-dimensional sheaves of finite length $d$ on the Calabi--Yau 2-fold $\AA^2$. Using $(d,d)^2=-2d^2$, we have to compute
\[ \sum_{d\in \NN} \frac{\left[\left\{M\in Y_d\mid \frac{\partial W}{\partial \alpha}(M)=0\, \forall \,\alpha\in C\right\}\right]}{[\Gl(d)]}\,t^d \]
which has already been done by Feit and Fine half a century ago in \cite{FeitFine}. The answer is
\[ \Sym\Big( \frac{1}{\LL-1}\sum_{d\ge 1} [\AA^2] t^d\Big), \]
and $DT(Q^{(3)},W)_d=\LL^{3/2}=\int_{\AA^3} \ICS^{mot}_{\AA^3}$ follows for all $d\ge 1$ if we define $\ICS^{mot}_{X}=\LL^{-\dim(X)/2}\unit_X\in  \underline{\Ka}_0(\Sch_X)^{st}_{mon}$ for every smooth equidimensional variety $X$. This example has been generalized to arbitrary Calabi--Yau 3-folds by Behrend, Bryan, Szendr\H{o}i.
\begin{theorem}[\cite{BBS}] \label{skyscraper}
 The motivic Donaldson--Thomas invariant for 0-dimensional sheaves of length $d\ge 1$ on a Calabi--Yau 3-fold $X$ is given by $\int_X \ICS^{mot}_X=\LL^{-3/2}[X]\in \Ka_0(\Sch_\kk)^{st}_{mon}$.
\end{theorem}
Note that the Donaldson--Thomas function $\DTS(Q^{(3)},W)_d$ is supported on the moduli space $\Sym^d(\AA^3)$ of semisimple $\kk[x,y,z]$ representations but the sublocus of simple representations is empty for $d>1$. However, the space of simple $\kk Q^{(3)}$-representations is nonempty even for $d>1$ which is the reason why $\DT(Q^{(3)},W)_d\not=0$. It seems plausible that $\DTS(Q^{(3)},W)_d$ is $\Delta_{d\,!}(\ICS^{mot}_{\AA^3})$, where $\Delta_d:\AA^3\hookrightarrow \Sym^d(\AA^3)$ is the diagonal embedding.   

\subsubsection{\rm\textbf{The 1-loop quiver with potential}}

Let us come back to the 1-loop quiver with $\kk Q^{(1)}=\kk[x]$ and choose an arbitrary potential $W\in \kk[x]$. Representations of the algebra $\kk Q/(dW/dx)=\kk[x]/(dW/dx)$ can be interpreted as (0-dimensional) sheaves of length $d$ on $\Crit(W)\subseteq \AAA$. Form the prime decomposition $dW/dx=c\prod_{i=1}^{r}(x-a_i)^{d_i-1}$ for some $1<d_i\in \NN, c\in \kk^\times, a_i\in \kk$ satisfying $a_i\not=a_j$ for $i\not=j$. As before, the choice of a stability condition does not effect the Donaldson--Thomas function. 
\begin{theorem}[\cite{DaMe1}] The Donaldson--Thomas function $\DTS(Q^{(1)},W)_1$ computed using  $\phi^{mot}$ is supported on $\Crit(W)$, i.e.\ is contained in $\underline{\Ka}_0(\Sch_{\Crit(W)})^{st}_{mon}\cong\prod_{i=1}^r \Ka_0(\Sch_\CC)^{st}_{mon}$, and its ``value'' at $a_i$ is given by  $\LL^{-1/2}[\AAA\xrightarrow{z^{d_i}} \AAA]\in \Ka_0(\Sch_\CC)^{st}_{mon}$. Moreover, $\DTS(Q^{(1)},W)_d=0$ for $d\not=1$. 
\end{theorem}
\begin{exercise}
 Prove this result using Theorem \ref{main_result_2} and the explicit form of the vanishing cycle given by embedded resolutions as in Theorem \ref{resolution}.
\end{exercise}
The formula remains true if we replace $\phi^{mot}$ with $\phi^R$ due to Exercise \ref{functoriality2}.

\subsubsection{\rm\textbf{Sheaves on (-2)-curves}}

Consider the following quiver $Q$
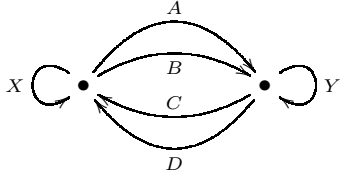
\[ \xymatrix @C=2cm { \bullet \ar@(ul,dl)[]_X \ar@/^1pc/[r]_B \ar@/^2pc/[r]^A & \bullet \ar@(ur,dr)[]^Y \ar@/^1pc/[l]_C \ar@/^2pc/[l]^D }\]
with potential 
\[ W_n=\frac{1}{n+1}\big(X^{n+1}-Y^{n+1}\big) - XCA+XDB+YAC-YBD \mbox{ for some }0<n\in \NN. \]
The bounded derived category $D^b\Jac(Q,W_n)$ has also a geometric interpretation. For this consider the singular affine variety $X_n=\{x^2+y^2+(z+w^n)(z-w^n)=0\}\subset \AA^4$ which is a local model for a 3-fold with an $A_{n}$-singularity. By blowing-up $X_n$ in $\{x=z\pm w^d=0\}$ we get two minimal resolutions $Y_n^\pm$ with an  exceptional locus $C\cong \PP^1$. The normal bundle of $C$ inside $Y_n^\pm$ is $\mathcal{O}_{\PP^1}(-1)\oplus \mathcal{O}_{\PP^1}(-1)$ for $n=1$ and $\mathcal{O}_{\PP^1}\oplus \mathcal{O}_{\PP^1}(-2)$ for $n>1$. In particular, $Y_n^\pm$ is a Calabi--Yau 3-fold which has actually a locally trivial fibration over $C$ with fiber $\AA^2$. For $d=1$, $Y^\pm_1$ is isomorphic to the normal bundle $\mathcal{O}_{\PP^1}(-1)\oplus \mathcal{O}_{\PP^1}(-1)$ and known as the conifold resolution. For $d>1$ this fibration is not a vector bundle as the transition functions are not linear. The resolutions $Y_n^+$ and $Y_n^-$ are related via a flop over $X_n$, and also isomorphic to each other. Moreover,
\[ D^b\Jac(Q,W_n) \cong D^b\Coh(Y_n^\pm) \]
and (complexes of) nilpotent representations on the left hand side correspond to (complexes of) sheaves supported on $C\subseteq Y^\pm_n$. We are only interested in these objects and choose any stability condition $\zeta=(\zeta_1,\zeta_2)$ with $\zeta_1\nparallel \zeta_2$ in $\mathbb{R}^2\cong\CC$. 
\begin{theorem}[\cite{DaMe2}]
 The Donaldson--Thomas invariant\\  $\DT(Q,W_n)^{nilp}_{(d_1,d_2)}\in \Ka_0(\Sch_\kk)^{st}_{mon}$ of nilpotent representations computed with respect to $\phi^{mot}$ is given by
 \[ \DT(Q,W_n)^{\nilp}_{(d_1,d_2)}=\begin{cases} \LL^{-3/2}[C]=\LL^{-3/2}(\LL+1) & \mbox{ if } 0\not= d_1=d_2, \\ \LL^{-1/2}[\AAA\xrightarrow{z^{n+1}}\AAA] &\mbox{ if }|d_1-d_2|=1, \\ 0& \mbox{ else.} \end{cases} \]
\end{theorem}
Of course, the formula remains true if we replace $\phi^{mot}$ with $\phi^R$ due to Exercise \ref{functoriality2}. Note that $\DT(Q,W_n)^{nilp}_{(d_1,d_1)}$ is just ``counting'' 0-dimensional sheaves on $Y_n^\pm$ supported on $C$ which explains the answer in view of Theorem \ref{skyscraper}. For $|d_1-d_2|=1$, there is just one simple nilpotent $\Jac(Q,W_n)$-representation $V$ with $\Ext^1(V,V)$ being of dimension one. However, the obstruction of deforming $V$ as a representation of $\Jac(Q,W_n)$ is controlled by some potential of the form $z^{n+1}$ induced by $W_n$. Hence, we are back in the context of the previous example. The case of $n=1$ has been studied earlier by Morrison, Mozgovoy, Nagao, Szendr\H{o}i in \cite{MMNS}.

\subsection{The Ringel--Hall algebra}

In the previous section we have seen some examples of Donaldson--Thomas invariants and functions. In all of these cases the choice of the stability condition did not play a crucial role. However, for a generic quiver this is not the case and the Donaldson--Thomas functions and invariants change as we vary the stability condition. There is a wall and chamber structure on the moduli space of stability conditions and these changes will only happen if we jump over a wall into a different chamber. There is, however, a formula - the wall-crossing formula - relating the  Donaldson--Thomas functions and invariants for various stability conditions. Before we state and prove the formula, let us introduce some fundamental objects in Donaldson--Thomas theory.

Fix two dimension vectors $d,d'$ and recall the following commutative diagram using the notation of section 2
\begin{equation*}  
 \xymatrix @C=2cm @R=1cm{
 \Mst_d \times \Mst_{d'} \ar[d]_{p_d\times p_{d'}}   & \Mst_{d,d'} \ar[l]_{\pi_1\times \pi_3} \ar[r]^{\pi_2}   & \Mst_{d+d'} \ar[d]^{p_{d+d'}}\\ 
 \Msp^{ssimp}_d\times \Msp^{ssimp}_{d'} \ar[rr]^(0.5)\oplus & & \Msp^{ssimp}_{d+d'} }
\end{equation*}
with $\Mst_{d,d'}$ being the stack of short exact sequences  $0\to V^{(1)} \to V^{(2)} \to V^{(3)} \to 0$  such that $\dim V^{(1)}=d, \dim V^{(3)}=d'$. The morphism $\pi_i$ maps a sequence to its i-th entry. By taking the disjoint union over all dimension vectors, we end up with
\[ \xymatrix @C=2.0cm{ \Mst\times \Mst &  \mathfrak{Exact} \ar[r]^{\pi_2} \ar[l]_{\pi_1\times \pi_3} & \Mst, }\]
where $\mathfrak{Exact}$ denotes the stack of all short exact sequences. 
\begin{definition} For a given motivic theory $R$, we call the  $R(\Spec\kk)$-module $R(\Mst)$  with the Ringel--Hall product
 \[ \ast:R(\Mst)\otimes R(\Mst) \xrightarrow{\boxtimes} R(\Mst) \xrightarrow{(\pi_1\times\pi_3)^\ast} R(\mathfrak{Exact}) \xrightarrow{\pi_{2\,!}} R(\Mst) \]
the Ringel--Hall algebra of the quiver $Q$ with respect to $R$.
\end{definition}
\begin{lemma} The Ringel--Hall algebra is an associative algebra with unit.
\end{lemma}
The proof is not very difficult but a nice exercise in dealing with successive extensions. 
\begin{exercise} Let us fix three dimension vectors $d,d',d''$.
 \begin{enumerate}
  \item Consider the following diagram, where the maps are given by mapping (successive) extensions to its subquotients or intermediate extensions and also by the identity on the factors $\Mst$ not being part of an extension. 
\[ \xymatrix { \Mst_d\times \Mst_{d'}\times \Mst_{d''} & \Mst_d\times \Mst_{d',d''} \ar[l] \ar[r] & \Mst_d\times \Mst_{d'+d''} \\
 \Mst_{d,d'}\times \Mst_{d''} \ar[u] \ar[d] & \Mst_{d,d',d''} \ar[l] \ar[r] \ar[d]\ar[u]& \Mst_{d,d'+d''} \ar[d]\ar[u]\\ 
 \Mst_{d+d'}\times \Mst_{d''} & \Mst_{d+d',d''} \ar[r] \ar[l] & \Mst_{d+d'+d''} } \]
 Show that the diagram commutes and that every square is  cartesian.
 \item Use this diagram and the base change property of a motivic theory to prove associativity of the Ringel--Hall product.
 \item The zero representation induces a map $\Spec\kk\xrightarrow{0}\Mst$. Show that $1_0:=0_!(1)\in R(\Mst)$ is a unit for the Ringel--Hall product.
 \end{enumerate}
\end{exercise}
Let us form the following cartesian product
\[ \xymatrix @C=1.5cm { \Mst^{\zeta-ss}_{\mu,\mu'} \ar@{^{(}->}[r]  \ar[d] & \mathfrak{Exact} \ar[r]^{\pi_2} \ar[d]^{\pi_1\times \pi_3} & \Mst \\ 
 \Mst^{\zeta-ss}_\mu\times\Mst^{\zeta-ss}_{\mu'} \ar@{^{(}->}[r] & \Mst\times \Mst. & }
\]
If $\mu > \mu'$, the composition $\Mst^{\zeta-ss}_{\mu,\mu'} \to \Mst$ is an isomorphism onto the image which is the substack of $\Mst$ consisting of all representations whose Harder--Narasimhan filtration has only one subquotient in $\Mst^{\zeta-ss}_\mu$ and another one in $\Mst^{\zeta-ss}_{\mu'}$. Indeed, the functoriality of the Harder--Narasimhan filtration ensures that taking the Harder--Narasimhan filtration provides an inverse morphism to $\Mst^{\zeta-ss}_{\mu,\mu'} \to \Mst$. Fixing another slope $\mu''$ with $\mu'> \mu''$ we continue this way and take the fiber product 
\[ \xymatrix @C=1.5cm { \Mst^{\zeta-ss}_{\mu,\mu',\mu''} \ar@{^{(}->}[r]  \ar[d] & \mathfrak{Exact} \ar[r]^{\pi_2} \ar[d]^{\pi_1\times \pi_3} & \Mst \\ 
 \Mst^{\zeta-ss}_{\mu,\mu'}\times\Mst^{\zeta-ss}_{\mu''} \ar@{^{(}->}[r] & \Mst\times \Mst. & }
\]
which can be identified with the substack of representations having a Harder--Narasimhan filtration with subquotients in $\Mst^{\zeta-ss}_\mu,\Mst^{\zeta-ss}_{\mu'},\Mst^{\zeta-ss}_{\mu''}$. As every quiver representation has a unique Harder--Narasimhan filtration, we obtain a locally finite stratification of $\Mst\!\setminus\!\{0\}$ with strata $\Mst^{\zeta-ss}_{\mu_1,\ldots,\mu_r}\!\!\setminus\!\{0\} \hookrightarrow \Mst\!\setminus\!\{0\}$, where $\Mst^{\zeta-ss}_{\mu_1,\ldots,\mu_r}$ is  defined as above by means of $r-1$ fiber products for every strictly decreasing sequence $\mu_1> \ldots > \mu_r$ in $(-\infty,+\infty]$ of  length $r$. Using the notation
\[ \delta^\zeta_{\mu_1,\ldots,\mu_r}:= (\Mst^{\zeta-ss}_{\mu_1,\ldots,\mu_r}\hookrightarrow \Mst)_!(\unit_{\Mst^{\zeta-ss}_{\mu_1,\ldots,\mu_r}}) \in R(\Mst), \]
this stratification can be written as 
\[ \unit_{\Mst}=1_0 + \sum_{{0<r\in \NN\atop \mu_1> \ldots > \mu_r}} (\delta^\zeta_{\mu_1,\ldots,\mu_r}-1_0).\]
\begin{exercise} Prove the formula $\delta^\zeta_{\mu_1,\ldots,\mu_r}=\delta_{\mu_1}^\zeta\ast \ldots \ast \delta_{\mu_r}^\zeta$ for all sequences $\mu_1>\ldots>\mu_r$ of real numbers. \end{exercise}
Applying the formula proven in the exercise we can rewrite the infinite sum as an infinite product 
\begin{equation} \label{wall_crossing_1} \unit_{\Mst}= \prod_{\mu \searrow}^\ast \delta^\zeta_\mu \end{equation}
which is well-defined as for every dimension vector only finitely many factors contribute. Note that the (infinite) Ringel--Hall product has to be taken in decreasing order of the slopes. If $\zeta'$ is another stability condition, we conclude the formula
\begin{equation} \label{wall_crossing_2} \prod_{\mu \searrow}^\ast \delta^\zeta_\mu=\prod_{\mu \searrow}^\ast \delta^{\zeta'}_\mu \end{equation}
which relates elements in $R(\Mst)$ defined by means of two different stability conditions $\zeta$ and $\zeta'$. In order to obtain a similar formula for Donaldson--Thomas functions, we need to related the Ringel--Hall algebra with corresponding objects on the coarse moduli space which will be the topic of the next subsection.

\subsection{Integration map}

As the Donaldson--Thomas function was an object defined on $\Msp^{\zeta-ss}$, we cannot compare Donaldson--Thomas functions taken with respect to different stability conditions as $\Msp^{\zeta-ss}$ might change. To make them comparable, we need to push them down along $q^\zeta$ to $\Msp^{ssimp}$ which is the ``smallest'' of all moduli spaces. Fix a stacky vanishing cycle $\phi$, and use the fact that $q^\zeta_!$ commutes with $\Sym$, that the open embedding $j:\Mst^{\zeta-ss}_\mu\hookrightarrow \Mst$ is smooth, and the projection formula to conclude
\begin{eqnarray*} \Sym\left(\frac{q^\zeta_! \DTS(Q,W)^\zeta_\mu}{\LL^{1/2}-\LL^{-1/2}}\right) &=& q^\zeta_! \Sym\left(\frac{\DTS(Q,W)^\zeta_\mu}{\LL^{1/2}-\LL^{-1/2}}\right) \\
 &=& (q^\zeta\circ p^\zeta)_!\big(\phi_{\WWW^\zeta}(\ICS_{\Mst^{\zeta-ss}_\mu})\big)\\
 &=& (p\circ j)_!\big(\phi_{\WWW^\zeta}(j^\ast\ICS_{\Mst})\big)\\
 &=& p_! j_!\big(j^\ast\phi_{\WWW}(\ICS_{\Mst})\big)\\
 &=& p_!\big(\delta^\zeta_\mu\cap \phi_{\WWW}(\ICS_{\Mst})\big).
 \end{eqnarray*}
 \begin{definition}
 The $R(\Spec\kk)$-linear map \[I^W:R(\Mst) \ni a\longmapsto p_!\big(a\cap \phi_{\WWW}(\ICS_\Mst)\big)\in R(\Msp^{ssimp})\] is called  integration map.  
 \end{definition}
Hence, we have proven 
\[ I^W(\delta^\zeta_\mu)=\Sym\left(\frac{q^\zeta_! \DTS(Q,W)^\zeta_\mu}{\LL^{1/2}-\LL^{-1/2}}\right).\]
This formula can also been used to define Donaldson--Thomas functions $\underline{\DTS}(Q,W)^\zeta_\mu \in R(\Msp^{ssimp})$ by means of 
\[ I^W(\delta^\zeta_\mu)=\Sym\left(\frac{\underline{\DTS}(Q,W)^\zeta_\mu}{\LL^{1/2}-\LL^{-1/2}}\right)\]
if $\zeta$ is not geometric, i.e.\ $\Mst^{\zeta-ss}$ has no coarse moduli space. If $\zeta$ is geometric, then $\underline{\DTS}(Q,W)^\zeta_\mu=q^\zeta_! \DTS(Q,W)^\zeta_\mu$. 
We wish to apply $I^W$ to equation (\ref{wall_crossing_1}) or (\ref{wall_crossing_2}) to obtain a wall-crossing formula for Donaldson--Thomas functions $\underline{\DTS}(Q,W)^\zeta_\mu$. Unfortunately, $I^W$ will not be an $R(\Spec\kk)$-algebra homomorphism from the Ringel--Hall algebra $(R(\Mst),\ast)$ to $R(\Msp^{ssimp})$ with the convolution product.
\begin{definition} Define the ``quantum'' or deformed convolution product on $R(\Msp^{ssimp})$ by means of 
 \[ (a_d)_{d\in \NN^{Q_0}}\ast (b_{d'})_{d'\in \NN^{Q_0}} := \left(\sum_{d+d'=d''} \LL^{\langle d,d'\rangle/2} a_db_{d'}\right)_{d''\in \NN^{Q_0}}\]
 and similarly on $R(\NN^{Q_0})$. 
\end{definition}
As $\dim:\Msp^{ssimp}\to \NN^{Q_0}$ is a monoid homomorphisms, it will preserve the convolution and, hence, also the deformed convolution product. The main result about the integration map was essentially proven by Reineke in \cite{Reineke_HN} for $W=0$ and by Kontsevich and Soibelman in \cite{KS1} for general potential. A rigorous proof can also be found in \cite{DavisonMeinhardt3}
\begin{theorem} \label{gmebra_homomorphism} The map $I^W:(R(\Mst),\ast) \longrightarrow (R(\Msp^{ssimp}),\ast)$ is a homomorphism of $R(\Spec\kk)$-algebras.
\end{theorem}

\subsection{The wall-crossing identity}

Let us assume the conditions of Theorem \ref{gmebra_homomorphism}. By applying the integration map $I^W$ to the equations (\ref{wall_crossing_1}) and (\ref{wall_crossing_2}), we finally get the wall-crossing identity
\[ I^W(\unit_\Mst)= \prod_{\mu \searrow}^\ast \Sym\left(\frac{\underline{\DTS}(Q,W)^\zeta_\mu}{\LL^{1/2}-\LL^{-1/2}}\right)=\prod_{\mu \searrow}^\ast \Sym\left(\frac{\underline{\DTS}(Q,W)^{\zeta'}_\mu}{\LL^{1/2}-\LL^{-1/2}}\right) \]
relating the Donaldson--Thomas functions $\underline{\DTS}(Q,W)^\zeta$ and $\underline{\DTS}(Q,W)^{\zeta'}$ of two stability conditions $\zeta$ and $\zeta'$. Since $\dim_!$ commutes with the deformed convolution product, we obtain the same formula for the Donaldson--Thomas invariants 
\[ \dim_!I^W(\unit_\Mst)= \prod_{\mu \searrow}^\ast \Sym\left(\frac{\DT(Q,W)^\zeta_\mu}{\LL^{1/2}-\LL^{-1/2}}\right)=\prod_{\mu \searrow}^\ast \Sym\left(\frac{\DT(Q,W)^{\zeta'}_\mu}{\LL^{1/2}-\LL^{-1/2}}\right). \]
Let us illustrate this formula with an example.
\begin{example}[cf.\ Example \ref{Kronecker}] \rm
 Consider the $A_2$-quiver $Q: \bullet_1 \longrightarrow \bullet_2$ with potential $W=0$ and the canonical vanishing cycle of $\underline{\Ka}_0(\Sch)[\LL^{1/2}]^{st}$.  
 \begin{exercise} Show that every representation $V_1\xrightarrow{M}V_2$ of $Q$ is a direct sum of copies of $S_1=(\kk\xrightarrow{0}0), S_2=(0\xrightarrow{0}\kk)$ and $S_{12}=(\kk\xrightarrow{\id}\kk)$. \end{exercise}
 Fix a stability condition $\zeta=(\zeta_1,\zeta_2)$ satisfying $\arg(\zeta_1)<\arg(\zeta_2)$. Given a $\zeta$-semistable representation $V\cong S_1^{m_1}\oplus S_{12}^{m_{12}}\oplus S_2^{m_2}$ with $d_1=m_1+m_{12}$ and $d_2=m_2+m_{12}$, two of the multiplicities $m_1,m_{12},m_2$ must be zero. Moreover, $m_{12}$ must be zero, too, since $S_2\hookrightarrow S_{12}$ destabilizes $S_{12}$ and the latter cannot be (semi)stable. Thus, $V\cong S_1^{d_1}$ or $V\cong S_2^{d_2}$. In particular, the category  of representations of dimension vector $(d_1,0)$ respectively $(0,d_2)$ is isomorphic to the category of representations of the quiver $Q^{(0)}$ with one vertex and no loop. Using $R(\NN^{Q_0})\cong R[[t_1,t_2]]$ we, therefore, obtain
 \[ \dim_!I^W(\unit_\Mst)= A^{(0)}(t_2)\ast A^{(0)}(t_1)=\Sym\left(\frac{t_2}{\LL^{1/2}-\LL^{-1/2}}\right)\ast \Sym\left(\frac{t_1}{\LL^{1/2}-\LL^{-1/2}}\right) .\]
 On the other hand, if we assume $\arg(\zeta_1)>\arg(\zeta_2)$, the representation $S_{12}$ is $\zeta$-stable, and $V\cong S_{12}^{d_1}$ is another class of semistable objects. Thus,
 \begin{eqnarray*} \lefteqn{\dim_!I^W(\unit_\Mst)\;=\; A^{(0)}(t_1)\ast A^{(0)}(t_1t_2)\ast A^{(0)}(t_2) }\\ &=&\Sym\left(\frac{t_1}{\LL^{1/2}-\LL^{-1/2}}\right)\ast\Sym\left(\frac{t_1t_2}{\LL^{1/2}-\LL^{-1/2}}\right)\ast  \Sym\left(\frac{t_2}{\LL^{1/2}-\LL^{-1/2}}\right),\end{eqnarray*}
 and the so-called quantum dilogarithm identity \[ A^{(0)}(t_2)\ast A^{(0)}(t_1)=A^{(0)}(t_1)\ast A^{(0)}(t_1t_2)\ast A^{(0)}(t_2)\] follows.
 Let us  also consider the case $\arg(\zeta_1)=\arg(\zeta_2)$. Then, all representations are semistable, and 
 \[ \dim_!I_W(\unit_\Mst)=\Sym\left(\frac{\sum_{(d_1,d_2)\not=(0,0)} \DT(Q)^\zeta_{(d_1,d_2)}\,t_1^{d_1}t_2^{d_2}}{\LL^{1/2}-\LL^{-1/2}}\right) \]
 allows the computation of the Donaldson--Thomas invariants as the left hand side of the equation is already known by the previous two cases. For example, comparing the coefficients of $t_1t_2$ yields
 \[ \frac{\LL^{1/2}}{(\LL^{1/2}-\LL^{-1/2})^2}=\frac{\DT(Q)^\zeta_{(1,1)}}{\LL^{1/2}-\LL^{-1/2}}, \]
 and $\DT(Q)^\zeta_{(1,1)}=\LL/(\LL-1)$ follows. Note that $\DT(Q)^\zeta_{(1,1)}\in \Ka_0(\Sch_\kk)[\LL^{1/2}]^{st}$ cannot be lifted under the map $\Ka_0(\Sch_\kk)[\LL^{-1/2}] \longrightarrow \Ka_0(\Sch_\kk)[\LL^{1/2}]^{st}$ which does not  contradict Theorem \ref{main_result_1} as $\zeta$ is not generic. In particular, the assumption of being generic cannot be dropped in Theorem \ref{main_result_1} and \ref{main_result_2}.  \\
Let us finally say some words about the Donaldson--Thomas functions $\DTS(Q)^\zeta_d$. For arbitrary stability condition $\zeta$ there is a unique polystable object of given dimension vector $d$ as the previous discussion shows. In particular, every stability condition is geometric with $\Msp^{\zeta-ss}_d=\Msp^{ssimp}_d=\Spec\kk$, and $\DTS(Q)^\zeta_d=\DT(Q)^\zeta_d$ follows for all $d\in \NN^{Q_0}$.  
\end{example}

\bibliographystyle{plain}
\bibliography{Literatur}

\begin{thebibliography}{10}

\bibitem{Behrend}
K.~Behrend.
\newblock {Donaldson--Thomas type invariants via microlocal geometry}.
\newblock {\em Ann. of Math. (2)}, 170, no. 3, 2009.
\newblock math.AG/0507523.

\bibitem{BBS}
K.~Behrend, J.~Byan, and B.~Szendr\H{o}i.
\newblock {Motivic degree zero Donaldson--Thomas invariants}.
\newblock {\em Invent. Math.}, 192, 2013.
\newblock arXiv:0909.5088.

\bibitem{Bittner04}
F.~Bittner.
\newblock The universal euler characteristic for varieties of characteristic
  zero.
\newblock {\em Comp. Math.}, 140:1011--1032, 2004.

\bibitem{DavisonMeinhardt4}
B.~Davison and S.~Meinhardt.
\newblock {Cohomological Donaldson--Thomas theory of a quiver with potential
  and quantum enveloping algebras}.
\newblock in preparation.

\bibitem{DaMe2}
B.~Davison and S.~Meinhardt.
\newblock {The motivic Donaldson-Thomas invariants of (-2) curves}.
\newblock 2012.
\newblock arXiv:1208.2462.

\bibitem{DavisonMeinhardt3}
B.~Davison and S.~Meinhardt.
\newblock {Donaldson--Thomas theory for categories of homological dimension one
  with potential}.
\newblock 2015.
\newblock in preparation.

\bibitem{DaMe1}
B.~Davison and S.~Meinhardt.
\newblock {Motivic DT-invariants for the one loop quiver with potential}.
\newblock {\em Geometry and Topology}, 2015.
\newblock DOI: 10.2140/gt.2015.19.2535.

\bibitem{DenefLoeser2}
J.~Denef and F.~Loeser.
\newblock {Geometry on arc spaces of algebraic varieties}.
\newblock In {\em {European Congress of Mathematics, Vol. I (Barcelona,
  2000)}}, volume 201 of {\em Progr. Math.}, pages 327--348. Birkh\"{a}user.

\bibitem{FeitFine}
W.~Feit and N.J. Fine.
\newblock {Pairs of commuting matrices over a finite field}.
\newblock {\em Duke Math. J.}, 27:91--94, 1960.

\bibitem{GLMH1}
S.M. Guseine-Zade, I.~Luengo, and A.~Melle-Hern\'{a}ndez.
\newblock {A power structure over the Grothendieck ring of varieties}.
\newblock {\em Math. Res. Lett.}, 11, no. 1:49--57, 2004.
\newblock math.AG/0206279.

\bibitem{JoyceI}
D.~Joyce.
\newblock {Configurations in abelian categories. I. Basic properties and moduli
  stacks}.
\newblock {\em Advances in Mathematics}, 203:194--255, 2006.
\newblock math.AG/0312190.

\bibitem{JoyceCF}
D.~Joyce.
\newblock {Constrictable functions on Artin stacks}.
\newblock {\em J. L.M.S.}, 74, 2006.
\newblock math.AG/0403305.

\bibitem{JoyceII}
D.~Joyce.
\newblock {Configurations in abelian categories. II. Ringel--Hall algebras}.
\newblock {\em Advances in Mathematics}, 210:635--706, 2007.
\newblock math.AG/0503029.

\bibitem{JoyceIII}
D.~Joyce.
\newblock {Configurations in abelian categories. III. Stability conditions and
  identities}.
\newblock {\em Advances in Mathematics}, 215:153--219, 2007.
\newblock math.AG/0410267.

\bibitem{JoyceMF}
D.~Joyce.
\newblock {Motivic invariants of Artin stacks and `stack functions'}.
\newblock {\em Quarterly Journal of Mathematics}, 58, 2007.
\newblock math.AG/0509722.

\bibitem{JoyceIV}
D.~Joyce.
\newblock {Configurations in abelian categories. IV. Invariants and changing
  stability conditions}.
\newblock {\em Advances in Mathematics}, 217:125--204, 2008.
\newblock math.AG/0503029.

\bibitem{JoyceDT}
D.~Joyce and Y.~Song.
\newblock {A theory of generalized Donaldson--Thomas invariants}.
\newblock {\em Mem.Amer. Math. Soc.}, 217(1020), 2012.
\newblock math.AG/08105645.

\bibitem{King}
A.~King.
\newblock {Moduli of representations of finite dimensional algebras}.
\newblock {\em Quart. J. Math. Oxford Ser. (2)}, 45 (180):515--530, 1994.

\bibitem{KS1}
M.~Kontsevich and J.~Soibelman.
\newblock {Stability structures, motive Donaldson--Thomas invariants and
  cluster transformations}.
\newblock 2008.
\newblock math.AG/08112435.

\bibitem{KS3}
M.~Kontsevich and Y.~Soibelman.
\newblock {Motivic Donaldson--Thomas invariants: summary of results}.
\newblock In {\em Mirror symmetry and tropical geometry}, volume 527 of {\em
  Contemp. Math.}, pages 55--89. Amer. Math. Soc., Providence, RI, 2010.

\bibitem{KS2}
M.~Kontsevich and Y.~Soibelman.
\newblock {Cohomological Hall algebra, exponential Hodge structures and motivic
  Donaldson--Thomas invariants}.
\newblock {\em Commun. Number Theory Phys.}, 5, 2011.
\newblock arXiv:1006.2706.

\bibitem{Looijenga1}
E.~Looijenga.
\newblock {Motivic measures}.
\newblock {\em Ast\'{e}risque}, 276:267--297, 2002.
\newblock S\'{e}minaire Bourbaki Vol. 1999/2000.

\bibitem{Meinhardt4}
S.~Meinhardt.
\newblock {Donaldson--Thomas invariants versus intersection cohomology for
  categories of homological dimension one}.
\newblock 2015.
\newblock arXiv:1512.03343.

\bibitem{MeinhardtReineke}
S.~Meinhardt and M.~Reineke.
\newblock {Donaldson--Thomas invariants versus intersection cohomology of
  quiver moduli}.
\newblock 2014.
\newblock arXiv:1411.4062.

\bibitem{MMNS}
A.~Morrison, S.~Mozgovoy, K.~Nagao, and B.~Szendr\H{o}i.
\newblock {Motivic Donaldson--Thomas invariants of the conifold and the refined
  topological vertex}.
\newblock {\em Adv. Math.}, 230, 2012.

\bibitem{Reineke_HN}
M.~Reineke.
\newblock The {H}arder-{N}arasimhan system in quantum groups and cohomology of
  quiver moduli.
\newblock {\em Invent. Math.}, 152(2):349--368, 2003.

\bibitem{Reineke5}
M.~Reineke.
\newblock {Moduli of representations of quivers}.
\newblock {\em Proceedings of the ICRA XII conference, Torun}, 2007.
\newblock arXiv:0802.2147.

\bibitem{Reineke2}
M.~Reineke.
\newblock Poisson automorphisms and quiver moduli.
\newblock {\em J. Inst. Math. Jussieu}, 9, no. 3:653--667, 2010.
\newblock arXiv:0804.3214.

\bibitem{Reineke3}
M.~Reineke.
\newblock Cohomology of quiver moduli, functional equations, and integrality of
  {Donaldson-Thomas} type invariants.
\newblock {\em Comp. Math.}, 147, no. 3:943--964, 2011.
\newblock arXiv:0903.0261.

\bibitem{Reineke4}
M.~Reineke.
\newblock {Degenerate Cohomological Hall algebra and quantized Donaldson-Thomas
  invariants for m-loop quivers}.
\newblock {\em Doc. Math.}, 17:1--22, 2012.
\newblock arXiv:1102.3978.

\bibitem{Reineke6}
M.~Reineke and S.~Schr\"oer.
\newblock {Brauer groups for quiver moduli}.
\newblock 2014.
\newblock arxiv:1410.0466.

\bibitem{Szendroi}
B.~Szendr\H{o}i.
\newblock {Cohomological Donaldson-Thomas theory}.
\newblock {\em {Proceedings of String-Math 2014}}, 2015.
\newblock arXiv:1503.07349.

\bibitem{Thomas1}
R.P. Thomas.
\newblock A holomorphic casson invariant for {Calabi--Yau} 3-folds, and bundles
  on {K3} fibrations.
\newblock {\em J. Diff. Geom.}, 54:367--438, 2000.
\newblock math.AG/9806111.

\end{thebibliography}

\vfill

\textsc{\small S. Meinhardt: Fachbereich C, Bergische Universit\"at Wuppertal, Gau{\ss}stra{\ss}e 20, 42119 Wuppertal, Germany}\\
\textit{\small E-mail address:} \texttt{\small meinhardt@uni-wuppertal.de}\\

\end{document}